\documentclass[reqno,10pt]{amsart}
\usepackage{euscript}
\usepackage{amssymb}
\usepackage{amscd}
\usepackage{amsmath}
\usepackage{epic}
\usepackage{graphics}
\usepackage{epsfig}
\usepackage{color}
\usepackage{hyperref}
\usepackage{setspace}

\numberwithin{equation}{section}
\newtheoremstyle{my}{1.5em}{0.5em}{\em}{}{\sc}{.}{0.5em}{}
\newtheoremstyle{your}{1.5em}{0.5em}{}{}{\sc}{.}{0.5em}{}

\theoremstyle{my}

\theoremstyle{my}
\newtheorem{thm}{Theorem}[section]
\newtheorem{Theorem}[thm]{Theorem}
\newtheorem*{Theorem*}{Theorem}
\newtheorem{Corollary}[thm]{Corollary}

\newtheorem*{corollary*}{Corollary}
\newtheorem{Lemma}[thm]{Lemma}

\newtheorem{Proposition}[thm]{Proposition}

\newtheorem*{conjecture*}{Conjecture}

\newtheorem*{question*}{Question}

\newtheorem*{definitions*}{Definitions}

\newtheorem*{rem*}{Remark}

\newtheorem*{remark*}{Remark}

\newtheorem*{remarks*}{Remarks}
\newtheorem*{example*}{Example}
\newtheorem*{examples*}{Examples}

\newtheorem*{convention*}{Convention}
\newtheorem*{conventions*}{Conventions}
\newtheorem*{Note*}{Note}
\newtheorem*{exercise*}{Exercise}
\newtheorem*{bibliographical-note*}{Bibliographical note}

\theoremstyle{your}
\newtheorem{Remark}[thm]{Remark}
\newtheorem{Definition}[thm]{Definition}
\newtheorem{Example}[thm]{Example}

\newcommand{\Acknowledgements}{{\em Acknowledgements.} }

\newcommand{\R}{\mathbb{R}}
\newcommand{\Z}{\mathbb{Z}}
\newcommand{\Q}{\mathbb{Q}}
\newcommand{\C}{\mathbb{C}}
\newcommand{\ev}{\operatorname{ev}}
\newcommand{\pr}{\operatorname{pr}}
\newcommand{\pa}{\partial}
\newcommand{\Ordo}{\mathcal{O}}

\newcommand{\bR}{\mathbb{R}}
\newcommand{\bZ}{\mathbb{Z}}

\newcommand{\bC}{\mathbb{C}}

\newcommand{\bP}{\mathbb{P}}

\newcommand{\id}{\mathrm{id}}
\newcommand{\ind}{\mathrm{ind}}

\renewcommand{\ker}{\mathrm{ker}}
\newcommand{\coker}{\mathrm{coker}}

\newcommand{\Hom}{\mathrm{Hom}}

\newcommand{\Diff}{\textrm{Diff}}
\renewcommand{\index}{\textrm{index}}

\newcommand{\sblv}{\mathcal{H}}
\newcommand{\cfig}{\mathcal{X}}

\newcommand{\tcfig}{\mathcal{Y}}

\newcommand{\scrB}{\mathcal{B}}
\newcommand{\scrJ}{\mathcal{J}}

\newcommand{\MM}{\mathcal{M}}
\newcommand{\FF}{\mathcal{F}}
\newcommand{\BB}{\mathcal{B}}
\newcommand{\barFF}{\overline{\mathcal{F}}}

\newcommand{\JJ}{\mathcal{J}}

\newcommand{\cc}{\boldsymbol{c}}
\newcommand{\EE}{\mathcal{E}}
\newcommand{\XX}{\mathcal{X}}
\newcommand{\YY}{\mathcal{Y}}
\newcommand{\NN}{\mathcal{N}}

\newcommand{\TT}{\mathcal{T}}
\renewcommand{\SS}{\mathcal{S}}
\newcommand{\DD}{\mathcal{D}}
\newcommand{\LL}{\mathcal{L}}

\newcommand{\zz}{\boldsymbol{\zeta}}

\newcommand{\diam}{\operatorname{diam}}

\newcommand{\Exp}{\operatorname{Exp}}
\newcommand{\Pre}{\operatorname{PG}}

\newcommand{\Fl}{{\,\mathrm{F}}}

\title{Exact Lagrangian immersions with a single double point}
\author{Tobias Ekholm and Ivan Smith}
\date{November 2011 (v1); revised June 2014 (v2), September 2014 (v3)}

\begin{document}
\thispagestyle{empty}

\begin{abstract}
We show that if a closed orientable $2k$-manifold $K$, $k>2$, with Euler characteristic $\chi(K)\ne -2$ admits an exact Lagrangian immersion into $\C^{2k}$ with one transverse double point and no other self intersections,  then $K$ is diffeomorphic to the sphere. The proof combines Floer homological arguments with a detailed study of moduli spaces of holomorphic disks with boundary in a monotone Lagrangian submanifold obtained by Lagrange surgery on $K$.  
\end{abstract}
\maketitle



\section{Introduction}
\label{Sec:Intro}

\subsection{The main result} \label{Sec:Formulation}
Consider $\C^{n}$ with coordinates $z=(x_1+iy_1,\dots,x_n+iy_n)$, Liouville form $\theta=\sum_{j=1}^{n}y_jdx_j$, and symplectic form $\omega=-d\theta$. An immersion $f\colon K\to\C^{n}$ of an $n$-manifold $K$ is \emph{Lagrangian} if $f^{\ast}\omega=0$, and \emph{exact Lagrangian} if the closed form $f^{\ast}\theta$ is also exact, i.e.~$f^{\ast}\theta=dz$ for some function $z\colon K\to\R$.   More explicitly still,  a smooth immersion $f: K \rightarrow \C^n$ is Lagrangian if the orthogonal complement $df(T_xK)^{\perp}$, taken with respect to the standard Euclidean metric, is the subspace $i(df(T_xK))$ given by multiplication by $i = \sqrt{-1}$, for every $x\in K$. Any Lagrangian immersion of a simply-connected manifold is exact.   Our main result gives a symplectic topological characterisation of the standard sphere in even dimensions  $>4$.

\begin{Theorem}\label{Thm:Main}
Let $K$ be a closed orientable $2k$-manifold, $k>2$, with Euler characteristic $\chi(K)\ne -2$. If $K$ admits an exact Lagrangian immersion $f\colon K\to\C^{2k}$ with one transverse double point and no other self intersections, then $K$ is diffeomorphic to the standard sphere $S^{2k}$.
\end{Theorem}
\begin{Example}
The \emph{Whitney sphere} is the exact Lagrangian immersion $w\colon S^n \to\C^n$ where
\begin{equation} \label{Eqn:Whitney}
S^n = \left\{ (x,y) \in \bR^n \times \bR \ \big | \,  \ |x|^2 + y^2 = 1\right\} \quad\text{and}\quad w(x,y)=(1+iy)x \in \C^n.
\end{equation}
It has exactly one transverse double point, $w(0,1)=w(0,-1)$.
\end{Example}

It is relatively straightforward to show that $K$ is a homotopy sphere. (The numbers of smooth homotopy spheres in the lowest relevant dimensions are: 1 in dimension 6, 2 in dimension 8, 6 in dimension 10, 1 in dimension 12, and 2 in dimension 14. For general dimension $2k$, the number is finite but can be arbitrarily large \cite{KM}.) 
The heart of the proof involves showing that $K$ bounds a parallelizable manifold.  The construction of such a manifold is broadly inspired by Donaldson's proof of the diagonalizability theorem for the intersection forms of definite four-manifolds \cite{Donaldson}, in the sense that we build a cobordism with boundary $K \# (S^1 \times S^{2k-1})$ from moduli spaces of solutions to a partial differential equation (a perturbed Cauchy-Riemann equation), where the boundary component $K \# (S^1 \times S^{2k-1})$ appears as a space of constant solutions.

If $f: K \rightarrow \C^{2k}$ is a smooth immersion of an oriented $2k$-manifold with normal bundle $\nu$, then the signed count of double points of $f$ is minus one half of the Euler number of $\nu$.    For Lagrangian immersions the normal and tangent bundles are isomorphic, so in the situation of Theorem \ref{Thm:Main},  necessarily $\chi(K) = \pm 2$. 

Note that  \emph{inexact} Lagrangian immersions with a single double point are easy to come by (take any Lagrangian immersion and remove all but one of the double points by surgery).

\subsection{Context} \label{Sec:Context} A basic question in symplectic topology is to understand the borderline between flexibility and rigidity phenomena (Gromov's famous ``soft-hard" dichotomy).  Consider for definiteness the following model question: given a closed orientable $n$-manifold $K$, does $K$ admit an exact Lagrangian immersion  into $\C^n$?

Gromov proved that no closed manifold $K^n$ admits an exact Lagrangian embedding in $\bC^n$ (rigidity), whilst every $K$ satisfying the homotopy-theoretic constraint required for a totally real immersion, namely that $TK \otimes \C$ is a trivial bundle, admits some exact Lagrangian immersion (flexibility).   It was natural in light of this result to expect that exact Lagrangian immersions with a \emph{fixed} number of double points display interesting rigidity phenomena, with Gromov's non-existence theorem an extreme example corresponding to the case of no double points at all.  A version of Arnol'd's chord conjecture asserts that the number of isolated double points of an exact Lagrangian immersion should be at least $(1/2)\,\sum_j b_j(K)$, half the sum of the Betti numbers of $K$.  This was established for Lagrangian immersions  $K$ under the technical hypothesis that the associated Legendrian contact homology algebra admits an augmentation \cite{Ekholm, EESori}, which however can be a rather difficult condition to check in practice.  Note that any exotic sphere admits Morse functions with exactly two critical points \cite{SmalePoincare}, so Theorem \ref{Thm:Main} detects a phenomenon that goes beyond Morse theory, and presumably beyond its Floer homological counterparts.

Recent developments (postdating the first version of this paper) reveal that the augmentation hypothesis is not  technical, but that the borderline between rigidity and flexibility is harder to locate than had been anticipated, with rigidity phenomena being more scarce.   Isolated double points of Lagrangian immersions have a mod 2 grading \cite{EES1}, which is the sign of the self-intersection if $K$ is oriented and even-dimensional.  In the presence of self-intersections of  ``negative" sign,  the papers \cite{EliashbergMurphy,YETI} prove that exact Lagrangian immersions may be flexible even if one contrains the number of double points.  In particular, \cite{YETI}  shows that any orientable $2k$-manifold $K$ with $TK\otimes \C$ trivial and Euler characteristic $\chi(K) = -2$ admits an exact Lagrangian immersion into $\C^{2k}$ with exactly one double point.  Furthermore, any odd-dimensional homotopy sphere admits an exact Lagrangian immersion into Euclidean space with one double point, whilst any even-dimensional homotopy sphere admits an exact  immersion with precisely three double points. Thus, Theorem \ref{Thm:Main} is in a sense sharp.




\subsection{About the proof of Theorem \ref{Thm:Main}}

 The choice of primitive $f^{\ast}\theta = dz$ defines an embedding (well-defined up to translation)  $f\times z\colon K \rightarrow \bC^n \times \bR$ which is Legendrian with respect to the contact form $dz-\theta$.  The assumptions one double point and $\chi(K)\ne -2$ ensure that the Legendrian homology of this lift  can be linearized, which in turn implies that $f$ satisfies a version of the Morse inequalities for double points of exact Lagrangian immersions originally conjectured by Arnol'd, see \cite{Ekholm, EESori}. 
The Morse inequalities, in combination with results of Damian \cite{Damian}, show that $K$ is a homotopy sphere, $K\approx\Sigma$. As mentioned previously, the substance of Theorem \ref{Thm:Main} is the construction  of a parallelizable manifold  with boundary $\Sigma$, which in even dimensions is sufficient to exclude all the exotic spheres \cite{KM}.  Theorem \ref{Thm:Main} accordingly gives non-trivial information in any dimension where there are exotic $2k$-spheres. 

The bounding manifold is constructed in two steps.  First,  an argument going back to Gromov \cite{Gromov} and Oh \cite{Oh} yields a non-compact cobordism with boundary a given Lagrangian submanifold of $\C^n$; this cobordism is built  from solutions to perturbed Cauchy-Riemann problems. Second, via techniques introduced by Fukaya, Oh, Ohta and Ono \cite{FO3},  such a cobordism can be compactified and capped off using fibered products of moduli spaces and abstract chains bounding such spaces. Such a strategy was spectacularly implemented in a similar context by Abouzaid \cite{Abouzaid}, following suggestions of Seidel. 

A more precise description of the argument is as follows. Given $f\colon K\to \C^{2k}$ as in Theorem \ref{Thm:Main},  we apply Lagrange surgery to create a monotone embedded smooth Lagrangian submanifold $L \subset \bC^{2k}$ of minimal Maslov number $2k$. By \cite{Damian}, $L$ then fibers over $S^{1}$ with fiber a homotopy $(2k-1)$-sphere. We study a $1$-parameter family of Floer equations (perturbed Cauchy-Riemann equations) with Lagrangian boundary condition $L$. The corresponding moduli space $\FF(0\beta)$ of Floer holomorphic disks in the trivial relative homotopy class
is used to explicitly construct a bounding manifold $\BB$ with $\pa \BB=L$ and with stably trivial tangent bundle. This construction is somewhat involved (the subsequent notation is that which appears in the body of the paper): $\FF(0\beta)$ is non-compact because of bubbling and has a compactification with boundary consisting of broken curves. In the case we study there is only one possible bubbling configuration, and the boundary $\NN$ is a fibered product of two other moduli spaces, $\NN=\FF^*(-\beta) \times_L \MM^*(\beta)$. Employing gluing analysis, we find a neighborhood of the boundary in $\FF(0\beta)$, the complement of which is a compact $C^{1}$-smooth manifold $\FF_{\rho_0}(0\beta)$ with $\pa\FF_{\rho_0}(0\beta)=L\cup \NN$, as well as an explicit collar neighborhood of $\NN\subset\FF_{\rho_0}(0\beta)$. Using the fibered product with one factor a manifold $\DD$ filling $\FF^{\ast}(-\beta)$ and the other $\MM^{\ast}(\beta)$,  we are able to fill the boundary $ \NN \subset \pa \FF_{\rho_0}(0\beta)$ and thereby create the cobordism $\BB=\FF_{\rho_0}(0\beta)\cup (\DD\times_{L}\MM^{\ast}(\beta))$. 

We next analyze the stable tangent bundle of $\BB$. The stable tangent bundles of $\FF(0\beta)$, $\FF^{\ast}(-\beta)$, and $\MM^{\ast}(\beta)$ are all restrictions of index bundles over the product of the space of smooth maps $(D,\pa D)\to(\C^{2k},L)$, where $D$ is the $2$-disk, and a half-line.  Here the Fredholm problem at a point $(u,r)$ is a Cauchy-Riemann operator with Lagrangian boundary condition given by the Lagrangian tangent planes of $L$ along $u|_{\pa D}$.    We view $L$ as a piecewise linear ({\sc pl}) embedding, which allows us to study these index bundles very explicitly. First, $L$ is {\sc pl}-homeomorphic to the manifold $W'$ obtained by Lagrange surgery on the Whitney sphere and, being embedded in double dimension, there is an ambient {\sc pl} isotopy taking $L$ to $W'$. Second, the isotopy can be covered by a homotopy of (stable) Lagrangian Gauss maps into the Lagrangian Grassmannian. Consequently, if an index bundle associated to $W'$ is stably trivial,  then so is the corresponding bundle associated to $L$. 

Restriction to the boundary gives a homotopy equivalence from the space of maps $(D,\pa D)\to (\C^{2k},W')$ to the free loop space of $W'\approx S^{1}\times S^{2k-1}$. Using the energy functional of the standard metric on $S^{2k-1}$, we get a $(6k-7)$-skeleton for the free loop space, over which we give an explicit trivialization of the index bundles. This shows that the tangent bundles of $\FF(0\beta)$, $\FF^{\ast}(-\beta)$, and $\MM^{\ast}(\beta)$ are all stably trivial. In order to relate trivializations of these three bundles near the boundary of $\FF(0\beta)$ we introduce ``coherent trivializations'', generalizing the more familiar idea of coherent orientations, and prove existence in the case under study. This allows us to reduce the question of extension of the stable trivialization over $\FF_{\rho_0}(0\beta)$ to the cap $\DD\times_{L}\MM^{\ast}(\beta)$ to a problem about spin structures. We prove the spin problem is not obstructed, and conclude that $T\BB$ is stably trivial. Finally, adding a $2$-handle to $\BB$, we produce a parallelizable filling for the homotopy sphere $K$.

The filling argument just outlined has much in common with the argument used by Abouzaid in \cite{Abouzaid}, where he started from an embedding $T^*S^{4k+1} \to \C\bP^{2k} \times \C^{2k+1}$, and a Hamiltonian isotopy displacing the image. The perturbed family of Cauchy-Riemann problems underlying Theorem \ref{Thm:Main} is analogously obtained from a Hamiltonian isotopy which displaces $L$ from itself. In Abouzaid's case, there were two bubble configurations, leading to additional complications.  His filling was accordingly built from a smooth structure constructed on a CW-approximation of the actual moduli space of Floer disks. Our argument is more direct, and gives a  $C^{1}$-structure on the compactified moduli space itself.  

\subsection{Lagrangian embeddings}
Because most of the argument for Theorem \ref{Thm:Main} is carried out on the embedded Lagrange surgery, the method of proof also gives results for Lagrangian embeddings, provided a homotopy obstruction arising from the stable Lagrangian Gauss map vanishes. If the dimension is divisible by 8 then the homotopy group in which the obstruction lies is zero. 

\begin{Corollary} \label{Cor:First}
If $k\ge 1$, then a monotone Lagrangian submanifold $L\subset \bC^{8k}$ of minimal Maslov number $8k$ bounds a parallelizable manifold.
\end{Corollary}

Any such manifold is known \cite[Theorem 1.7]{Damian} to be a fiber bundle over $S^1$ with fiber a homotopy $(8k-1)$-sphere $Q$. For general homotopy spheres $Q$, the mapping class group $\pi_0 \Diff(Q)$ is unknown. However, if the fiber is the standard sphere, $Q \approx S^{8k-1}$, then the mapping torus (being orientable) is diffeomorphic to a connected sum $(S^1 \times S^{8k-1}) \# \Sigma^{8k}$ for some homotopy $8k$-sphere $\Sigma$, cf.~Remark \ref{Rem:Cerf}, and such a manifold framed bounds only if $\Sigma$ is diffeomorphic to $S^{8k}$.  Thus, Corollary \ref{Cor:First} implies that such a monotone $L$ is either the obvious product of standard spheres, or an exotic sphere bundle.

To state a further result, let $P=S^{1}\times S^{8k-1}$, $k\geq 1$, and recall \cite[Theorem 1(c)]{SchultzA}  that $P\#\Sigma$ is diffeomorphic to $P$ if and only if $\Sigma$ is diffeomorphic to $S^{8k}$.  (The connect sums $P \# \Sigma$ do not exhaust smooth structures on the topological  manifold $P$; the general such is realised by $(S^1 \times \Sigma^{8k-1}) \# \Sigma^{8k}$, with $\Sigma^j$ a homotopy sphere of  dimension $j$, \cite[Theorem A]{Schultz}.)

\begin{Corollary} \label{Cor:Main}
If $k\geq 1$ and $\Sigma$ is a homotopy $8k$-sphere, then the cotangent bundles $T^{\ast}P$ and $T^{\ast}(P\#\Sigma)$ are symplectomorphic if and only if $\Sigma \approx S^{8k}$ is the sphere. \end{Corollary}

We remark that $T^*(P\# \Sigma)$ and $T^*P$ \emph{are} diffeomorphic, cf.~Remark \ref{Rem:Diffeomorphic}.  Corollary \ref{Cor:Main} is analogous to Abouzaid's Theorem \cite{Abouzaid} which says that if $\Sigma$ is a homotopy $(4k+1)$-sphere which does not bound a parallelizable manifold then $T^*\Sigma$ is not symplectomorphic to $T^*S^{4k+1}$. It contributes to recent progress on Arnold's ``nearby Lagrangian submanifold" conjecture \cite{Arnold:firststeps}, which asserts in particular that $T^*M$ and $T^*N$ are symplectomorphic only when $M$ and $N$ are diffeomorphic.  

\subsection{Gluing results and the odd-dimensional case}

The second half of the paper contains the gluing analysis needed to obtain $C^1$-smooth structures (and hence tangent bundles) on the compactified moduli space and on $\BB$. Constructions of such $C^1$-structures are complicated by the fact that holomorphic disks bubbling off have automorphisms. We overcome these complications by using intersections with ambient hypersurfaces to stabilize the domains of the bubbles. In order to control these intersection points as the disks vary in the moduli space, we are led to study moduli spaces with jet conditions. To facilitate that study we make the Lagrangian $L$ real analytic and take the almost complex structure standard in a neighborhood of $L$. This provides standard local solutions at boundary points with arbitrary $m$-jet for any $m\ge 0$, which allow us to deal with jet-conditions using essentially finite dimensional techniques. 
Because we work in $\C^n$, and the holomorphic bubble disks all realize primitive relative homotopy classes, our analytical arguments are rather explicit. In particular, geometric arguments suffice for transversality, there is no need for abstract perturbations, and there are no orbifold complications.

The gluing analysis applies to a Lagrangian embedding $L \to \C^n$ of minimal Maslov number $n$ and with $\pi_1(L)=\bZ$, with no assumption on the parity of $n$. If $n$ is odd, however, the topological side of the story is complicated at various places. The original Lagrange immersion $f\colon K \to \C^n$ may have a unique double point of Legendrian homology grading (Maslov index) $1$,  rather than $n$, in which case the Legendrian homology of the resulting  Lagrange surgery is not  linearizable. Even when the Lagrangian $L$ obtained by surgery has Maslov index $n$, the bounding manifold obtained from spaces of holomorphic disks cannot be parallelizable since $L$ is not orientable; the boundary of the orientation double cover of the filling could contain two canceling copies of an exotic sphere with opposite orientations.

\subsection{Recent developments.}  Since the first version of this paper was submitted, a number of relevant developments have taken place, some of which were mentioned in Section \ref{Sec:Context}.

\begin{enumerate}
\item  In \cite{EkholmSmith-Sequel}, we revisit the constructions of this paper, but studying moduli spaces of holomorphic disks with boundary directly on the image of the Lagrangian immersion, rather than on its monotone surgery.  Aside from bringing some technical simplifications, this circumvents the non-orientability issue mentioned in the previous paragraph, and leads to a result valid in all dimensions:  \emph{if a closed oriented $n$-manifold $K$ admits a Lagrangian immersion into $\C^n$ with a unique  double point, which is moreover ``positive"   (i.e. of even Maslov grading), then  $K$ is a homotopy sphere which bounds a parallelizable manifold}.  In contrast, Corollary \ref{Cor:Main} relies essentially on the appearance of the Lagrange surgery, and does not seem amenable to the strategy of \cite{EkholmSmith-Sequel}.

\item Murphy's theory of loose Legendrian embeddings, and the resulting Eliashberg-Murphy theory of Lagrangian caps \cite{EliashbergMurphy}, led to flexibility results for Lagrangian immersions \cite{YETI} which in particular demonstrate the necessity of the Euler characteristic hypothesis in Theorem \ref{Thm:Main}, which we had previously presumed was technical. 
\end{enumerate}

The Appendix to  \cite{EkholmSmith-Sequel} gives a detailed and entirely self-contained proof of the existence of the $C^1$-smooth structure on the compactified moduli space which arises in that argument (that, however,  is intrinsically simpler than the moduli space occuring here, having boundary a product and not a fibred product of other moduli spaces). Some readers may find that a helpful precursor to the later sections of this paper.  

\subsection{Organization of the paper}
Section \ref{Sec:Top+FH} recalls the Floer homological arguments which imply, in the situation of Theorem \ref{Thm:Main}, that $K \approx \Sigma$ is a homotopy sphere. 
Section \ref{Sec:bounding} constructs a parallelizable bounding manifold for $\Sigma$, subject to establishing a suitable  analytical framework for the relevant moduli spaces, and finishes with the proofs of Theorem \ref{Thm:Main} and its corollaries.  The underlying analysis is deferred to Sections \ref{Sec:basicsetup} - \ref{Sec:gluing2}.  Whilst much of this is encompassed by standard pseudo-holomorphic curve theory, we include enough background to give an essentially complete proof of the notable exception, Theorem \ref{Thm:gluing}, which constructs a $C^1$-structure on a certain compactified moduli space of Floer holomorphic disks. Finally, Section \ref{Sec:IndexBundle} discusses index bundles over the loop space of $S^{1}\times S^{2k-1}$, and constructs coherent trivializations. 

\smallskip

\Acknowledgements T.E.~is partially supported by the Knut and Alice Wallenberg Foundation, as a Wallenberg Scholar. I.S.~was partially supported by European Research Council grant ERC-2007-StG-205349. The authors are indebted to Mohammed Abouzaid and Paul Seidel for helpful conversations.

\section{Topological restrictions from Floer homology}
\label{Sec:Top+FH}
In this section we discuss topological constraints on exact Lagrangian immersions with a single double point that can be derived from Floer homological arguments. 

\subsection{Legendrian homology}
Let $f\colon K\to \C^{n}$ be an exact Lagrangian immersion of an orientable $2k$-manifold with a single transverse double point.  As explained in Section \ref{Sec:Intro} there is a Legendrian lift $\tilde f=f\times z\colon K\to\C^{n}\times\R$. In this case $\tilde f$ is necessarily (and not just generically) an embedding, see Remark \ref{Rem:EmbeddingForFree}. 

Since $K$ is orientable the Maslov class of $K$ is even and we can define a $\Z_2$-grading on the Reeb chords of $\tilde f$, see \cite{EES1}. If $a$ is a Reeb chord then we write $|a|_2\in\Z_2=\{0,1\}$ for this grading. Furthermore, in case the Maslov class of $K$ vanishes then we get a well defined $\Z$-grading $|a|$ of Reeb chords $a$.

Assume now that $n=2k$. Then we have the following relation between Reeb chord gradings and Euler characteristic.

\begin{Lemma}
Let $g\colon M\to\C^{2k}$ be a self transverse exact Lagrangian immersion of an orientable closed $2k$-manifold such that the Legendrian lift $\tilde g$ of $g$ is an embedding. Let $\mathcal{Q}$ denote the set of Reeb chords of $\tilde g$. Then the Euler characteristic $\chi(M)$ of $M$ satisfies
\[
\chi(M)=2\sum_{c\in\mathcal{Q}}(-1)^{|c|_2}.
\]
In particular, if $f\colon K\to\C^{2k}$ is an exact Lagrangian immersion with exactly one double point then $\chi(K)=\pm 2$, where the sign is positive if the unique Reeb chord $a$ of $\tilde f$ satisfies $|a|_2=0$ and negative if $|a|_2=1$.   
\end{Lemma}

\begin{proof}
The Euler characteristic formula is straightforward, see \cite[Equation (3.2) and Proposition 3.2 (2)]{EES2}, and the last statement is an immediate consequence of it.
\end{proof}

Consider now a self transverse exact Lagrangian immersion $f\colon K\to\C^{2k}$ with a single double point and corresponding unique Reeb chord $a$ of the Legendrian lift $\tilde f$.  The Legendrian homology algebra is a differential graded algebra (DGA)  generated by Reeb chords of $\tilde{f}$, with the empty chord corresponding to the unit element $1$.   If $|a|_2=1$ then, in principle, one could have the relation $\pa a =1$ in the Legendrian DGA. This would then not be linearizable, i.e. it would admit no degree zero chain map to the DGA given by the ground field (in degree zero) with trivial differential.  In the non-linearizable setting it seems hard to use Legendrian homology to draw conclusions about the topology of $K$ (this situation corresponds to the Lagrangian Floer homology of the immersed Lagrangian $f(K)$ being \emph{obstructed}). If on the other hand $|a|_2=0$ then we have the following.  
 
\begin{Lemma}\label{Lem:Maslov=0}
Let $f\colon K\to\C^{2k}$ be an exact Lagrangian immersion such that the unique Reeb chord $a$ of the Legendrian lift $\tilde f$ satisfies $|a|_2=0$. Then $K$ is a $\Z$-homology sphere. Consequently, its Maslov class vanishes, and moreover $|a|=2k$. 
\end{Lemma}

\begin{proof}
Consider first $\Z_{2}$-coefficients; $|a|_2=0$ implies that the Legendrian DGA of $\tilde f$ is linearizable. It then follows from the $\Z_2$-graded Morse inequalities for linearized Legendrian homology over $\Z_{2}$, see \cite[Theorem 1.2]{EESa} that $\dim(H_\ast(K;\Z_2))\le 2$. On the other hand $\dim(H_0(K;\Z_2))=\dim(H_{2k}(K;\Z_2))=1$ and we conclude that $H_{j}(K;\Z_2)=0$ for $j\ne 0,2k$. In particular, $H^{1}(K;\Z_2)=0=H^{2}(K;\Z_2)$, so $K$ is spin and we can define the Legendrian homology DGA with arbitrary field coefficients. 

Repeating the above argument with $\Q$-coefficients shows that $H^{1}(K;\Q)=0$. Hence the Maslov class of $f$ vanishes and there is a $\Z$-grading on the linearized Legendrian homology. Repeating the argument again with $\Z_p$ coefficients, $p$ prime shows that $H_j(K;\Z)=0$, $j\ne 0,2k$, and then \cite[Theorem 5.5]{EESa} implies that $|a|=2k$.
\end{proof}

\begin{Remark} \label{Rem:EmbeddingForFree}
Similar general properties of Legendrian contact homology justify our earlier claim that $\tilde{f}$ is necessarily an embedding.   If $\tilde{f}$ was an immersion, there could be no holomorphic disk with boundary on $K$ (its area would be given by a positive multiple of the length of the Reeb chord of the Legendrian lift, which vanishes if the lift still has a double point). In that situation, the Legendrian homology of $K$, or of two parallel copies $K\sqcup K$ of $K$, can necessarily be linearized. The linearized Legendrian homology of $K \sqcup K$ would equal  $H_\ast(K)$ by \cite{EESa}, but this is impossible: $K$ is displaceable so the linearized homology must vanish.
\end{Remark}

\subsection{Lagrange surgery}
Let $f\colon K \to \C^{2k}$ be a Lagrangian immersion with a unique transverse double point and Legendrian lift $\tilde f=f\times z$, and suppose the conditions of Lemma \ref{Lem:Maslov=0} hold. We write $\{p_{+},p_{-}\}\subset K$ for the preimage of the double point, $f(p_+)=f(p_-)$, and choose notation so that $z(p_+)>z(p_-)$. There is a surgery procedure which resolves the double point \cite{Polterovich}. It gives a Lagrangian embedding of $K$ with a $1$-handle attached, $K\# (S^{2k-1}\times S^{1})$. In fact, there are two ways of adding the handle, which lead to Lagrangian non-isotopic submanifolds. 

To simplify our discussion of this surgery procedure we first note that after Hamiltonian isotopy we may assume that the double point is located at $0\in\C^{2k}$ and that near $0$ the two sheets of $f(K)$ agree with
\begin{equation}\label{eq:rightangles}
\Gamma \ = \ \bR^{2k} \cup i\bR^{2k} \ \subset \ \bC^{2k}.
\end{equation}

Since the Maslov class of $f$ vanishes we can define a \emph{phase function} $\phi\colon K\to \R$ which measures local contributions to the Maslov index and is unique up to additive constant, see \cite{Seidel:graded} or \cite[p.6]{BEE}. In terms of this function we have
\begin{equation}\label{Eq:phasegrading}
|a|=\phi(p_-)-\phi(p_+)+k-1.
\end{equation}  

Lagrange surgery is defined locally. Consider $\Gamma$ as above and let $x+iy$ be the standard coordinate on $\C$. Write
\[
Q_{\pm}^{\ast}=\{x+iy\colon x\ge 0,\; \pm y\ge 0,\; x+iy\ne 0\}
\]
Fix smooth embedded paths $\gamma_\pm\colon\R\to Q_{\pm}^{\ast}$ which satisfy 
\[
\gamma_\pm(t) =
\begin{cases}
 - t &\text{for }t < -\epsilon,\\
 \pm it &\text{for }t > \epsilon.
\end{cases}
\]
Thinking of $S^{2k-1}$ as the unit sphere in $\R^{2k}$ we define the two \emph{Lagrange handles} $H_{\pm}$ as follows:  	  
\begin{equation}\label{eq:Lhandle}
H_\pm = \bigcup_{t\in\R} \gamma_\pm(t)\cdot S^{2k-1} \subset \C^{2k},
\end{equation} 
where $\zeta\cdot$ denotes multiplication by the complex scalar $\zeta$. Then $H_{\pm}$ are embedded Lagrangian submanifolds diffeomorphic to $S^{2k-1} \times \bR$ and co-inciding with $\Gamma$ outside of a ball of radius $2\epsilon$ around the origin. As above there is a phase function $\phi\colon H_{\pm}\to\R$ which is unique up to additive constant. 

\begin{Lemma}\label{Lem:LhandleMaslov}
Let $\phi\colon H_{\pm}\to\R$ be the phase function which equals $0$ on $H_{\pm}\cap\R^{2k}$. Then
\begin{enumerate}
\item on $H_+$, $\phi(\eta)=k-1$ for $\eta\in i\R^{n}\cap H_+$, and
\item on $H_-$, $\phi(\eta)=-(k-1)$ for $\eta\in i\R^{n}\cap H_-$.
\end{enumerate}
\end{Lemma}

\begin{proof}
Consider the tangent planes along the path $\gamma_{\pm}(t)\cdot v$, for $v\in S^{2k-1}$. On the subspace of the tangent space perpendicular to $v$ this is just multiplication by $\gamma_{\pm}(t)$ which contributes $\pm\frac{2k-1}{2}$ to the phase function on $H_{\pm}$. The tangent vector $\dot\gamma_{\pm}(t)\cdot v$ makes a quarter rotation of sign $\mp$ on $H_{\pm}$, which contributes $\mp\frac{1}{2}$ to the phase function. 
\end{proof}

We next discuss Lagrange surgery on $f\colon K\to\C^{2k}$. Choose coordinates so that $p_-$ lies in the $\R^{2k}$-sheet of $f$ and $p_-$ in the $i\R^{2k}$-sheet at the double point $0$ of $f$. Removing $3\epsilon$-disks in $\R^{2k}$ and $i\R^{2k}$ around $0$ and replacing them by the compact part of $H_{\pm}$ bound by the spheres of radius $3\epsilon$ in $\R^{2k}$ and $i\R^{2k}$, we construct two embedded Lagrangian submanifolds $L_{\pm}$. Topologically, $L_{\pm}$ is $K$ with a $1$-handle attached. In particular, $H_1(L_{\pm};\Z)\cong\Z$.
If $L\subset \C^{n}$ is a Lagrangian submanifold,  let $\mu_{L}\in\Z_{\ge 0}$ denote its minimal Maslov number. 

\begin{Lemma}\label{Lem:Maslovnumber}
Let $f\colon K\to\C^{2k}$ be a Lagrangian immersion as in Lemma \ref{Lem:Maslov=0} and let $L_{\pm}$ be constructed by Lagrange surgery on $f$, as described above. Then
\[
\mu_{L_+}=2k\quad\text{ and }\quad \mu_{L_-}=2.
\] 
\end{Lemma}

\begin{Remark}
Note that $L_{\pm}$ is orientable if and only if $\mu_{L_{\pm}}$ is even. Hence, $L_{\pm}$ is orientable. 
\end{Remark}

\begin{proof}
Observe that the Maslov index of the loop which starts in the $i\R^{2k}$-sheet of $K$ at $p'_+$ in the $3\epsilon$-sphere around $p_+$, follows $K$ to $p'_-$ in the $3\epsilon$-sphere around $p_-$ in the $\R^{2k}$-sheet, and then connects $p'_-$ to $p'_+$ across the added handle $H_{\pm}$ equals
\[
\phi_{K}(p'_-)-\phi_{K}(p'_+)-\phi_{H_{\pm}}(p'_-)+\phi_{H_{\pm}}(p'_+),
\] 
where $\phi_{L}$ denotes a phase function of the Lagrangian $L$. The lemma then follows from \eqref{Eq:phasegrading} and Lemma \ref{Lem:LhandleMaslov}.
\end{proof}

\begin{Example}\label{Ex:LefschetzWhitney}
Let $\pi\colon \C^{2k} \to \C$ be the model Lefschetz fibration 
\[
\pi(z_1,\ldots, z_{2k})= \sum_{j=1}^{2k} z_j^2.
\]
Let $\gamma$ be a smooth embedded closed curve in $\C$ passing through the origin.  The union of the vanishing cycles of $\pi$ along $\gamma$ defines an immersed Lagrangian sphere $W\subset\C^n$ with a single transverse double point, which is one model for the Whitney immersion, see Section \ref{Sec:Intro}.  For instance, if $\gamma \cap B_0(\delta) = (-\delta, \delta) \subset \R$, then the vanishing cycles $V_t = \sqrt{t}S^{2k-1} \subset \pi^{-1}(t)$ along $\gamma$ are locally given by
\[
\bigcup_{t \in [0,\delta)} V_t \ = \ \R^{2k}; \quad \bigcup_{t\in (-\delta,0]} V_t = i\R^{2k}
\]
which meet transversely. The two surgeries of $W$ are given by perturbing $\gamma$ to an embedded curve in $\C^*$ which either does or does not enclose the origin. The Maslov $n$  surgery $W_+ = W'$ is Lagrangian isotopic to the Lagrangian $S^{1}\times S^{2k-1}$ in the model Lefschetz fibration,
\[
W' \ = \ \bigcup_{t\in S^1} \sqrt{t} S^{2k-1},
\]
which is also the mapping torus of the antipodal map on $S^{2k-1}$.
\end{Example}

\subsection{Homotopy type}

In \cite{Damian}, Damian defined the ``lifted Floer homology" of monotone Lagrangian embeddings, which  is essentially the Floer homology of the given embedding equipped with a local system defined by pushing forward the trivial local system from the universal cover. He used it to derive constraints on the possible fundamental groups of monotone Lagrangian submanifolds of $\bC^n$ of large Maslov number. For the case under consideration here, his work implies the following.

\begin{Lemma}\label{Lem:Damian}
With notation as in Lemma \ref{Lem:Maslovnumber}, if $k\ge 3$ then $L_+$ is a smooth manifold which fibers smoothly over the  circle with fiber a homotopy $(2k-1)$-sphere.
\end{Lemma}

\begin{proof}
Lemma \ref{Lem:Maslovnumber} implies that $L_+$ is a monotone Lagrangian submanifold of $\C^{2k}$ with minimal Maslov number $2k$. Since every compact subset of $\C^{2k}$ is displaceable by Hamiltonian isotopy, \cite[Theorem 7, last two lines]{Damian} gives the result.  (Note that existence of a \emph{smooth} fibration is a consequence of the exactness of the Morse-Novikov inequalities, whereby vanishing of Novikov homology implies existence of a non-singular closed 1-form, see \cite{Pazhitnov}.)
\end{proof}

\begin{Corollary}\label{Cor:LandKdiffeo}
The manifold $K$ is diffeomorphic to a homotopy sphere $\Sigma$ of dimension $2k$ and the manifold $L_+$ is diffeomorphic to $(S^{1}\times S^{2k-1})\#\Sigma$.  
\end{Corollary}

\begin{proof}
Smoothly, the manifold $L_+$ is obtained by adding a $1$-handle to $K$. To obtain $K$ from $L_+$ we add a $2$-handle that cancels the original $1$-handle. Since $L_+$ fibers over the circle with fiber a homotopy sphere it follows that $K$ is a homotopy sphere. Finally, it is clear that the manifold that results from adding a $1$-handle to $\Sigma$ is $(S^{1}\times S^{2k-1})\#\Sigma$.    
\end{proof}

\begin{Remark} \label{Rem:Cerf}
Damian's result \ref{Lem:Damian} does not itself exclude exotic spheres from admitting Lagrangian immersions with a single double point.  For $n\geq 6$, Cerf's pseudo-isotopy theorem \cite{Cerf} implies that the group $\Theta_n$ of $h$-cobordism classes of exotic spheres corresponds bijectively to the (orientation-preserving) mapping class group $\pi_0 \Diff (S^{n-1})$ of the $(n-1)$-sphere, via the construction of homotopy spheres from clutching functions.  One can correspondingly re-interpret the connect sum $\Sigma' = \Sigma \# (S^{1}\times S^{2k-1})$ as the total space of a fiber bundle over the circle, with fiber $S^{2k-1}$ and monodromy the diffeomorphism corresponding to $\Sigma$.
\end{Remark}

\begin{Remark} \label{Rem:Diffeomorphic}
Let $P$ be the manifold $S^1 \times S^{2k-1}$. For any exotic sphere $\Sigma$, the manifolds $T^*P$ and $T^*(P \#\Sigma)$ \emph{are} diffeomorphic.  
Indeed, both $P$ and $P \# \Sigma$ admit Morse functions with exactly 4 critical points. The difference in their differentiable structures comes from the attaching maps of their top-dimensional handles. A handle decomposition for a cotangent bundle is obtained from a handle decomposition of its base by thickening all handles. In particular, for the cotangent bundles considered here, the attaching $(2k-1)$-spheres for top-dimensional handles have codimension $2k$ in the boundary and are thus smoothly isotopic.
\end{Remark}


\section{Cobordisms from Cauchy-Riemann problems}\label{Sec:bounding}
In this section we construct a parallelizable manifold bounding any homotopy $2k$-sphere $\Sigma$ which admits an exact Lagrangian immersion into $\C^{2k}$ with only one double point,  thereby proving Theorem \ref{Thm:Main}. The proof uses moduli spaces of holomorphic disks. In order not to obscure the steps of the construction, we defer the functional analytic arguments needed to establish transversality and gluing results for these moduli spaces  to Sections \ref{Sec:basicsetup} -- \ref{Sec:gluing2}.

We will use the following notation below. Consider standard coordinates $x+iy=(x_1+iy_1,\dots,x_n+iy_n)$ on $\C^{n}$. We write
\begin{equation}\label{Eq:standardnotation}
\omega_0=\sum_{j=1}^{n} dx_j\wedge dy_j\quad\text{ and }\quad
J_0\colon T\C^{n}\to T\C^{n}
\end{equation}
for the standard symplectic structure and the standard complex structure (which corresponds to multiplication by the complex unit $i$) on $\C^{n}$, respectively.
If $(M, \omega)$ is a symplectic manifold and $H\colon M\times[0,1]\to\R$ is a smooth time-dependent Hamiltonian function then we write $X_H\colon M\times[0,1]\to TM$ for the time-dependent Hamiltonian vector field of $H$, defined by the equation $\omega(X_H, \cdot) = dH_t$, where $H_t\colon M\to M$ denotes the restriction $H|_{M\times\{t\}}$.

\subsection{Cauchy-Riemann equations and Hamiltonian displacement}\label{ssec:ham}
Let $L \subset \C^n$ be a Lagrangian submanifold with $\pi_1(L)=\Z$ and of minimal Maslov number $n$. We assume that $L$ is real analytic, see Lemma \ref{Lem:reanbdry}. Consider a compactly supported time-dependent Hamiltonian function
$H\colon \C^n \times [0,1] \to \R$ and suppose that the following hold.
\begin{itemize}
\item There exists $\epsilon_0>0$ such that $H_t$ is constant for $t\in [0,\epsilon_0]\cup[1-\epsilon_0,1]$. 
\item The time $1$ flow $\phi^1$ of the Hamiltonian vector field $X_{H}$ of $H$ displaces $L$ from itself: $\phi^1 (L) \cap L = \varnothing$.
\end{itemize}
Consider the strip $\R\times[0,1]\subset\C$ with coordinates $s+it$. Fix a smooth family of functions $\alpha_r\colon \R\to[0,1]$, where $r\in[0,\infty)$, with the following properties:
\begin{itemize}
\item $\alpha_{r}=1$ for $|s|\le r$ and $\alpha_{r}=0$ for $|s|\ge r+1$.
\item $\frac{d\alpha_{r}}{ds}\ge 0$ for $s\le 0$ and  $\frac{d\alpha_{r}}{ds}\le 0$ for $s\ge 0$.
\end{itemize}
We also fix a non-decreasing smooth function $\beta\colon [0,\infty)\to[0,1]$ such that $\beta(r)=0$ near $r=0$ and $\beta(r)=1$ for $r\ge 1$.  If $D$ denotes the unit disk in $\C$, then there is a unique conformal map 
\begin{equation} \label{Eqn:xi}
\xi\colon \R\times[0,1]\to D \quad  \text{with} \quad \xi(\pm \infty)=\pm 1 \ \text{and} \ \xi(0)=-i,
\end{equation}
with inverse 
\begin{equation}\label{Eqn:xi-1}
\xi^{-1}=s+it\colon D\to\R\times[0,1].
\end{equation}
Write $\gamma_r$ for the $1$-form on $D$ such that 
\[
\xi^{\ast}\gamma_{r}=\beta(r)\alpha_{r}(s)\,dt.
\]  

Fix a small $\delta>0$ and let $\JJ_L$ denote the space of almost complex structures on $\C^{n}$ which agrees with $J_0$ in a $\delta$-neighborhood of $L$. Associated to $J\in \JJ_L$ and $r\in [0,\infty)$ is a Floer equation (which is a perturbed Cauchy-Riemann equation):
\begin{equation} \label{Eqn:CR1}
(du + \gamma_r \otimes X_{H})^{0,1} = 0 \quad \text{for} \ u\in C^{\infty}\left((D^2, \pa D),(\C^n, L)\right).
\end{equation}
Here the vector field $X_{H}$ depends on $z\in D$, and is given by $X_{H}(u(z),t(z))$, see \eqref{Eqn:xi-1}, and
$A^{0,1}$ denotes the $(J, i)$ complex anti-linear part of the linear map $A\colon TD\to T\C^{n}$. Since $H_t$ is constant near $t=0,1$ and since $\gamma_r$ has compact support, it follows that for each $r_0$ there exists $\delta_0$ such that $\gamma_r\otimes X_H=0$ in a $\delta_0$-neighborhood of $\pa D$. Since in addition $J=J_0$ near $L$,   \eqref{Eqn:CR1} reduces to the ordinary $\bar\pa$-equation $\bar\pa_{J_0} u=du^{0,1}=0$ in this $\delta_0$-neighborhood.  

Via the identification $\xi$, the equation in local co-ordinates $s+it\in\R\times[0,1]$ is 
\begin{equation} \label{Eqn:CR-2}
\partial_s u + J \left( \partial_t u - \beta(r)\alpha_r(s) X_H(u(\xi(s,t)), t)\right) = 0,
\end{equation}
with boundary conditions $u(s,0)\in L$ and $u(s,1)\in L$.  Removal of singularities implies that any solution of \eqref{Eqn:CR-2} with finite energy
\[
\int_{\R \times [0,1]} (\left| \partial_s u \right|^2+\left| \partial_t u \right|^2)\; ds\wedge dt \ < \infty
\]
extends smoothly to the disk, hence gives a solution of the Floer equation \eqref{Eqn:CR1}. We call solutions of the Floer equation \emph{Floer holomorphic disks}.

Recall that $\pi_1(L)=\Z$ and fix the generator $\beta\in\pi_1(L)$ of Maslov number $+n$.   Restriction to the boundary gives an isomorphism $\pi_2(\C^{n},L)\to\pi_1(L)$ and thus homotopy classes of disks $u\colon (D,\pa D)\to(\C^{n},L)$ are indexed by the integers. In particular, we have moduli spaces of Floer holomorphic disks in classes $j\beta$ for any $j\in\Z$.   Fix any $k\geq 2$.

\begin{Definition}
Let $\FF(j\beta)$ denote the space of pairs $(u,r)\in C^{m}((D,\pa D),(\C^{n},L))\times[0,\infty)$, $m>0$ of solutions to \eqref{Eqn:CR1} in homotopy class $j\beta$.  
The fiber of the canonical map $\FF(j\beta) \rightarrow [0,\infty)$,  $(u,r) \mapsto r$, will be denoted $\FF^r(j\beta)$.
\end{Definition}
 
 Elliptic regularity implies that all solutions to \eqref{Eqn:CR1} are actually smooth. Here we write $C^{m}$ rather than $C^{\infty}$ in order to indicate that $\FF(j\beta)$ inherits its topology from a Banach space, see Remark \ref{Rem:normsame}. In Section \ref{Sec:basicsetup} we set up the functional analytic framework for studying moduli spaces of Floer holomorphic disks that we actually use to define $C^{1}$-structures and we refer to there for a more precise treatment. In Lemma \ref{Lem:tvmdli} we recall the well-known proof of the following result:

\begin{Proposition} \label{Prop:Fismfld}
When $r=0$, every solution to \eqref{Eqn:CR1} is constant, $\FF^{0}(0\beta)$ is transversally cut out and canonically $C^{1}$-diffeomorphic to $L$, and when $r\gg0$, the equation has no solutions. Furthermore, for generic $J\in \JJ_{L}$ and Hamiltonian function $H\colon \C^{n}\times[0,1]\to\R$:
\begin{enumerate}
\item $\FF(0\beta)$ is a $C^{1}$-smooth $(n+1)$-manifold, with boundary $\partial \FF(0\beta) = \FF^0(0\beta)= L$. 
\item $\FF(-\beta)$ is a $C^{1}$-smooth closed 1-manifold and $\FF(j\beta)$ is empty for $j\le -2$.
\end{enumerate}
\end{Proposition}

\begin{Remark}\label{Rem:dimhmtpy}
The space of solutions to \eqref{Eqn:CR1} in the relative homotopy class $j\beta \in \pi_2(\C^n, L)$ for \emph{fixed} $r$ has formal dimension
\begin{equation} \label{Eqn:Dimension}
\dim\left(\FF^r(j\beta)\right)= n + \mu(j\beta) = n(1+j),
\end{equation}
by the Riemann-Roch theorem, where $\mu$ denotes the Maslov index, and adding $+1$ for the parameter $r\in[0,\infty)$ gives the dimensions $\dim(\FF(j\beta))=n(1+j)+1$ for the moduli spaces appearing in Proposition \ref{Prop:Fismfld}. 
\end{Remark}

\begin{Remark}\label{Rem:normsame}
The $C^{1}$-smooth structure on the moduli spaces, which is inherited from the ambient configuration space (a Banach manifold), is actually independent of the choice of that configuration space: by elliptic regularity, all Sobolev norms in question are equivalent to the $C^{0}$-norm on the space of solutions. 
\end{Remark}

\subsection{Broken disks}
The manifold $\FF(0\beta)$ of Floer holomorphic disks is necessarily non-compact since the $\Z_{2}$-degree of the boundary evaluation map into $L$, restricted to $\pa \FF(0\beta)=\FF^{0}(0\beta)$, equals $1$. However, it admits a canonical compactification, due to Gromov \cite{Gromov} and Floer \cite{Floer:lagrangian}, which we shall use to obtain a bounding manifold for $L$.  The Gromov-Floer compactification 
\[
\barFF(0\beta) \ = \ \FF(0\beta) \cup \FF^{\mathrm{bd}}(0\beta)
\]
is obtained by adding broken disks as solutions to the equations (\ref{Eqn:CR1}). To describe a broken disk we first introduce the notion of a \emph{holomorphic tree} (sometimes called a bubble tree), which is a finite rooted tree $\Gamma$ such that:  to each vertex $v\in\Gamma$ corresponds a $J_0$-holomorphic disk $u_{v}\colon (D,\pa D)\to(\C^{n},L)$; to the edges $e\in\Gamma$ adjacent to $v$ correspond pairwise distinct boundary marked points $\zeta_{e}\in\pa D$ in the domain of $u_v$; and there is one additional marked point on the boundary of the source disk $D$ of the root (distinct from all other marked points on $\pa D$), which we call the \emph{root point}. We require that each holomorphic disk is stable (i.e.~either non-constant or constant with at least three distinct boundary marked points), and, furthermore, that if $v$ and $w$ are vertices of $\Gamma$ connected by an edge $e$ then $u_{v}(\zeta_v)=u_{w}(\zeta_w)$. We define the homotopy class of a holomorphic tree as the sum of the homotopy classes of all its vertex disks.        
A broken disk in $\FF^{\mathrm{bd}}(0\beta)$ then comprises  
\begin{itemize}
\item a solution $u_0$ to the equations \eqref{Eqn:CR1} in a homotopy class $\beta_0$ with pairwise distinct marked points $\zeta_1,\dots,\zeta_\nu$; and
\item a collection of holomorphic trees $\Gamma_i$, $1\leq i\leq \nu$, such that $u_0(\zeta_i)=\ev_{\mathrm{root}}(\Gamma_i)$, where $\ev_{\mathrm{root}}$ denotes evaluation at the root point, in homotopy classes $\beta_i$.
\end{itemize}
Here the homotopy classes are subject to the constraint that
\begin{equation}\label{Eqn:BubbleSum}
\beta_0 + \sum_{j=1}^{\nu} \beta_j \ = \ 0.
\end{equation}

Critically, in our situation there is a unique possible configuration of such broken solutions. To give a precise statement, we introduce a further notation. Note that if $H_R\equiv 0$, the equations \eqref{Eqn:CR1} reduces to the usual Cauchy-Riemann equation
\begin{equation} \label{Eqn:UsualCR}
du^{0,1} = 0 \quad \text{for } \ u\colon (D,\pa D) \to (\C^n, L)
\end{equation}
for the given almost complex structure $J\in\scrJ_L$.

\begin{Definition}
Let $\MM(j\beta)$ denote the quotient space of solutions $u$ to \eqref{Eqn:UsualCR} in class $j\beta$, modulo the group $G=PSL_2(\bC)$ of conformal  automorphisms of the disk.  Furthermore, we define
\begin{enumerate}
\item $\FF^*(j\beta)$ as the space of solutions to \eqref{Eqn:CR1} where the disk has a single boundary marked point; 
\item $\MM^*(j\beta)$ as the space of solutions to \eqref{Eqn:UsualCR} where the disk has a single boundary marked point.
\end{enumerate}
\end{Definition}

The space $\FF^*(k\beta) \cong \FF(k\beta) \times \partial D$ is diffeomorphic to a product (see Remark \ref{Rem:dimhmtpy}  for the induced topology), whilst there is a natural forgetful map $\MM^*(j\beta) \rightarrow \MM(j\beta)$ which is a fiber bundle with fiber $G/G_1 \cong S^1$ isomorphic to the group of rotations, where $G_{1}$ is the subgroup of automorphisms that fix the point $1\in\pa D$.  Again, Section \ref{Sec:basicsetup} recalls the standard analytic framework for studying these moduli spaces.

\begin{Proposition} \label{Prop:Mismfld}
After arbitrarily small perturbation of $J\in \scrJ_{L}$ or arbitrarily small real analytic Hamiltonian isotopy of $L$, one has:
\begin{enumerate}
\item $\MM(j\beta)$ is empty for $j<0$; 
\item $\MM(0\beta)$ is the set of constant maps;
\item $\MM(\beta)$ is a transversely cut out closed manifold of dimension $2n-3$.
\end{enumerate}
\end{Proposition}

That one can achieve transversality by perturbing $L$ but leaving $J=J_0$ fixed for disks in the primitive homology class $\beta$ is a theorem of Oh  \cite{Oh:Perturb}.  The other statements are standard, cf.~Lemma \ref{Lem:tvmdli}.  The smooth structure on $\MM(\beta)$ is inherited from the ambient Banach configuration space as follows. Fix a parameterization $u\colon D\to\C^{n}$ of an element in $\MM(\beta)$. Since $u|_{\pa D}$ is non-constant the derivative is non-zero at some point in the boundary. Fixing three disjoint real codimension one hyperplanes in $L \subset \bC^n$ near this point gives three marked points on the boundary of any disk in a neighborhood of $u$ that we use to fix parameterization and thereby construct local $C^{1}$-chart on $\MM(\beta)$ near $u$. Since $\MM(\beta)$ is compact it has a finite cover of charts as just described. It is straightforward to check that transition functions between charts are smooth, thus giving a $C^{1}$-structure on $\MM(\beta)$,  and that any two finite covers give rise to the same $C^{1}$-structure; see Section \ref{Sec:BasicStructure} for details.    

The moduli spaces $\MM^*(\beta)$ and $\FF^*(k\beta)$ come with canonical evaluation maps to $L$, which we denote by $\ev$.   Again, the following is standard, cf.~Section \ref{Sec:basicsetup} and Lemma \ref{Lem:fiberprod1}:

\begin{Proposition} \label{Prop:FibreProduct}
For generic $J\in\JJ_L$ and $H\in C^{\infty}(\C^{n}\times[0,1],\R)$, the product of the evaluation maps 
\[
\ev\colon\MM^*(\beta) \rightarrow L \quad \text{and} \quad \ev\colon \FF^*(-\beta) \rightarrow L,
\]
$\ev\times\ev$, is transverse to the diagonal $\Delta_L\subset L\times L$. Hence, the Gromov-Floer boundary $\FF^{\mathrm{bd}}(0\beta)$ is a closed $n$-manifold $C^1$-diffeomorphic to the fiber product of these maps:
\[
\FF^{\mathrm{bd}}(0\beta) \ = \ \FF^*(-\beta) \times_{L} \MM^*(\beta) \ = \ (\ev\times\ev)^{-1}(\Delta_{L}).
\]
\end{Proposition}

Most of Sections \ref{Sec:theboundary} and \ref{Sec:gluing2}, culminating in Theorem \ref{Thm:gluing}, is devoted to proving the following strengthening of the preceding result.

\begin{Theorem}
For generic $J\in\JJ_L$ and $H\in C^{\infty}(\C^{n}\times[0,1],\R)$, there is a $C^{1}$-embedding 
\[
\Psi\colon \FF^{\mathrm{bd}}(0\beta)\times[0,\infty)\to\FF(0\beta)
\]
with image in the interior of $\FF(0\beta)$ and
such that $\FF(0\beta)\backslash\Psi\left(\FF^{\mathrm{bd}}(0\beta)\times(0,\infty)\right)$ is a compact manifold with boundary $L\sqcup \Psi(\FF^{\mathrm{bd}}(0\beta)\times\{0\})$. 
Consequently, the compactified moduli space
\[
\barFF(0\beta) \ = \ \FF(0\beta) \cup \FF^{\mathrm{bd}}(0\beta)
\]
admits the structure of a $C^1$-smooth manifold with boundary, whose boundary is diffeomorphic to the disjoint union
\[
\partial \, \barFF(0\beta) \ = \  L \sqcup \FF^{\mathrm{bd}}(0\beta).
\]
\end{Theorem}

\subsection{Capping the outer boundary}
\label{ssec:capping}
To obtain a compact manifold with unique boundary component $L$, we will fill the outer boundary of $\barFF(0\beta)$.  We adapt an ingenious trick from \cite[Section 2c]{Abouzaid}.  The outer boundary is diffeomorphic, via a gluing map, to a fiber product
\[
\NN = \FF^{\mathrm{bd}}(0\beta) \, = \, \FF^*(-\beta) \times_L \MM^*(\beta).
\]
Whilst the second factor $\MM^*(\beta)$ is an unknown orientable $(2n-2)$-manifold, the first factor is a finite union of 2-tori. Let $\DD$ be a finite union of solid tori with boundary $\pa\DD=\FF^{\ast}(-\beta)$. Since $H_1(L;\Z)\approx \Z$ and $H_2(L;\bZ) = 0$, the map $\ev\colon\pa \DD\to L$ extends to a map 
\[
\overline{\ev}\colon\DD\to L.
\]

\begin{Lemma}
After arbitrarily small perturbation of $J\in\scrJ_L$ and Hamiltonian $H$, the fiber product
\[
\TT \ = \ \DD \times_L \MM^*(\beta) \ = \ (\overline{\ev}\times\ev)^{-1}(\Delta_L) 
\]
is transverse, and $\TT$ is a compact $C^1$-smooth manifold with boundary diffeomorphic to $\NN$.
\end{Lemma}

A proof is given in  Lemma \ref{Lem:fiberprod2}.  
Note that, since we perturb $J \in \scrJ_L$, the boundary $\pa\TT$ is not strictly equal to the initially given $\NN$: since the moduli spaces are slightly deformed the boundary which is the fibered product is deformed as well. However, the perturbation can be made arbitrarily small and hence there is an arbitrarily small diffeomorphism from the initial $\NN$ to the version of $\NN$ after perturbation. This gives in particular the following.

\begin{Corollary}
The space $\scrB = \barFF(0\beta) \cup_{\NN} \TT$ is a compact $C^1$-smooth manifold with $\pa\BB=L$.
\end{Corollary}

\subsection{Tangent and index bundles of moduli spaces}\label{ssec:tangentinfo}
For spaces of (Floer) holomorphic disks which are transversely cut out, the tangent bundle is described by index theory. In Section \ref{sec:closeddisks} we introduce a configuration space $\cfig$ of disks with boundary on $L$ modeled on a Sobolev space of maps $u: D\to\C^{n}$ with two derivatives in $L^{2}$ and which satisfy $(du)^{0,1}|_{\partial D} = 0$, and view the Floer operator as a map
\[
\bar\pa_{\rm F}\colon\XX\times[0,\infty)\to\YY,
\]  
where $\YY$ is the subspace of the Sobolev space of complex anti-linear maps $TD\to\C^{n}$ with one derivative in $L^{2}$ comprising elements that vanish along the boundary. The coordinate in $[0,\infty)$ parameterizes the Hamiltonian term of the Floer operator. We write $\XX(j\beta)$ for the component of $\XX$ containing maps that represent the homotopy class $j\beta$ when restricted to the boundary. Let $D(\bar\pa_{\rm F})\colon T\cfig\to\tcfig$ denote the linearization of $\bar\pa_{\rm F}$ and write
\[
\left[ T\XX(j\beta) \xrightarrow{D(\bar\pa_{\rm F})}\YY \right]
\]
for the index bundle of $D(\bar\pa_{\rm F})$ over $\cfig(j\beta)$. Similarly, we write $\bar\pa$ for the Cauchy-Riemann operator without Hamiltonian term and $D(\bar\pa)$ for its linearization.

\begin{Lemma} \label{Lem:Index}
For a generic $J\in\scrJ_{L}$ and Hamiltonian $H$, there are the following identities:
\begin{enumerate}
\item  in $KO(\FF(0\beta))$,
\[ 
T\FF(0\beta)\ \simeq \ \left.\left[ T\XX(0\beta) \xrightarrow{D(\bar\pa_{\rm F})}\YY \right] \right|_{\FF(0\beta)},
\]
\item in $KO(\FF^{\ast}(-\beta))$,
\[ 
T\FF^{\ast}(-\beta)\ \simeq \ \pi^{\ast}\left.\left[ T\XX(-\beta) \xrightarrow{D(\bar\pa_{\rm F})}\YY \right] \right|_{\FF(-\beta)},
\]
where $\pi\colon\FF^{\ast}(-\beta)\to\FF(-\beta)$ is the projection that forgets the marked point.
\item in $KO(\MM^{\ast}(\beta))$,
\[ 
T\MM^{\ast}(\beta)\ \simeq \ \left.\left[ T\XX(\beta) \xrightarrow{D(\bar\pa)}\YY \right] \right|_{\MM^{\ast}(\beta)},
\]
where we embed $\MM^{\ast}(\beta)\subset \XX(\beta)$ by choosing a smooth section of the bundle over $\MM^{\ast}(\beta)$ whose fiber is the (contractible) automorphism group $G_1$ of the given holomorphic disk. \end{enumerate}
\end{Lemma}

\begin{proof}
The identification of the $K$-theory class of the tangent bundle with the restriction of the corresponding linearized operator from the ambient space of smooth maps is a general feature of elliptic problems, cf.~for instance McDuff \cite[Proposition 4.3]{McDuff:K-class} for the Cauchy-Riemann case.
\end{proof}

The moduli spaces we work with all consists of smooth maps and we will therefore often consider the restriction of index bundles from the configuration space to the space $\cfig_{\rm sm}$ of $C^{m}$ maps for some large integer $m$.

\subsection{Index bundles and a piecewise linear isoptopy}
Assume that the dimension $n=\dim(L)$ is even, $n=2k>4$.
As a first step toward proving that the tangent bundle of $\scrB$ is trivial we show
that the index bundles in Lemma \ref{Lem:Index} can be understood through the corresponding bundles for the Lagrangian $S^{1}\times S^{2k-1}$, $W'\subset\C^{2k}$ which results from Lagrange surgery on the Whitney sphere $W$.

\begin{Lemma} \label{Lem:StablyTrivial}
The tangent bundle $TL$ of $L$ is stably trivial.
\end{Lemma}
 
\begin{proof}
The manifold $L$ is the connect sum of a homotopy $2k$-sphere and the product $S^{1}\times S^{2k-1}$, since each of the factors is stably parallelizable, see \cite[Theorem 3.1]{KM}, so is their connected sum. 
 \end{proof}

Any Lagrangian immersion $f\colon L \rightarrow \bC^n$ defines a Gauss map to the Lagrangian Grassmannian $U(n)/O(n)$ that takes $q\in L$ to $df(T_qL)\in U(n)/O(n)$. Stabilizing, one obtains a stable Gauss map 
\begin{equation} \label{Eqn:Gauss}
G_f\colon L \rightarrow U/O,
\end{equation}
if $f$ is the inclusion map we sometimes write $G_L$ instead of $G_f$.
  
\begin{Lemma} \label{Lem:GaussInterpolate}
There is a one-parameter family of $C^0$-homeomorphisms $\phi_t\colon \bC^{2k} \rightarrow \bC^{2k}$ and a 1-parameter family of continuous maps $\psi_t\colon S^1 \times S^{2k-1}  \rightarrow U/O$, with the properties that
\begin{enumerate}
\item $\phi_0$ is the identity;
\item $\phi_1(L) = W'$;
\item $\psi_0=G_L$ and $\psi_1 = G_{W'} \circ \phi_1$.
\end{enumerate}
\end{Lemma}

\begin{proof}
Any homotopy $n$-sphere $\Sigma$, $n> 4$, admits a Morse function with exactly two critical points \cite{KM}, and hence admits the structure of a {\sc pl}-manifold {\sc pl}-homeomorphic to the standard sphere $S^n$.  Consider the Lagrangian immersion $f\colon K\to\C^{2k}$ and the Whitney sphere $w\colon S^{2k}\to\C^{2k}$ as {\sc pl}-immersions. After an initial Hamiltonian isotopy of $\C^{2k}$ we may assume that $f$ and $w$ agree near their unique double points. Let $B\subset\C^{n}$ be a small ball around the double point and let $\tilde B\subset S^{2k}$ denote $f^{-1}(B)=w^{-1}(B)$.

Since the codimension of the immersions is $2k>3$ and $\C^{2k}-B$ is $4k-2$  connected, it follows from \cite{Hudson} that the {\sc pl}-embeddings $f,w\colon (S^{2k}-\tilde B,\pa \tilde B)\to (\C^{2k}-B,\pa B)$ are isotopic rel boundary through {\sc pl}-embeddings. The combinatorial isotopy extension theorem of  \cite{HudsonZeeman} then implies that there is a global continuous $1$-parameter family of {\sc pl}-homeomorphism $\phi_t\colon\C^{2k} \rightarrow \C^{2k}$, $0\le t\le 1$, with $\phi_0=\id$ and $\phi_1\circ f = w$.  

Consider now the stable Gauss maps $G_{f}$ and $G_{\phi_1\circ f}$, which map a $2k$-sphere into $U/O$.  Fix a stable framing of the standard sphere.  A choice of stable framing $\eta$ of the homotopy sphere $K$ induces a ``difference element" $\delta(\eta) \in \pi_{2k}(U)$ as follows.  By hypothesis, $\phi_1\circ f$ and $f$ agree near the double point. Since $f$ and $w$ have equal Maslov classes, the Gauss maps $G_{\phi_1 \circ f}$ and $G_f$ may further be homotoped (relative a neighbourhood of the double point) to agree on a neighbourhood of a path joining the two double point preimages.  The complement of such a path in either $S^{2k}$ or $K$ is a $2k$-disk.  Since the tangent bundles are stably framed, the Gauss maps naturally lift to $U$, and agree on the boundaries of these $2k$-disks, hence patch to define the element $\delta(\eta) \in \pi_{2k}(U)$.  Since the stable homotopy groups of $U$ vanish in all even dimensions,  we conclude that $G_{f}$ and $G_{\phi_1\circ f}$ are homotopic.


Since $L$ and $W'$ are obtained by Lagrange surgery on $f$ and $w$, respectively, which is a local operation near the double point the lemma clearly follows from the above.
\end{proof}

Lemma \ref{Lem:GaussInterpolate} shows that to understand the index bundle over the space $\XX_{\rm sm}(L)$ of smooth disks with boundary on $L$, it will be sufficient to understand the corresponding index bundle over  $\XX_{\rm sm}(W')$. In Section \ref{Sec:IndexBundle} we study this bundle explicitly using the fact that $\XX_{\rm sm}(W')$ is homotopy equivalent to the free loop space of $W'$, and working with a skeleton of the loop space obtained from the Morse-Bott energy functional of the round metric on $S^{2k-1}$. Our main results are summarized in the following lemma:

\begin{Lemma}\label{Lem:summarytrivindex}
The index bundle
\[
\left[ T\XX(j\beta) \xrightarrow{D(\bar\pa)}\YY \right]
\]
is trivial over the $(6k-7)$-skeleton of $\XX_{\rm sm}(W')$. 
\end{Lemma}

Lemmas \ref{Lem:Index}, \ref{Lem:GaussInterpolate}, and \ref{Lem:summarytrivindex} imply that a choice of a homotopy class of stable trivialization of the index bundle  over the $(6k-7)$-skeleton of $\XX_{\rm sm}(W')$ induces homotopy classes of stable trivializations of the tangent bundles to each of $\FF(0\beta)$, $\FF^{\ast}(-\beta)$, and $\MM^{\ast}(\beta)$. In the next section we describe how these trivializations interact near the Floer-Gromov boundary of $\FF(0\beta)$.
    
\subsection{Coherent trivialization}
\label{sec:cohertriv}
Let $\zeta$ be a boundary puncture on $\pa D$. Consider the conformal map $\psi_{\zeta}\colon D-\zeta\to H$, where $H=\{w=u+iv\colon v\ge 0\}$ is the upper half plane,
\[
\psi_{\zeta}(z)=-i\;\frac{\zeta^{-1}z+1}{\zeta^{-1}z-1},
\]
which takes $\zeta$ to $\infty$, $i\zeta$ to $-1$, and $-i\zeta$ to $1$. For $R>0$, consider the conformal map 
\[
\kappa_R\colon [0,\infty)\times[0,1]\to H,\quad
\kappa_R(\tau+it)= R e^{\pi(\tau+it)}.
\]
Let $E(R)\subset H$ be its image. Let $h$ be a metric on $H$ which agrees with the metric $du^{2}+dv^{2}$ in the disk of radius $2$, which agrees with the metric $d\tau^{2}+dt^{2}$ in $\tau+it$ coordinates on $E(3)$, and which interpolates smoothly between the two in the annular region $E(2)-E(3)$.

Thus $h$ is a metric on $H\approx D-\zeta$ in which the boundary puncture has a neighborhood which is a strip like end. For $\eta\in\R$, let $\mathbf{e}_{\eta}\colon H\to\R$ be the function
\begin{equation}
\mathbf{e}_{\eta}(w)=
\begin{cases}
1 &\text{ for }z\in H-E(3),\\
e^{\eta|\tau|} &\text{for }w=\tau+it\in[0,\infty)\times[0,1]\approx E(3).
\end{cases}
\end{equation}
In Section \ref{sec:punctures}, we construct configuration spaces for punctured (Floer) holomorphic disks using weight functions as above with $\eta=\delta\in(0,\pi)$. The configuration spaces are now modeled on the direct sum of a weighted Sobolev space with two derivatives in $L^{2}$ and a finite dimensional space of cut-off constant solutions supported near $\infty$. We denote the corresponding configuration space $\cfig_{\delta}$ and the corresponding Sobolev space of complex anti-linear maps which is the target for the $\bar\pa$-operator $\tcfig_{\delta}$. The $C^{1}$-structures on moduli spaces induced from $\cfig_{\delta}$ and $\cfig$ are canonically diffeomorphic. 

In this setting we introduce a linear gluing operation as follows. Recall that $L$ is assumed real analytic and that $J$ is standard near $L$. There is a natural map $\ev\colon \cfig_{\delta}\to L$ which is given by evaluation at the puncture. Let $\cfig_{\delta}(j\beta)$ denote the component of the space $\cfig_{\delta}$ with boundary in homotopy class $j\beta$. Consider two maps $u_1\in\cfig_{\delta}(j_1\beta)$ and $u_2\in\cfig_{\delta}(j_2\beta)$ with $\ev(u_{1})=\ev(u_{2})=q\in L$. Then there is $\rho_0>0$ such that for $\rho>\rho_0$, $[\rho_0,\infty)\times[0,1]$ is mapped to a standard coordinate neighborhood of $q$.

Write $H_{j}\approx H$, $j=1,2$ for the domain of $u_j$, $j=1,2$. Here we think of the strip neighborhood of $\infty$ in $H_2$ as $[0,\infty)\times[0,1]$, as usual. For simpler formulas we think of the neighborhood of $\infty$ in $H_1$ (where $\infty\in H_1$ corresponds to the puncture $\zeta\in\pa D$) as $(-\infty,0]\times[0,1]$. Note that for both domains the weight function looks like $e^{\delta|\tau|}$, $\tau+it\in[0,\infty)\times[0,1]$ for $H_2$ and $\tau+it\in(-\infty,0]\times[0,1]$ for $H_2$.

Let $\rho\ge 2$ and write 
\begin{align*}
H_{1;\rho} &= H_1-((-\infty,-\rho)\times[0,1]),\\
H_{2;\rho} &= H_2-((\rho,\infty)\times[0,1]). 
\end{align*}
Define
\begin{equation}\label{Eq:H_1+H_2}
H_{1}\#_{\rho} H_{2}= (H_{1;\rho}\cup H_{2;\rho})/\sim,
\end{equation}
where 
\[
\rho+it\;\sim\;-\rho+it,\quad\quad -\rho+it\in H_1\text{ and }\rho+it\in H_2.
\]

Then $H_{1}\#_{\rho} H_{2}$ comes equipped with a metric that agrees with the given metrics on $H_{1;\rho}$ and $H_{2,\rho}$ (which are the standard Euclidean metrics on the strip neighborhoods of $\infty$). Furthermore, $H_1\#_{\rho} H_2$ contains the finite strip region in the middle,
\begin{equation}\label{Eq:midstrip}
[-\rho,\rho]\times[0,1]\approx [0,\rho]\times[0,1]\cup[-\rho,0]\times[0,1]\subset H_1\#_\rho H_2.
\end{equation}
Here the weight functions of $H_{1;\rho}$ and $H_{2;\rho}$ fit together to a weight function on $H_1\#_{\rho} H_2$ given by $\tau+it\mapsto e^{\delta(|\rho|-|\tau|)}$ for $\tau+it\in[-\rho,\rho]\times[0,1]$, i.e.~in the middle strip. 

There is a preglued map $u_{1}\#_{\rho}u_{2}\colon H_{1}\#_{\rho} H_2 \rightarrow \bC^n$, which agrees with $u_j$ on $H_{j;\rho-1}$ and interpolates between the two maps in the remaining rectangle inside the standard neighborhood of $q$. 
Furthermore, both the Lagrangian boundary condition for the linearized operators $D(\bar\pa_{j})$, $j=1,2$ of the two disks and the operators themselves match where they are glued and hence induce an operator $D(\bar\pa_{\rho})$ and a Lagrangian boundary condition along $H_{1}\#_{\rho}H_{2}$. We will view this linearized operator as acting on the direct sum of a Sobolev space of $\C^{n}$-valued functions with two derivatives in $L^{2}$ weighted by $\mathbf{e}_{\delta}$ and which vanish at the point $0\in[-\rho,\rho]\times[0,1]$, and a $2k$-dimensional space of cut-off constant solutions $V_{\rm sol}$ supported in $[-\rho,\rho]\times[0,1]$. It is straightforward to check that
\[
\index(D(\bar\pa_{\rho}))=\index(D(\bar\pa_{1}))+\index(D(\bar\pa_{2}))-2k.
\]
Since we are interested in index bundles we consider stabilizations of the linearized operators. More precisely, let $D(\bar\pa)\oplus \R^{N}$ denote the operator $D(\bar\pa)$ augmented by a map $\psi\colon\R^{N}\to \tcfig_{\delta}$, where $\psi(v)$ has compact support in the interior of $H$ for any $v$. 

Let $u_1\in\cfig_{\delta}(-\beta)$ and $u_2\in\cfig_{\delta}(\beta)$ and consider stabilizations $D(\bar\pa_{1})\oplus\R^{N_1}$ and $D(\bar\pa_{2})\oplus\R^{N_2}$ that are surjective.
Then for $\rho$ sufficiently large, since both $\psi_j\colon \R^{N_j}\to\YY_{\delta}$ map to elements with compact support in the interior of $H_j$, $j=1,2$, we get an induced stabilization $D(\bar\pa_{\rho})\oplus\R^{N_1}\oplus\R^{N_2}$. Consider the further stabilization $D(\bar\pa_{\rho})\oplus\R^{N_1}\oplus\R^{N_2}\oplus\R^{2k}$ obtained by adding an extra copy of $V_{\rm sol}\approx \R^{2k}$ to $D(\bar\pa_{\rho})\oplus\R^{N_1}\oplus\R^{N_2}$. In Section \ref{Sec:LinearGluing} we establish the following result.

\begin{Lemma}\label{Lem:L2project}
For all $\rho$ that are sufficiently large,  $L^{2}$-projection gives an isomorphism 
\[
\ker(D(\bar\pa_{\rho})\oplus\R^{N_1}\oplus\R^{N_2}\oplus\R^{2k})= \ker(D(\bar\pa_{1})\oplus\R^{N_1})\oplus \ker(D(\bar\pa_{2})\oplus\R^{N_2}).
\] 
\end{Lemma}
  
Lemma \ref{Lem:L2project} implies that stable trivializations of index bundles on the two pieces of a family of broken disks induce a stable trivialization of the index bundle after pregluing.  We formalize this with a notion of \emph{coherence} of trivializations. Write $\mathbf{I}_{j\beta}^{\ast}$ for the index bundle over $\cfig_{\delta}(j\beta)$ and $\mathbf{I}_{j\beta}$ for its restriction to the subspace $\cfig_{\delta}(1;j\beta)$ of disks with puncture at $1\in\pa D$.   Assume that,  over sufficiently large compact subsets of the configuration spaces $\cfig_{\delta}$,  we have fixed stabilizations  $Z^{\ast}_{-\beta}$ of $\mathbf{I}^{\ast}_{-\beta}$, $Z_{\beta}$ of $\mathbf{I}_{\beta}$, and $Z_{0\beta}$ of $\mathbf{I}_{0\beta}$ with properties as described before Lemma \ref{Lem:L2project}, so in particular the corresponding families of operators are everywhere surjective. Consider a compact CW-complex $Q$ with maps $p\colon Q\to \cfig_{\delta}(-\beta)$ and $q\colon Q\to\cfig_{\delta}(1;\beta)$ such that $\ev\circ \, p=\ev_1\circ \, q$. Assume furthermore that the map $p$ factors as follows:
\begin{equation} \label{Eqn:Qfactors}
\begin{CD}
Q @>{p'}>> A\times S^{1} @>{a\times\id}>> \cfig_{\delta}(-\beta)=\cfig_{\delta}(1;-\beta)\times\pa D
\end{CD},
\end{equation}
where $A$ is a compact CW-complex. Write $\Pre\colon Q\times[\rho_0,\infty)\to\cfig_{\delta}(0\beta)$ for the map that applies pregluing to the pair $(p,q)$ and consider the pull-back bundle $\Pre^{\ast}\mathbf{I}_{0\beta}$ over $Q$. By Lemma \ref{Lem:L2project} we find that there are two stable trivializations of this bundle, namely those given by 
\[
p^{\ast}Z_{-\beta}^{\ast}\oplus q^{\ast}Z_{\beta} \quad \textrm{and} \quad \Pre^{\ast}(Z_{0\beta}\oplus Z_{TL}),
\]
where $Z_{TL}$ is a fixed trivialization of $TL$.

\begin{Definition}\label{Def:CoherTriv}
The triple of stable trivializations $Z^{\ast}_{-\beta},Z_{\beta}, Z_{0\beta}$ is \emph{$(d',d)$-coherent} if the two stable trivializations $p^{\ast}Z_{-\beta}^{\ast}\oplus q^{\ast}Z_{\beta}$ and $\Pre^{\ast}(Z_{0\beta}\oplus Z_{TL})$ are homotopic for all $Q$ and $A$ as above with $\dim(A)\le d'$ and $\dim(Q)\le d$. 
\end{Definition}

The rather roundabout nature of Definition \ref{Def:CoherTriv} has avoided needing to talk about 
transversality of fiber products. The key result, proved in Section \ref{Sec:IndexBundle}, is then the following.  As always, by a stable trivialization of a given index bundle we mean one over some sufficiently large compact subset.

\begin{Lemma} \label{Lem:CoherTrivExist}
Let $d\leq 6k-7$. For any stable trivialization $Z_{-\beta}^{\ast}$ of $\mathbf{I}^{\ast}_{-\beta}$, there are stable trivializations $Z_{\beta}$ of $\mathbf{I}_{\beta}$ and $Z_{0\beta}$ of $\mathbf{I}_{0\beta}$ such that $(Z_{-\beta}^{\ast},Z_{\beta},Z_{0\beta})$ is $(1,d)$-coherent. 
\end{Lemma}

\begin{Remark}
Coherent trivializations  refine the usual idea of ``coherent orientations", but involving framings of the index bundles  rather than of their associated determinant lines.  The axioms satisfied by Givental's quantum $K$-theory, cf.~\cite[Proposition 11]{Lee}, also rely on a version of Lemma \ref{Lem:L2project}, to ensure that the virtual Euler characteristics of the 2-term perfect obstruction theories over moduli spaces of stable maps are compatible under degeneration.  However, in the situations studied by Givental and Lee, the relevant index-type bundles are not stably trivial.
\end{Remark}

\subsection{Proofs of the main results}
We prepare for the proofs by presenting two general lemmas.
Let $\cong_{\mathrm{s}}$ denote stable isomorphism of vector bundles $E$ and $F$ over a space $X$, so $E \cong_{\mathrm{s}} F$ if and only if the bundles $E$ and $F$ define the same class in $\widetilde{KO}(X)$. Let $\Sigma X$ denote the suspension of $X$, so $\Sigma X = S^1 \wedge X$.

\begin{Lemma}  \label{Lem:KOobs}
Suppose $X = U \cup V$ is a decomposition of a cell complex into two subcomplexes. Suppose $E, F\rightarrow X$ are vector bundles with the property that $E|_U \cong_\mathrm{s} F|_U$ and $E|_V \cong_\mathrm{s} F|_V$. Then the obstruction to stable isomorphism of $E$ and $F$ over $X$ is a class in the $\widetilde{KO}$-theory of the suspension $\Sigma(U\cap V)$.
\end{Lemma}
  
  \begin{proof}
  Immediate from the Mayer-Vietoris sequence in $\widetilde{KO}$.
  \end{proof}

  \begin{Lemma} \label{Lem:KO}
 $\widetilde{KO}(\Sigma T^2) = \bZ_2 \oplus \Z_2$, with the summands detected by the $Spin$ structure on the circle factors of $T^2$.  \end{Lemma}
  
  \begin{proof} The suspension $\Sigma T^2$ contains a wedge of 2-spheres, from suspending the 1-skeleton of the torus, and a 3-sphere, from suspending the 2-cell of $T^2$. The result then follows from the Mayer-Vietoris theorem in $\widetilde{KO}$-theory.  \end{proof}

\begin{proof}[Proof of Theorem \ref{Thm:Main}]
Consider the (stable) tangent bundle of $\scrB$ restricted to the pieces of $\scrB$: 
\begin{enumerate}
\item On $\barFF(0\beta)$, $T\scrB = \mathbf{I}_{0\beta}$.
\item In the collar on the boundary $\NN \times [0,\infty)$ the tangent bundle is stably given by $\mathbf{I}_\beta + \mathbf{I}^*_{-\beta} - TL$.
\item Over the filling $\TT$, a fiber product over the union of solid tori $\DD$, the tangent bundle is stably given by $T\DD + \mathbf{I}_\beta - TL$.
\end{enumerate}
We compare the stably trivialized virtual bundles on the overlap of $(2)$ and $(3)$: 
\[
\mathbf{I}_\beta + \mathbf{I}^*_{-\beta} - TL \quad \textrm{and} \quad T\DD + \mathbf{I}_\beta - TL
\]
near the boundary of the fiber product 
\[
\DD \times_L \MM^*(\beta).
\]  
We fix trivializations of the summands $\mathbf{I}_\beta$ and $TL$, and regard them as being extended trivially over $\DD$. The remaining summands $\mathbf{I}^*_{-\beta}$ and $T\DD$ are both pulled back from a 2-dimensional torus inside $\NN$, which brings us into the situation of Lemma \ref{Lem:KOobs} with $U \cap V \approx T^2$.  More explicitly, the $1$-skeleton of $\cfig_{\rm sm}(-\beta)$ corresponds to the non-trivial loop in $L$, and the preimage of this 1-skeleton defines a torus $T^2_{-\beta}$ in $\cfig_{\rm sm}(-\beta)\times\pa D$.  Under evaluation into $L$, and in a suitable homology basis, exactly one of the two circle factors of this torus bounds.  We pick a trivialization of the bundle $\mathbf{I}^*_{-\beta}$ corresponding to the bounding spin structure on the bounding loop of $T^2_{-\beta}$ and which is arbitrary on the non-bounding loop of $T^2_{-\beta}$. This trivialization gives a stable trivialization of $T\DD$ near the boundary of $\DD$, which we claim extends to all of $\DD$. Indeed, Lemma \ref{Lem:KO} implies that the only condition for this is that the spin structure on any bounding loop is bounding.    

Since $k>2$, the space $\barFF(0\beta)$ has dimension $2k+1 < 6k-7$. Lemma \ref{Lem:CoherTrivExist} then implies that there is a stable trivialization of $T\barFF(0\beta)$ which is compatible with (meaning homotopic to) the stable trivialization induced by the triple $\index(\beta)$, $\mathbf{I}^{\ast}_{-\beta}$ and $TL$ near the outer boundary component $\NN$. It follows that the  stable trivialization over the collar region extends both to the Floer moduli space $\FF_{\rho_0}(0\beta)$ and to the capping space $\TT$, so we conclude that $T\BB$ is stably trivial.  Since $\pa\BB=L\ne \varnothing$, stable triviality of $T\BB$ implies $\BB$ is parallelizable, cf.~\cite[Lemma 3.5]{KM}.

It remains to show that if $L$ bounds a parallelizable manifold then so does the original homotopy sphere $K$. By construction, $L$ was obtained from $K$ by adding a $1$-handle, so $K$ may be obtained from $L$ by adding a canceling 2-handle. There are two possibilities:  the attaching circle of this 2-handle is trivial in $H_1(\BB;\bZ_2)$, meaning the circle mod 2 bounds, or it is non-trivial. In the former case the trivialization of $T\BB$ along the attaching sphere necessarily corresponds to the bounding spin structure, and hence extends over the 2-handle. In the latter case, up to homotopy we may choose the trivialization of $T\BB$ freely over the attaching circle, hence may choose it to correspond to the null-cobordant spin structure, and hence the trivialization extends over the 2-handle in this case as well. It follows that $K$ bounds a parallelizable manifold. 
\end{proof}

The proofs of the corollaries use the following result.
\begin{Lemma}\label{Lem:forCor}
Let $M \subset \bC^{8k}$ be a monotone Lagrangian submanifold {\sc pl}-homeomorphic to $S^{1}\times S^{8k-1}$ and of Maslov number $8k$. Then there is a {\sc pl}-isotopy $\phi_t\colon \C^{8k}\to\C^{8k}$ such that $\phi_0=\id$ and $\phi_1(M)=W'$, the surgery on the Whitney sphere, and a homotopy of stable Gauss maps $\psi_t\colon M\to U/O$ such that $\psi_0=G_M$ and $\psi_1=G_{W'}\circ\phi_1$.
\end{Lemma}

\begin{proof}
The existence of $\phi_t$ follows from a repetition of the corresponding argument in Lemma \ref{Lem:GaussInterpolate}. The homotopy of stable Gauss maps can then be constructed as follows. Since $\pi_{8k-1}(U/O)=0$, we can find the desired homotopy over $\{pt\}\times S^{8k-1}$. The Maslov number then allows us to extend the homotopy further over $S^{1}\vee S^{8k-1}$. What remains is the extension over an $(8k)$-cell. Using the triviality of the tangent bundle we lift the map to $U$ and use $\pi_{8k}(U)=0$.
\end{proof}

\begin{proof}[Proof of Corollary \ref{Cor:First}]
Lemma \ref{Lem:forCor} implies that the index bundles determined by $L$ are homotopically the same as those determined by $W'$. The proof of Theorem \ref{Thm:Main} then shows that $L$ bounds a parallelizable manifold.
\end{proof}

\begin{proof}[Proof of Corollary \ref{Cor:Main}]
Recall that $P= S^1 \times S^{8k-1}$. Suppose that $\Sigma$ is an exotic $8k$-sphere with the property that
\[
(T^*(\Sigma \# P), d\theta_{\rm can}) \ \cong \ (T^*P, d\theta_{\rm can})
\]
are symplectomorphic when equipped with their canonical symplectic structures.  There is then a Lagrangian embedding $\iota: \Sigma \# P \rightarrow T^*P$ which moreover induces an isomorphism on first cohomology.  By composing $\iota$ with an appropriate ambient translation in the $T^*S^1$ factor of $T^*P = T^*S^1 \times T^*S^{8k-1}$, we obtain a new Lagrangian embedding $\sigma: \Sigma \# P \rightarrow T^*P$ for which the period of $\sigma^*(\theta_{can})$ vanishes. Thus $\sigma$ yields an \emph{exact} Lagrangian embedding
\[
\sigma\colon \ \Sigma \# P \ \hookrightarrow \ T^*P
\]
 in $T^*P$, in the homology class of the zero-section.  Results of Abouzaid and Kragh imply that any closed exact Lagrangian in the cotangent bundle has vanishing Maslov class \cite[Proof of Theorem E.2]{Kragh}.  We now view $W' \approx P\hookrightarrow \bC^{2k}$ as an embedded Lagrangian submanifold, via surgery on the Whitney sphere. By composing $\sigma$ with a conformal symplectomorphism which radially shrinks the fibers, if necessary, we can assume that it has image inside the disk bundle $D_{\varepsilon}(T^*P)$ for any given $\varepsilon>0$ (measured with respect to an arbitrary choice of metric).  By Weinstein's theorem, such small disk bundles embed symplectically into $\bC^{8k}$, so we obtain a Lagrangian embedding 
\[
\sigma'\colon \Sigma \# P \ \hookrightarrow \ \bC^{8k}.
\]
Since $\sigma$ had Maslov class zero, the Maslov class of $\sigma'$ is obtained by restriction from the embedding $D_{\varepsilon}(T^*P) \subset \bC^{8k}$, hence $\Sigma \# P \subset \bC^{8k}$ has Maslov class $8k$ (equal to that of $P=W'$). Corollary \ref{Cor:First} therefore implies that $\Sigma$ bounds a parallelizable manifold, which (via Kervaire and Milnor) gives Corollary \ref{Cor:Main}. 
\end{proof}


\section{Moduli spaces -- set up and basic properties}
\label{Sec:basicsetup}
In this section we present the basic functional analytic set up for moduli spaces. Much of the material is standard and appears (with small variations) in many places, see e.g.~\cite{EES1}. We are including it in order to make the $C^{1}$-structures on moduli spaces, which are of central importance to our main results, maximally explicit. Proofs, however, will sometimes be sketched.   

Throughout this section we will write $L\subset\C^{n}$ for a Lagrangian submanifold diffeomorphic to the mapping torus of the antipodal map $(S^{1}\,\widetilde{\times}\,S^{n-1})\#\Sigma$, where $\Sigma$ is a homotopy $n$-sphere, with minimal Maslov number $n$. We write $\beta\in\pi_1(L)=H_1(L)\cong \Z$ for the generator of Maslov index $n$. Note that the Lagrangian submanifold $L_+$ in Lemma \ref{Lem:Maslovnumber} constructed by Lagrange surgery from an immersion $f\colon K\to\C^{2k}$ with exactly one transverse double point has the properties of $L$, with homotopy sphere $\Sigma$ diffeomorphic to $K$, see Corollary \ref{Cor:LandKdiffeo}, and that in even dimensions the mapping torus is a product.

\subsection{Geometric preliminaries}
\label{sec:Geomprel}
We discuss two constructions in the underlying geometry of $L\subset\C^{n}$ that will simplify our functional analytic treatment of (Floer-)holomorphic disks. Choose a real analytic structure on $L$.
\begin{Lemma}\label{Lem:reanbdry}
Real analytic Lagrangian embeddings of $L$ into $\C^{n}$ are dense in the space of all Lagrangian embeddings.
\end{Lemma}
\begin{proof}
Consider a Lagrangian embedding $\phi\colon L\to\C^{n}$. Let $\mathfrak{a}=\int_{\beta} \phi^{\ast}(y\cdot dx)$. Then $L$ admits a Legendrian lift into $\C^{n}\times (\R/\mathfrak{a}\Z)$ with contact form $dz-y\cdot dx$. Furthermore, the lift is unique up to translation in the added circle direction. After arbitrarily small perturbation of $\phi$, the projection $\phi(L)\to\R^{n}\times(\R/\mathfrak{a}\Z)$ is a front with front-generic singularities. Such fronts determine the corresponding Lagrangian embeddings uniquely (using $y_j=\frac{\pa z}{\pa x_j}$). The fronts arising from approximations of the projection by real analytic maps $L\to\R^{n}\times(\R/\mathfrak{a}\Z)$ are still generic, hence give real analytic Lagrangian embeddings arbitrarily near $\phi$.
\end{proof}

Let $B_{\C^{n}}(r_q)$ denote the ball of radius $r_q>0$ around $0\in\C^{n}$, and let $B_{\R^{n}}(r_q)=B_{\C^{n}}(r_q)\cap \R^{n}$. If $L\subset \C^{n}$ is a real analytic Lagrangian submanifold then any point $q\in L$ has a neighborhood $(W_q, L \cap W_q) \subset (\C^{n},L)$ that is bi-holomorphic via $\psi_q$ to $(B_{\C^{n}}(r_q), B_{\R^n}(r_q))$.
Furthermore, by compactness of $L$, we may take $r_q=r'>0$, where $r'$ is independent of $q$.    
We will use the following notation below. Take $0<r<\frac{r'}{\sqrt{2}}$ and define a  product neighborhood of $q\in L$:
\begin{equation}\label{Eq:disk^2nbhd}
U_q=\psi_q\left(B_{\R^{n}}(r)\times iB_{\R^{n}}(r)\right).
\end{equation}

We next review a construction in \cite[Section 5.2]{EES1}. Consider the restriction $T_L(TL)$ of the tangent bundle of $TL$ to the $0$-section $L\subset TL$. Let $J_{L}\colon T_L(TL)\to T_L(TL)$ denote the natural complex structure which maps a horizontal vector tangent to $L\subset TL$ at $q\in L$ to the corresponding vector tangent to the fiber $T_qL\subset TL$ at $0\in T_qL$.  Using Taylor expansion in the fiber directions, it is straightforward to check that the inclusion $\iota\colon L\to \C^{n}$ admits an extension $P\colon U\to \C^{n}$, where $U\subset TL$ is a neighborhood of the $0$-section, such that $P$ is an immersion with $J_0\circ dP = dP\circ J_L$ along the $0$-section $L\subset U$.  (Recall that $J_0$ denotes the standard complex structure on $\C^n$.) 

We next construct a metric $\hat g$ on a neighborhood of the $0$-section in $TL$. Fix a Riemannian metric $g$ on $L$. 
Let $v\in TL$ with $\pi(v)=q$. Let $X$ be a tangent vector to $TL$ at $v$. The Levi-Civita connection of  $g$ gives the decomposition $X=X_{H}+X_{V}$, where $X_{V}$ is a vertical tangent vector, tangent to the fiber, and where $X_{H}$ lies in the horizontal subspace at $v$ determined by the connection. Since $X_V$ is a vector in $T_qL$ with its endpoint at $v\in T_qL$ we can translate it linearly to the origin $0\in T_qL$; we also use $X_V$ to denote this translated vector. Write $\pi X\in T_qL$ for the image of $X$ under the differential of the projection $\pi\colon TL\to L$. Let $R$ denote the curvature tensor of $g$ and define the field $\hat g$ of quadratic forms along $TL$ as follows: 
\begin{equation}
\hat{g}(v)(X,Y)=g(q)(\pi X,\pi Y)+g(q)(X_{V},Y_{V})+g(q)(R(\pi X,v)\pi Y,v),
\end{equation} 
where $v\in TL$, $\pi(v)=q$, and $X,Y\in T_v(TL)$.

\begin{Lemma}
There exists $\delta>0$ such that $\hat g$ is a Riemannian metric on $\{v\in TL\colon g(v,v)< \delta\}$. In this metric the $0$-section $L$ is totally geodesic and the geodesics of $\hat g$ in $L$ are exactly those of the original metric $g$. Moreover, if $\gamma$ is a geodesic in $L$ and $X$ is a vector field in $T(TL)$ along $\gamma$ then $X$ is a Jacobi field if and only if $J_L X$ is. 
\end{Lemma}

\begin{proof}
This is \cite[Proposition 5.3]{EES1}.
\end{proof}

Consider the immersion $P\colon U\to \C^{n}$, where $U$ is a neighborhood of the $0$-section in $TL$, with $J_0\circ dP=dP\circ J_L$. The push forward under $dP$ of the metric $\hat g$ gives a metric on a neighborhood of $L$ in $\C^{n}$. Extend the metric $\hat g$ to a metric, still denoted $\hat g$, on all of $\C^{n}$ which we take to agree with the standard flat metric on $\C^{n}$ outside a (slightly larger) neighborhood of $L$. We write $\exp\colon T\C^{n}\to\C^{n}$ for the exponential map in the metric $\hat g$. Since $L$ is totally geodesic for $\hat g$, $\exp$ takes tangent vectors to $L$ to points on $L$.

\subsection{Configuration spaces for closed disks}
\label{sec:closeddisks}
In this section we describe a Banach manifold set-up for the study of \eqref{Eqn:CR1},  and derive basic properties of the solution spaces. We point out that if $X_H=0$, the operator in the right hand side of \eqref{Eqn:CR1} reduces to the standard $\bar\pa$-operator, $du^{0,1}=\bar\pa u$, whose solutions are $J$-holomorphic disks. Our framework covers the cases $X_H = 0$ and $X_H \neq 0$ simultaneously.

Consider the unit disk $D\subset \C$ with the standard metric and let $\gamma_r$, $r\in[0,\infty)$ and $H$ be as in Section \ref{ssec:ham}. Let 
$\sblv^{s}(D,\C^{n})$
denote the Sobolev space of $\C^{n}$-valued functions with $s$ derivatives in $L^{2}$ with the Sobolev $s$-norm (viewed as a vector space over $\R$).  More specifically, we consider the closure in the $s$-norm on $D$ of functions that extend smoothly to the double of $D$. We write
$\sblv^{s}(\pa D,\C^{n})$  
for the Sobolev space of functions on the boundary $\pa D$ with respect to the induced metric. If $s>\frac12$ then there is a continuous restriction (or trace) map
\[
\sblv^{s}(D,\C^{n})\to\sblv^{s-\frac12}(\pa D, \C^{n}),\quad u\mapsto u|_{\pa D}.
\]
Recall that elements in $\sblv^{s}(D,\C^{n})$ (respectively $\sblv^{s}(\pa D,\C^{n})$) are uniquely represented by $C^{l}$-functions, where $l$ is the largest integer such that $l<s-1$ (respectively such that $l<s-\frac{1}{2}$).

The metrics on $D$ and on $\C^{n}$  give metrics on all bundles of $\C^{n}$-valued tensors on $D$. Let $\Hom(TD,\C^{n})$ and $\Hom^{0,1}(TD,\C^{n})$ denote the bundles over $D$ with fiber at $z\in D$ equal to the space of linear maps $T_{z}D\to\C^{n}$ and $(i,J)$-complex anti-linear maps $T_{z}D\to\C^{n}$, respectively. We write 
$\sblv^{s}(D,\Hom(TD,\C^{n}))$ and $\sblv^{s}(D,\Hom^{0,1}(TD,\C^{n}))$ for the Sobolev spaces of sections with $s$ derivatives in $L^{2}$ of the bundles indicated. We will also use the subspace $\dot\sblv^{1}(D,\Hom^{0,1}(TD,\C^{n}))$ of sections that vanish on the boundary:
\begin{equation}\label{Eq:sblvdbar=0onbdry}
\dot\sblv^{1}(D,\Hom^{0,1}(TD,\C^{n}))=
\left\{
A\in \sblv^{1}(D,\Hom^{0,1}(TD,\C^{n}))\colon A|_{\pa D}=0
\right\}.
\end{equation}

The configuration space for \eqref{Eqn:CR1} is the subset $\cfig\subset\sblv^{2}(D,\C^{n})$ of all $u\in\sblv^{2}(D,\C^{n})$ that satisfy the following conditions:
\begin{enumerate}
\item $u|_{\pa D}(z)\in L\subset\C^{n}$ for all $z\in\pa D$.
\item $(du)^{0,1}|_{\pa D}=0\in\sblv^{\frac{1}{2}}(\pa D,\Hom^{0,1}(TD,\C^{n}))$. 
\end{enumerate}

\begin{Lemma}\label{Lem:expmapclosed}
The subset $\cfig\subset\sblv^{2}(D,\C^{n})$ is closed and is a Banach submanifold. The tangent space $T_{u}\cfig$ at $u$ is the subspace of vector fields  $v\in\sblv^{2}(D,\C^{n})$ that satisfy 
\begin{enumerate}
\item $v(z)\in T_{u(z)}L$ for all $z\in\pa D$ and
\item $(\nabla v)^{0,1}|_{\pa D}=0\in\sblv^{\frac{1}{2}}(\pa D,\Hom^{0,1}(TD,\C^{n}))$,
\end{enumerate}
where $\nabla$ denotes the Levi-Civita connection of the metric $\hat g$, see Section \ref{sec:Geomprel}. Furthermore, the map $\Exp\colon T_{u}\cfig\to\cfig$,
\[
[\Exp(v)](z)=\exp_{u(z)}(v(z)),\quad z\in D,
\]
gives local $C^{1}$-coordinates near $u$ when restricted to $v$ in a sufficiently small ball around $0 \in T_{u}\cfig$. 
\end{Lemma} 

\begin{proof}
This is a simpler version of \cite[Lemma 3.2]{EES} (see \cite[Proposition 5.9]{EES1} for the details of the standard argument alluded to there).
\end{proof}

The connected components of the space $\cfig$ are in $1-1$ correspondence with $\pi_2(\C^{n},L)\cong H_1(L)$. For $j\in\Z$, we write $\cfig(j\beta)$ for the connected component of all $u\in\cfig$ such that the continuous loop $u|_{\pa D}$ represents the class $j\cdot\beta$.

The target space for the operator in \eqref{Eqn:CR1} corresponding to $\cfig$ is  $\tcfig = \dot\sblv^{1}(D, \Hom^{0,1}(TD,\C^{n}))$. Consider the operator corresponding to the Floer equation, $\bar\pa_{\Fl}\colon \cfig\to\tcfig$,
\[
\bar\pa_{\Fl}(u)=(du + \gamma_r\otimes X_{H})^{0,1},
\]
where $r\in[0,\infty)$ and write $\bar\pa_{\Fl;j\beta}=\bar\pa_{\Fl}|_{\cfig(j\beta)}$.
\begin{Lemma}\label{Lem:Floermapindexclosed}
The map $\bar\pa_{\Fl;j\beta}$ is a $C^{1}$ Fredholm map of index
\[
\ind(\bar\pa_{\Fl;j\beta})=n(1+j),
\]
which depends smoothly on the parameter $r\in[0,\infty)$.
\end{Lemma}

\begin{proof}
The linearization of $\bar\pa_{\Fl}$ is
\[
D(\bar\pa_{\Fl})=\bar\pa v + K v,
\] 
where $v\in T_u\cfig\subset \sblv^{2}(D,\C^{n})$ and where the operator $K$ is compact. The Lagrangian boundary condition is the loop  $T_{u(z)}L$, $z\in\pa D$ which has Maslov index $jn$ if $u\in\cfig(j\beta)$. The lemma follows from the well-known index formula for the $\bar\pa_{\Lambda}$-operator on the disk with $n$-dimensional Lagrangian boundary condition $\Lambda(z)$, $z\in\pa D$ of Maslov index $\mu(\Lambda)$: $\ind(\bar\pa_{\Lambda})=n+\mu(\Lambda)$.
\end{proof}

Let $\FF^{r}(j\beta)={\bar\pa_{\Fl;j\beta}}^{-1}(0)$, with $r$  the parameter of the Hamiltonian term in the operator. Let
\[
\FF(j\beta) \ =\bigcup_{r\in[0,\infty)}\FF^{r}(j\beta)\times\{r\} \ \ \ \subset \ \ \ \cfig\times[0,\infty).
\]
In the case of trivial Hamiltonian term, i.e.~when $X_{H}=0$, we write $\bar\pa$ instead of $\bar\pa_{\Fl}$; the solution space $\bar\pa^{-1}(0)\subset\cfig$ consists of $J$-holomorphic disks, and carries an action of the group $G$ of conformal automorphisms of the disk. We write
\begin{equation}
\MM(j\beta)= \FF^0({j\beta})/G
\end{equation}

\subsection{Basic structure of moduli spaces} \label{Sec:BasicStructure}
In this section we collect standard properties of moduli spaces of Floer disks that follow from Floer-Gromov compactness and transversality. Since the moduli spaces $\MM(j\beta)$ of holomorphic disks are defined as quotients, in order to get a $C^{1}$-structure we need to fix gauge. To  this end we start with a discussion of marked points.

Consider an element $\hat u$ in $\MM(\beta)$ represented by a map $u\colon (D,\pa D)\to (\C^{n},L)$. Since $u$ is non-constant, $u|_{\pa D}$ is non-constant and there exists an arc $A\subset \pa D$ where $u|_{\pa D}$ is an immersion. Fix a codimension 1 hypersurface $T\subset L$ such that $u(\zeta)\in T$, $du(\zeta)\ne 0$ and with $T$ transverse to $u$ at $u(\zeta)$, for $\zeta$ in some finite subset  $\zz\subset \pa D$. 
Explicitly, we take $T= T_1\cup T_2\cup T_3$ to be three small disjoint $(n-1)$-disks intersecting $u|_{\pa D}$ transversely at $u(\zeta_1')$, $u(\zeta_2')$, and $u(\zeta_3')$. Write $u_{T}=u\circ\phi$ where $\phi\colon D\to D$ is the unique conformal map such that $\phi(1)=\zeta_1'$, $\phi(i)=\zeta_2'$, and $\phi(-1)=\zeta_3'$. For simpler notation write $\zeta_1=1$, $\zeta_2=i$, and $\zeta_3=-1$.

Then there is a neighborhood $W(u_{T})\subset\cfig(\beta)$ such that any $w\in \bar\pa^{-1}(0)\cap W(u_{T})$ intersects $T_j$ transversely at a point $\zeta_j'\in\pa D$ close to $\zeta_j$, $j=1,2,3$. To see this, we observe that the norm in $\cfig(\beta)$ controls the $C^{0}$-norm which in turn controls all other norms for holomorphic functions. In particular, by shrinking $W(u_T)$ we infer that $w$ is arbitrarily $C^{1}$-close to $u$,  and hence a neighborhood as claimed exists. We obtain local coordinates on an open set of $\MM(\beta)$ around $\hat u$ by identifying the neighborhood with
\[
\bar\pa^{-1}(0)\cap \{w\in W(u_{T})\colon w(\zeta_j)\in T_j,\,j=1,2,3\}.
\]
Since all holomorphic disks in $W(u_{T})$ meet $T_j$ transversely, the action of the conformal group $G$ of $D$ is transverse to the second factor of the intersection, so  provided $\bar\pa$ is itself transverse near $\hat u$ we get a $C^{1}$-smooth chart.

We next consider coordinate changes. Thus let $T$ and $T'$ be two collections of hypersurfaces as above. On the overlap of the corresponding open subsets, the coordinate change is given by the reparametrization $w_{T}=w_{T'}\circ \phi$, where $\phi\colon D\to D$ is the unique conformal automorphism of $D$ such that $\phi(\zeta_j)=\zeta'_j$, where $\zeta'_j$ are points such that $w_T(\zeta_j')\in T_j$. To see that this gives a $C^{1}$-map we simply note that the $C^{1}$-norm of the location of the marked points is controlled by the $C^{1}$-norm of $u_T$ which in turn is controlled by the norm in $\cfig(\beta)$ since the functions are holomorphic. More specifically, let $u\colon D\to\C^{n}$ be holomorphic with $u(\{1,i,-1\})\notin T'$ and suppose $u$ intersects $T'$ transversely at the points $\xi_1,\xi_2,\xi_3$. Then there exists $\phi\in G$ such that if $\hat u=u\circ \phi$ then $u(\{1,i,-1\})\in T'$. Note that the derivative of the $\xi_j$-component of the inverse map in direction of the gauge orbit has components
\[
\frac{\pa \xi_j}{\pa \phi}=- \left\langle\nu,du_{\xi_j}\left(\frac{\pa \phi}{\pa v}\right)\right\rangle,
\] 
where $\nu$ is the normal of $T'$ at $u(\xi_j)$ and where $v$ is the vector tangent to the boundary at $\xi$.

In conclusion we find that transition functions are smooth; picking a finite cover of the compact Hausdorff space $\MM(\beta)$ we get a $C^{1}$-smooth structure. Note that the argument used to show that transition functions are smooth, in combination with the existence of a common refinement of any two finite covers, shows that the $C^{1}$-structure is independent of the cover. 

\begin{Lemma}\label{Lem:tvholdisks}
For a generic almost complex structure $J\in\JJ_{L}$ on $\C^{n}$,
the moduli spaces $\MM(j\beta)$ are empty for $j<0$, $\MM(0\beta)$ consists of constant maps, and $\MM(\beta)$ is a compact $C^{1}$-manifold of dimension $2n-3$.
\end{Lemma}

\begin{proof}
The first and second statements hold since the symplectic area of a curve in class $j\beta$ is negative when $j<0$  and zero if $j=0$. For $C^{1}$-smoothness of $\MM(\beta)$, given the definition of the $C^{1}$-structure above, we need only show that the linearization $D(\bar\pa)$ is generically surjective. For disks in primitive homology classes, this is due to Oh \cite{Oh:Perturb}. Finally, by Gromov-Floer compactness the space $\MM(\beta)$ is compact modulo bubbling. The homology classes of the disks in a bubble tree must sum to $\beta$, but each disk lies in homology class $j\beta$, $j>0$. Therefore the bubble tree must be trivial and the space is compact. 
\end{proof}

For the perturbed equation the smooth structure of moduli spaces is more straightforward, and we have the following result.

\begin{Lemma}\label{Lem:tvmdli}
For generic Hamiltonian and for an almost complex structure $J\in\JJ_L$ on $\C^{n}$ for which Lemma \ref{Lem:tvholdisks} holds:
\begin{enumerate}
\item The moduli space $\FF(0\beta)$ is a transversely cut out $(n+1)$-manifold with boundary. Its boundary is canonically diffeomorphic to $L$, $\pa \FF(0\beta)=\FF^{0}(0\beta)=L$, viewed as the space of constant maps into $L$.
\item The moduli space $\FF(-\beta)$ is a transversely cut out closed manifold of dimension $1$.
\item The moduli space $\FF(j\beta)=\varnothing$ for $j<-1$.
\end{enumerate}   
\end{Lemma}

\begin{proof}
The transversality properties follow from standard arguments perturbing the Hamiltonian.
It remains to show the compactness statements. It follows from Gromov-Floer compactness that $\FF(-\beta)$ is compact up to bubbling of holomorphic disks. By transversality, any bubble must lie in $\MM(j\beta)$ for $j\ge 1$, which implies that the Floer disk in the limit lies in $\FF(l\beta)$, $l<-1$; but these spaces are empty, so $\FF(-\beta)$ is compact. 
\end{proof}

It is straightforward to include a boundary marked point on solutions. First we consider a version of the moduli space $\FF(j\beta)$ with a marked point on the boundary. Write $\FF^{\ast}(j\beta)$ for this space and note that it is a product:
\[
\FF^{\ast}(j\beta)=\FF(j\beta)\times\pa D,
\] 
where the first factor encodes the map and the second the location of the marked point. In particular, the $C^{1}$-smooth structure of $\FF(j\beta)$ gives a $C^{1}$-smooth structure on $\FF^{\ast}(j\beta)$. Also, there is a natural smooth evaluation map $\ev\colon\FF^\ast(j\beta)\to L$, 
\[
\ev((u,\zeta))=u(\zeta).
\] 
Second we consider the space $\MM^{\ast}(\beta)$ which is a version of $\MM(\beta)$ with a marked point on the boundary. More precisely, let $G_{1}\subset G$ denote the subgroup consisting of conformal automorphisms of $D$ that fix $1\in\pa D$, and define
\[
\MM^{\ast}(\beta)= \FF^{0}(\beta)/G_{1}.
\]
Again there is an induced $C^{1}$-smooth structure and a natural evaluation map $\ev\colon\MM^{\ast}(\beta)\to L$,
\[
\ev(u)=u(1).
\]

The following result describes the broken disks which are added to compactify  $\FF(0\beta)$.   
\begin{Lemma}\label{Lem:bubbles}
For a generic Hamiltonian and almost complex structure $J \in \scrJ_L$ 
the boundary of $\FF(0\beta)$ in the Gromov-Floer compactification consists of two-level broken curves, with top level $(u_1,\zeta)\in\FF^{\ast}(-\beta)$, lower level $u_2\in \MM^{\ast}(\beta)$, and with $\ev(u_1,\zeta)=\ev(u_2)$.
\end{Lemma}

\begin{proof}
The non-compactness of $\FF(0\beta)$ comes from holomorphic disks bubbling off at boundary points. By Lemma \ref{Lem:tvmdli} the only possible bubble is one holomorphic disk in $\MM(\beta)$. This implies that the other component of the broken configuration is a disk $u_1\in\FF(-\beta)$. The marked point $\zeta\in\pa D$ is the point where the holomorphic disk is attached, and by our definition of marked points for holomorphic disks, the domain $D$ of the holomorphic bubble disk is attached at $1\in\pa D$ to the domain of $u_1$ at $\zeta$.
\end{proof}

\begin{Remark}
As we shall see later, the Gromov-Floer boundary is diffeomorphic to the (generically) transverse fibered product
\[
\FF^{\ast}(-\beta)\times_{L} \MM^{\ast}(\beta)=(\ev\times\ev)^{-1}(\Delta_L),
\]  
where $\Delta_L\subset L\times L$ is the diagonal.
This identification follows from a gluing result which uniquely constructs all unbroken Floer disks near any broken configuration. The identification can be studied with various degrees of accuracy: one can view it as a $C^{0}$-statement and obtain the compactification as a topological space, or one can carry out the gluing more carefully to  identify the compactification as a $C^{1}$-manifold with boundary. Schemes for carrying $C^{1}$-data in gluing  problems were worked out in \cite{FO3} and \cite{HWZ} in more complicated situations. In the terminology of \cite{HWZ}, the gluing problems studied in the present paper can be phrased in the language of M-polyfolds. We will present a $C^{1}$-identification in Section \ref{Sec:gluing2} below.
\end{Remark}

\subsection{Configuration spaces for disks with punctures and jet-conditions}
\label{sec:punctures}
In this section we present a functional analytic setting for describing disks with one boundary puncture, which we will use in our description of the boundary of the space of Floer disks $\FF(0\beta)$. In fact, our description uses jet-conditions at the puncture and we present a set up for higher jet conditions that does not require imposing higher regularity conditions on the maps in  the configuration space. Many constructions are similar to those in Section \ref{sec:closeddisks} and some of the details from there will not be repeated here. 

Let $\zeta\in\pa D$. In Section \ref{sec:cohertriv} we identified $(D,\zeta)$ with the upper half plane $(H,\infty)$, constructed a half strip neighborhood of $\infty$ in $H$ and introduced a weight function $\mathbf{e}_{\delta}$ depending on a real parameter $\delta$ and given by $e^{\delta|\tau|}$ for $\tau+it$ in the half strip neighborhood $[0,\infty)\times[0,1]$ 
around $\infty$.  For $s>0$, we write $\sblv^{s}_\delta(D-\zeta,\C^{n})$
for the weighted Sobolev space of $\C^{n}$-valued functions with $s$ derivatives in $L^{2}$ weighted by $\mathbf{e}_\delta$, with the weighted Sobolev $s$-norm.  More specifically, we consider the closure, in the weighted $s$-norm on $D-\zeta\approx H$ with respect to the cylindrical end metric $h$ weighted by $\mathbf{e}_{\delta}$, of functions that extend smoothly to the double of $D-\zeta$.

In parallel with Section \ref{sec:closeddisks}, we will also use the Sobolev spaces $\sblv^{s}_{\delta}(\pa D-\zeta,\C^{n})$  
of functions on the boundary with respect to the metric induced by $h$ and weight induced by $\mathbf{e}_{\delta}$. Again, if $s>\frac12$, there is a continuous restriction map
\[
\sblv^{s}_{\delta}(D-\zeta,\C^{n})\to\sblv^{s-\frac12}_{\delta}(\pa D-\zeta,\C^{n}),\quad u\mapsto u|_{\pa D-\zeta}.
\]
Furthermore, we use the spaces 
$\sblv^{s}_{\delta}(D-\zeta,\Hom(TD,\C^{n}))$, 
$\sblv^{s}_{\delta}\left(D-\zeta,\Hom^{0,1}(TD,\C^{n})\right)$, and
$\dot\sblv^{1}_{\delta}\left(D-\zeta,\Hom^{0,1}(TD,\C^{n})\right)$,
which are defined in analogy with the corresponding spaces in Section \ref{sec:closeddisks}, using the metric $h$ and weight $\mathbf{e}_{\delta}$.

We next construct Banach manifolds that we use to build configuration spaces for punctured Floer disks. Fix $\delta>0$ and let $q\in\C^{n}$. Write $\sblv_{\delta}^{s}(D-\zeta,\C^{n};q)$
for the affine Sobolev space modeled on the vector space $\sblv_{\delta}^{s}(D-\zeta,\C^{n})$ with origin shifted to $q$. More formally, let $q$ denote the constant function on $D-\zeta$ with value $q\in\C^{n}$. Then elements  $u\in\sblv_{\delta}^{s}(D-\zeta,\C^{n};q)$ are functions of the form $u=u'+q$ where $u'\in\sblv_{\delta}^{s}(D-\zeta,\C^{n})$, and the distance between $u_1$ and $u_2$ is $\|u_1-u_2\|_{s;\delta}$, where $\|\cdot\|_{s;\delta}$ is the norm in $\sblv^{s}_{\delta}(D-\zeta,\C^{n})$.

Let $q\in L$ and consider the subspace $\cfig_{\delta}(\zeta,q)\subset\sblv^{2}_{\delta}(D-\zeta,\C^{n};q)$ of all $u$ which satisfy the following conditions:
\begin{enumerate}
\item $u|_{\pa D-\zeta}(z)\in L\subset\C^{n}$ for all $z\in\pa D-\zeta$.
\item $(du)^{0,1}|_{\pa D-\zeta}=0\in\sblv^{\frac{1}{2}}_{\delta}(\pa D-\zeta,\Hom^{0,1}(TD,\C^{n}))$. 
\end{enumerate}

\begin{Lemma}
The subset $\cfig_{\delta}(\zeta,q)\subset\sblv^{2}_{\delta}(D-\zeta,\C^{n};q)$ is closed and is a Banach submanifold. The tangent space $T_{u}\cfig_{\delta}(\zeta,q)$ at $u$ is the subspace of vector fields  $v\in\sblv^{2}_{\delta}(D-\zeta,\C^{n})$ such that 
\begin{enumerate}
\item $v(z)\in T_{u(z)}L$ for all $z\in\pa D-\zeta$ and
\item $(\nabla v)^{0,1}|_{\pa D-\zeta}=0\in\sblv^{\frac{1}{2}}(\pa D-\zeta,\Hom^{0,1}(TD,\C^{n}))$,
\end{enumerate}
where $\nabla$ denotes the Levi-Civita connection of the metric $\hat g$, see Section \ref{sec:Geomprel}. Furthermore, the map $\Exp\colon T_{u}\cfig_{\delta}(\zeta,q)\to\cfig_{\delta}(\zeta,q)$,
\[
[\Exp(v)](z)=\exp_{u(z)}(v(z)),\quad z\in D-\zeta,
\]
gives local $C^{1}$-coordinates near $u$ when restricted to $v$ in a sufficiently small ball around $0\in T_{u}\cfig_{\delta}(\zeta,q)$. 
\end{Lemma} 

\begin{proof}
As with Lemma \ref{Lem:expmapclosed}, this is a simpler version of \cite[Lemma 3.2]{EES}.
\end{proof}

The connected components of $\cfig_{\delta}(\zeta,q)$ are in $1-1$ correspondence with $\pi_2(\C^{n},L)\cong H_1(L)$, and for $j\in\Z$ we write $\cfig_{\delta}(j\beta;\zeta,q)$ for the connected component of all $u\in\cfig_{\delta}(\zeta,q)$ such that the continuous loop given by the arc $u|_{\pa D-\zeta}$ completed with the point $q$ represents the class $j\beta$.

Our basic configuration spaces for punctured Floer disks are built by patching together functions in the Banach manifolds $\cfig_{\delta}(\zeta,q)$. To this end we first describe certain cut-off functions. Consider $q\in L$ and the neighborhood $U_q$ of $q$, see \eqref{Eq:disk^2nbhd} and use notation as there. Let $a\colon [0,r]\to[0,1]$ be a smooth decreasing function such that $a$ equals $1$ on $[0,\frac{r}{3}]$ and equals $0$ outside $[0,\frac{2r}{3}]$. Let $b\colon [-r,r]\to[-\frac{r^{2}}{100n},\frac{r^{2}}{100n}]$ be a function with $b(0)=0$ and $b'(0)=1$, and which equals $0$ outside $[-\frac{r}{100},\frac{r}{100}]$. Use coordinates $z=x+iy$ on $\C^n$, $x,y\in\R^{n}$. Let $\hat a(x)=a(|x|)$, and define the the cut off function $\alpha_q\colon U_q\to\C$ as follows
\begin{equation}\label{Eq:cut-off}
\alpha_{q}(x+iy)=\hat a(x)\hat a(y)+i\sum_{j=1}^{n}\frac{\pa \hat a}{\pa x_j}b(y_j).
\end{equation}
The role of the second summand and in particular the function $b$ is to adjust the standard cut off function corresponding to the first summand so that it becomes holomorphic along $L$,  here  corresponding to $y_j=0$, $j=1,\dots,n$, without altering its values along $L$ and over all altering its values only by a small amount.

\begin{Lemma} \label{Lem:alphaq}
The function $\alpha_q$ has the following properties:
\begin{enumerate}
\item $\alpha_q$ equals $1$ on $B_{\R^{n}}(\frac{r}{3})\times iB_{\R^{n}}(\frac{r}{3})$.
\item $\alpha_q$ equals $0$ outside $B_{\R^{n}}(\frac{2r}{3})\times iB_{\R^{n}}(\frac{2r}{3})$.
\item $\alpha_q$ is real-valued on $U_q\cap L$
\item $\alpha_q$ is holomorphic, $d\alpha+i\circ d\alpha \circ J=0$, along $L\cap U_q$.  
\end{enumerate}
\end{Lemma}

\begin{proof}
Straightforward.
\end{proof}

By $(2)$ we extend $\alpha_q$ smoothly by $0$ to all of $\C^{n}$. It then follows that it has properties $(3)$ and $(4)$ with $U_q$ replaced by $\C^{n}$. 

We next define a family of diffeomorphisms parameterized by $L\times B_{\R^{n}}(r_0)$. Recall the local biholomorphic maps $\psi_{q}$, $q\in L$, see \eqref{Eq:disk^2nbhd}. For $z\in U_q$ write $w=\psi_q^{-1}(z)$ and define $\Phi_{q}[c_0]\colon (\C^{n},L)\to(\C^{n},L)$ as follows:
\begin{equation}\label{Eq:diffeoatp}
\Phi_{q}[c_0](z)=
\begin{cases}
z &\text{for } z\notin U_q,\\
\psi_q\left(
w+\alpha_{q}(w)\cdot c_0
\right)&\text{for } z\in U_q.
\end{cases}
\end{equation}
Then $\Phi_q[c_0]$ is a diffeomorphism for $r_0>0$ sufficiently small and it follows from Lemma \ref{Lem:alphaq} that $\Phi_q[c_0]$ is holomorphic along $L$.  Now fix $\delta>0$ and consider the bundle
\[
\cfig_{\delta}(\pa D,L)\to \pa D\times L,
\] 
the fiber of which at $(\zeta,q)\in \pa D\times L$ equals $\cfig_{\delta}(\zeta,q)$. As we shall see, this bundle is a $C^{1}$-smooth locally trivial bundle and in particular the total space of the bundle is a Banach manifold. 

We construct local trivializations using the diffeomorphisms of \eqref{Eq:diffeoatp} and use notation as there. Let $V_p=\psi_p(B_{\R^{n}}(r_0))\subset L$ and let $A$ be an arc around $\zeta\in\pa D$. Define the trivialization 
\[
\phi\colon A\times V_p\times\cfig_{\delta}(\zeta,q)\to \cfig_{\delta}(\pa D,L)
\]
as follows:
\begin{equation}\label{Eq:trivviadiffeo}
\phi(\zeta',q', u)(z) = \Phi_{q}[\psi_q^{-1}(q')]\left(u\left(\tfrac{\zeta}{\zeta'} z\right)\right).
\end{equation}
It is straightforward to verify that coordinate changes are $C^{1}$-smooth. 
We will sometimes also consider $\cfig_{\delta}(\pa D, L)$ as a bundle over $\pa D$, composing the bundle projection to $\pa D\times L$ with the projection $\pa D \times L \rightarrow \pa D$, and we write
\[
\cfig_{\delta}(L)=\cfig_{\delta}(\pa D, L)|_{1\in\pa D}.
\]
Furthermore, we will consider the sub-bundle $\cfig_{\delta}(j\beta;\pa D, L)$ with fiber over $(\zeta,q)$ equal to $\cfig_{\delta}(j\beta;\zeta,q)$ and the corresponding sub-bundle $\cfig_{\delta}(j\beta;L)$ of $\cfig_{\delta}(L)$.  Similarly, we let
\[
\tcfig_{\delta}(\zeta) = \dot\sblv^{1}_{\delta}(D-\zeta, \Hom^{0,1}(TD,\C^{n})),
\]
and define $\tcfig_{\delta}(\pa D)$ as the bundle over $\pa D$ with fiber over $\zeta$ equal to $\tcfig_{\delta}(\zeta)$ and write $\tcfig_{\delta}=\tcfig_{\delta}(1)$. We define $C^{1}$-smooth local trivializations using re-parametrization as above.

We next construct configuration spaces for disks with jet conditions at the boundary.
Let $J^{m}(\pa D,L)$ denote the space of $m$-jets of maps $\pa D\to L$. This map fibers over the $0$-jet space $\pa D\times L$ with fiber $J^{m}_{(\zeta,q)}(L)$ over $(\zeta,q)\in\pa D\times L$ equal to $(\R^{n})^{m}$, with components corresponding to the coefficients of the Taylor expansion. Thus, in local coordinates $\R$ on $\pa D$ around $\zeta$ and $\R^{n}$ on $L$ around $q$, the element $\cc=(c_1,\dots,c_m)$ represents the $m$-germ of the map
\[
t\mapsto c_1t+c_2 t^{2}+\dots+ c_mt^{m}.
\]  
We write $J^{m}(L)$ for the fiber of $J^{m}(\pa D,L)$ over $1\in\pa D$, so that $J^{m}(L)$ fibers over $L$ with fiber $(\R^{n})^{m}$. Let $\pi\colon J^{m}(\pa D,L)\to \pa D\times L$ denote the projection and fix $\eta$ with $0<\eta<\pi$. The configuration spaces for punctured disks that we will use are the following:
\begin{equation}
\cfig_{m\pi+\eta}(J^{m}(\pa D,L))=\pi^{\ast}\cfig_{m\pi+\eta}(\pa D, L)
\end{equation}
and
\begin{equation}
\cfig_{m\pi+\eta}(J^{m}(L))=\cfig_{m\pi+\eta}(J^{m}(\pa D,L))|_{1\in\pa D}.
\end{equation}

In order to use $\cfig_{m\pi+\eta}(J^{m}(\pa D,L))$ as configuration spaces for Floer type equations we will associate Sobolev functions to its elements. To this end we fix a smooth family of smooth maps parameterized by $\pa D\times J^{m}(L)$. More precisely, for $(\zeta,q)\in \pa D\times L$ and $\cc\in J^{m}_{(\zeta,q)}(L)$ we define a map $v_{q}[\zeta,\cc]\colon D-\zeta\to U_q$ with the following properties:
\begin{enumerate}
\item $v_q[\zeta,\cc](\pa D-\zeta)\subset L$,
\item $v_q[\zeta,\cc](\pa D-\zeta)$ is holomorphic on the boundary, and
\item there is $\rho>0$ such that in the strip neighborhood $[\rho,\infty)\times[0,1]$ of the puncture we have (in $U_q$-coordinates)
\begin{equation}
v_{q}[\zeta,\cc](z)=\sum_{j=1}^{m} c_j e^{-j\pi z}.
\end{equation}
\end{enumerate}
It is clear that there is a smooth family of functions with these properties.  We then define
\begin{equation}\label{Eq:weight+Taylor}
\Psi^{m}\colon \cfig_{m\pi+\eta}(J^{m}(\pa D,L))\to\cfig_{\eta}(\pa D,L)
\end{equation}
fiber wise as follows:
\[
\Psi^{m}[u](z)= u(z) + \alpha_q(u(z))\cdot v_q[\zeta;\cc](z),
\]
where the addition refers to standard addition in the $\C^{n}$-coordinates of $U_q$ (note that $\alpha_q$ is supported well inside $U_q$).

We consider the Floer equation for punctured disks in the setup described above. Fix $0<\eta<\pi$ and an integer $m\ge 0$. We start with the case of bundles over the $0$-jet space. Let $r\in[0,\infty)$ and consider the $\pa D$-bundle map $\bar\pa_{\Fl}\colon \cfig_{m\pi+\eta}(\pa D,L)\to\tcfig_{m\pi+\eta}(\pa D)$,
\[
\bar\pa_{\Fl}(u)=(du + \gamma_r\otimes X_H)^{0,1},
\]
and write $\bar\pa_{\Fl;j\beta}=\bar\pa_{\Fl}|_{\cfig_{m\pi+\eta}(j\beta;\pa D,L)}$. As in Section \ref{sec:closeddisks}, in cases when the Hamiltonian term vanishes we write $\bar\pa$ instead of $\bar\pa_{\Fl}$ and we take the puncture to be fixed at $1\in\pa D$.
\begin{Lemma}\label{Lem:Floermapindexpnctrs}
The map $\bar\pa_{\Fl;j\beta}\colon \cfig_{m\pi+\eta}(j\beta;\zeta,L)\to \tcfig_{m\pi+\eta}(\zeta)$ is a $C^{1}$ Fredholm map of index
\[
\ind(\bar\pa_{\Fl;j\beta})=n(1+j-m),
\]
which depends smoothly on $\zeta\in\pa D$ and $r\in[0,\infty)$.
\end{Lemma}

\begin{proof}
The proof is a modification of the proof of Lemma \ref{Lem:Floermapindexclosed}. The index of the linearized operator equals the index of the ordinary $\bar\pa$-operator with  positive exponential weight at the puncture $\zeta$ and with
Lagrangian boundary condition the loop $T_{u(z)}L$, $z\in\pa D$, which has Maslov index $jn$ if $u\in\cfig_{m\pi+\eta}(j\beta;\pa D,L)$.  The positive exponential weight lowers the index by $n\cdot (m+1)$ (in terms of spectral flow we pass the $m+1$ eigenvalues $0,\pi,\dots,m\pi$ of the asymptotic operator). There is an additional $n$-dimensional contribution to the index from the directions spanned by $L$ in the domain. We find that the total index equals $n+jn-(m+1)n+n$.
\end{proof}

Consider next the bundles over jet-spaces. Define 
\[
\bar\pa_{\Fl}^{(m)}\colon \cfig_{m\pi+\eta}(J^{m}(\pa D,L))\to \tcfig_{m\pi+\eta}(\pa D)
\]
as the composite 
\[
\bar\pa_{\Fl}^{(m)}=\bar\pa_{\Fl}\circ \Psi^{m},
\]
where $\Psi^{m}\colon\cfig_{m\pi+\eta}(J^{m}(L))\to\cfig_{\eta}(\pa D,L)$ is the map in \eqref{Eq:weight+Taylor}. Note that it follows from the definition of $\Psi^{m}$ that $\bar\pa_{\Fl}^{(m)}(u)$ lies in the target space with weights as claimed: the function added to $u$ in \eqref{Eq:weight+Taylor} is holomorphic in the strip like end. We write $\bar\pa_{\Fl;j\beta}^{(m)}$ for 
$\bar\pa_{\Fl}^{(m)}|_{\cfig_{m\pi+\eta}(J^{m}(j\beta;\pa D,L))}$. Following our usual practice, we write $\bar\pa^{(m)}$ and $\bar\pa^{(m)}_{j\beta}$, dropping the subscript $\Fl$ when the Hamiltonian term vanishes. 

\begin{Lemma}\label{Lem:Floermapjet}
The map $\bar\pa^{(m)}_{\Fl;j\beta}\colon \cfig_{m\pi+\eta}(J^{m}(\zeta,L))\to \tcfig_{m\pi+\eta}(\zeta)$ is a $C^{1}$ Fredholm map of index
\[
\ind(\bar\pa^{(m)}_{\Fl;j\beta})=n(1+j),
\]
which depends smoothly on $\zeta\in\pa D$ and $r\in[0,\infty)$.
\end{Lemma}

\begin{proof}
This is a consequence of Lemma \ref{Lem:Floermapindexpnctrs}, with the $mn$ dimensions added in the domain corresponding to the fiber-dimension of $J^{m}(\pa D,L)\to\pa D\times L$. 
\end{proof}

We next show that for any $m$, if $0$ denotes the $0$-section, the solution spaces $(\bar\pa^{(m)})^{-1}(0)$ are $C^{1}$-diffeomorphic to the moduli spaces of disks with marked points discussed in Section \ref{sec:closeddisks}. We first describe the smooth structure on the moduli space in the case of vanishing Hamiltonian. Recall that in this case we take the puncture fixed at $1\in\pa D$. The definition of the $C^{1}$-structure on $(\bar\pa^{(m)})^{-1}(0)/G_{1}$ parallels the corresponding definition for $\MM$ in Section \ref{sec:closeddisks} and we give only a brief sketch. 

We use hypersurfaces $T=T_1\cup T_2$ that intersect a representative of a class transversely to fix gauge locally. (Note that only two marked points are now needed since the puncture is fixed by $G_1$.) The moduli space has a finite cover of such charts. To check that transition functions are smooth in the topology induced from the ambient configuration space, we must control the effect of conformal automorphisms on punctured disks. We thus consider the effect of moving the location of the auxiliary marked points on a punctured disk. Identify the domain of a punctured holomorphic disk $u$ with the upper half plane $H$, with puncture  at $\infty$ that will be fixed throughout and two marked points from intersections with hypersurfaces  at $-1,1$. If $\zeta_1,\zeta_2\in \pa H$ are intersection points with hypersurfaces arising from some other chart, the point corresponding to $u$ is $u\circ\phi$, where $\phi\colon H\to H$ is the unique automorphism with $\phi(1)=\zeta_1$ and $\phi(-1)=\zeta_2$. Any such $\phi$ is a composition of a translation, keeping the distance between the marked points fixed, and a scaling, keeping one of the marked points fixed and moving the other. 

A translation by a real constant $c$ in the upper half plane gives the coordinate change $\phi_c(w')=w+c$ on $H$. In the cylindrical ends, setting $w=e^{\pi z}$ and $w'=e^{\pi z'}$, the corresponding change of coordinates is
\begin{equation}\label{Eq:movemarked1}
z'=z+\frac{1}{\pi}\sum_{n>0} (-1)^{n+1}n^{-1} c^{n}e^{-n\pi z},
\end{equation}
which induces a $C^{1}$-diffeomorphism of the charts. To move only one of the punctures, we may assume that the other one is at $0\in\pa H$ via the above, and then scale by a real positive constant $\sigma_r(w')=rw$. The corresponding coordinate change in the cylindrical end is then given by
\begin{equation}\label{Eq:movemarked2}
z'=z+\log r,
\end{equation}
which again induces a $C^{1}$-diffeomorphism.  After these preliminaries we state the result that relates disks with and without punctures.
\begin{Lemma}\label{Lem:markedandpctrs}
For a generic non-trivial family of Hamiltonians $H_R$ we have, for $j\in\Z$,
\[
(\bar\pa_{\Fl;j\beta}^{(m)})^{-1}(0) \approx_{C^{1}} \FF^{\ast}(j\beta).
\]
For trivial Hamiltonian we have, for $j\in\Z$:
\[
(\bar\pa_{j\beta}^{(m)})^{-1}(0)/G_1\approx_{C^{1}}\MM^{\ast}(j\beta).
\]
\end{Lemma}

\begin{proof}
The operator $F$ agrees with the standard $\bar\pa$-operator in a neighborhood of the boundary. Using $U_{u(\zeta)}$-coordinates, $\zeta\in\pa D$, we have solutions near $\zeta$ in the closed disk model (which we think of as the half plane around $0$) given by their Taylor expansion
\[
u(z)=\sum_{k\ge 0} c_k z^{k}.
\]
The corresponding Fourier expansion for a strip like end, $z=e^{-\pi w}$, is
\[
u(w)=c_0+\sum_{k>0} c_k e^{-\pi k w} ,
\]
where $c_k\in\R^{n}$, $k\ge 0$.  The Taylor coefficients of a holomorphic function are controlled by and control the $\|\cdot\|_2$-norm; similarly, the Fourier coefficients $c_j$, $j>0$, are controlled by the $\|\cdot\|_{2;\delta}$-norm, whilst $c_0$ is controlled by the topology of $L$. It follows that the natural map is $C^{1}$ with respect to the topologies induced from the ambient spaces. 
\end{proof}

\section{The Floer-Gromov boundary of $\FF(0\beta)$}
\label{Sec:theboundary}
It follows from Lemma \ref{Lem:bubbles} that the Floer-Gromov boundary of $\FF(0\beta)$ is the fibered  product
\[
\NN=\FF^{\ast}(-\beta)\times_L \MM^{\ast}(\beta).
\]
In this section we show that $\NN$ is an orientable $C^{1}$-manifold for generic data, and then we consider gauge fixing of the second factor in broken disks that arise as limits of sequences of smooth disks in $\FF(0\beta)$.

\subsection{Transversality for fibered products}
\label{sec:FF(-b)}
By Lemma \ref{Lem:tvmdli}, $\FF(-\beta)$ is an orientable closed $1$-manifold and consequently
\[
\FF^{\ast}(-\beta)=\FF(-\beta)\times\pa D
\] 
is a union of tori. As mentioned above there is a natural evaluation map
\[
\ev\colon\FF^{\ast}(-\beta)\to L,\quad \ev(u,\zeta)=u(\zeta).
\]
As explained in Section \ref{ssec:capping} we will use a filling of $\FF^{\ast}(-\beta)$. Thus let $\DD$ be a collection of solid tori such that $\pa\DD=\FF^{\ast}(-\beta)$ and let $\ev\colon \DD\to L$ be an extension of $\ev\colon\FF^{\ast}(-\beta)\to L$ that is constant in a small collar neighborhood of the boundary.    

\begin{Lemma}\label{Lem:fiberprod1}
For a generic family of Hamiltonians and almost complex structure such that Lemma \ref{Lem:tvmdli} holds, the map
\[
\ev\times\ev\colon\FF^{\ast}(-\beta)\times\MM^{\ast}(\beta)\to L\times L
\]
is transverse to $\Delta_L$ and consequently $\NN$ is a $C^{1}$-smooth orientable manifold. 
\end{Lemma}

\begin{proof}
Standard arguments show that for generic $H$ and $J$ the map
\[
\mathbf{F}\colon\cfig_{\delta}(-\beta;\pa D, L)\times\cfig_{\delta}(\beta; L)\to
\tcfig_{\delta}(\pa D)\times \tcfig_{\delta}\times L\times L,
\]
given by
\[
\mathbf{F}(u_1,u_2)=\left(\bar\pa_{\Fl;-\beta}(u_1),\bar\pa_\beta(u_2),\pr_L(u_1),\pr_L(u_2)\right),
\]
where $\pr_L$ denote the bundle projection into $L$, is a Fredholm map of index $2 + 2n-2n=2$ which is transverse to the product of the $0$-sections and the diagonal. The fibered product is then $\NN=\mathbf{F}^{-1}(\Delta_L)/G_1$, where $G_1$ acts on the second factor, and is an orientable $n$-manifold. 
\end{proof}

\begin{Lemma}\label{Lem:fiberprod2}
For a generic family of Hamiltonians and almost complex structure such that Lemmas \ref{Lem:tvmdli} and \ref{Lem:fiberprod1} hold, and for a generic extension $\ev\colon\DD\to L$, the map
\[
\ev\times\ev\colon\DD\times\MM^{\ast}(\beta)\to L\times L
\]
is transverse to $\Delta_L$ and consequently $\TT$ is a smooth orientable manifold with $\pa \TT=\NN$. 
\end{Lemma}

\begin{proof}
Standard arguments show that for generic $\DD$, after small perturbation of $H$ and $J$, the map
\[
\mathbf{F}\colon\DD\times\cfig_{\delta}(\beta; L)\to
\tcfig_{\delta}(\pa D)\times L\times L,
\]
given by
\[
\mathbf{F}(u_1,u_2)=\left(\bar\pa_\beta(u_2),\ev(\delta),\ev(u_2)\right),
\]
is a Fredholm map of index $3+2n-2n=3$ which is transverse to the product of the $0$-section and the diagonal. The fibered product is then $\TT=\mathbf{F}^{-1}(\Delta_L)/G_1$, where $G_1$ acts on the second factor, and is an orientable $(n+1)$-manifold with boundary equal to the fibered product over $\pa\DD$, which is $\NN$.
\end{proof}

\begin{Remark}
The $C^{1}$-structures of $\NN$ and $\TT$ are defined through finite open covers, in analogy with definition of the $C^{1}$-structure of $\MM^{\ast}(\beta)$, to take account of the quotient in the second factor by the group $G_1$ of conformal automorphisms.
\end{Remark}

In order to find a $C^{1}$-collar neighborhood of the Gromov-Floer boundary of $\FF(0\beta)$ we will require more detailed information on $\NN$. We write elements $\xi\in\NN$ as $\xi=((u_1,\zeta),u_2)$, where $(u_1,\zeta)\in\FF^{\ast}(-\beta)=\FF(-\beta)\times\pa D$ and $u_2\in\MM^{\ast}(\beta)$ and where $\ev(u_1,\zeta)=\ev(u_2)$.  Recall that $\ev(u_2)=u_2(1)$ and define
\begin{equation}\label{Eq:1stder=0}
\NN_{0}=\{\xi\in\NN\colon du_2(1)=0\}.
\end{equation}  
Below we will write $\pr_2\colon \FF^{\ast}(-\beta)\times\MM^{\ast}(\beta) \rightarrow \MM^{\ast}(\beta)$ for the projection to the second factor. Recall that $\NN_{0}\subset\NN\subset\FF^{\ast}(-\beta)\times\MM^{\ast}(\beta)$. 

\begin{Lemma}\label{Lem:singNN}
For a generic Hamiltonian and almost complex structure such that Lemmas \ref{Lem:tvmdli} and \ref{Lem:fiberprod1} hold, $\NN_0$ is a transversely cut out $0$-manifold and $\pr_{2}\colon\NN_{0}\to \MM^{\ast}(\beta)$ is an embedding. Furthermore, for any $\xi=((u_1,\zeta),u_2)\in\NN_0$, the second derivative of $u_2$ at the marked point satisfies $d^{(2)}u_2(1)\ne 0$. 
\end{Lemma}

\begin{proof}
Consider the map
\[
\mathbf{G}\colon \cfig_{\delta}(-\beta;\pa D,L)\times \cfig_{2\pi+\delta}(\beta;J^{2}(L))
\ \longrightarrow \
\tcfig_{\delta}(\pa D)\times\tcfig_{2\pi+\delta}\times L\times J^{2}(L),  
\]
where
\[
\mathbf{G}\left((u_1,\zeta),(u_2,\cc_{2})\right)=
\left(\bar\pa_{\Fl;-\beta}(u_1),\bar\pa_{\beta}(u_2,\cc_2),\ev(u_1,\zeta),\pr_{J^{2}(L)}(u_{2},\cc_2)\right).
\]
This is a Fredholm map of index $2+2n-4n=2-2n$. Consider the preimage of the product of $0$-sections and the subset of $L\times J^{2}(L)$ over $\Delta_L$ with first derivative component in $J^{2}(L)$ equal to $0$. Generically the solution space has dimension $2-2n+2n=2$. The $2$-dimensional automorphism group $G_1$ acts on the holomorphic component and it follows that $\NN_{0}$ is a transversely cut out $0$-manifold. For the statement on the second derivative we consider instead the inverse image of the product of $0$-sections and the subset of $L\times J^{2}(L)$ over $\Delta_L$ with both first and second derivative equal to $0$. Here the formal dimension equals $2-n$ and we conclude that the locus is generically empty, as claimed. A similar transversality argument shows that $\pr_{2}|_{\NN_{0}}$ is an embedding.
\end{proof}

Similarly, define
\begin{equation}\label{Eq:1stder=0'}
\TT_{0}=\{\xi\in\TT\colon du_2(1)=0\}
\end{equation}  
The corresponding result in this situation is the following:

\begin{Lemma}\label{Lem:singTT}
For a generic Hamiltonian and almost complex structure such that Lemmas \ref{Lem:tvmdli} and \ref{Lem:fiberprod2} hold, $\TT_0$ is a transversely cut out $1$-manifold. Furthermore, for any $\xi=((u_1,\zeta),u_2)\in\TT_0$, the second derivative of $u_2$ at the marked point satisfies $d^{(2)}u_2(1)\ne 0$. 
\end{Lemma}

\begin{proof}
Straightforward modification of the proof of Lemma \ref{Lem:singNN}.
\end{proof}

\subsection{Gauge fixing} 
\label{sec:gauge}
We describe normal forms for elements in $\NN$ near the point of intersection of the two constituent disks. Consider a pair $((u_1,\zeta),u_2)\in\FF^{\ast}(-\beta)\times\MM^{\ast}(\beta)$ that contributes to $\NN$, i.e.~such that $u_1(\zeta)=u_2(1)$.  In holomorphic coordinates $(\C^{n},\R^{n})$ centered at $u_1(\zeta)\in L\subset\C^{n}$, the disk $u_2$ has a Fourier expansion of the following form, for $z\in[\rho,\infty)\times[0,1]$:  
\begin{equation}\label{Eq:normalform}
u_{2}(z)=\sum_{n>0} c_n e^{-n\pi z},
\end{equation}
where $c_n\in\R^{n}$. If $((u_1,\zeta),u_2)\in\NN_0$ then $c_1=0$ and $c_2\ne 0$ and if $((u_1,\zeta),u_2)$ lies outside a fixed neighborhood of $\NN_0$ then $|c_1|>\delta>0$ for some $\delta$. 

Consider small nested balls $B(u_1(\zeta);\epsilon)\subset B(u_1(\zeta);2\epsilon)$ around $u_1(\zeta)$ in $L$ and a holomorphic disk $u_2\in \MM^{\ast}(\beta)$ represented by $\hat u_2\colon (H,\pa H)\to(\C^{n},L)$ with $\hat u_2(\infty)=u_1(\zeta)$. Then for any parameterization $\hat w_2\colon (H,\pa H)\to(\C^{n},L)$ of a holomorphic disk $w_2\in\MM^{\ast}(\beta)$ in a small neighborhood of $u_2$, let $I_{2\epsilon}\subset\pa H$ denote the interval around the puncture $\infty$ which is the connected component of $\hat w_2^{-1}(B(2\epsilon;u_1(\zeta)))\cap \pa H$ that contains $\infty$. 
We consider points in $I_{2\epsilon}$ which map to $\pa B(\epsilon;u_1(\zeta))$ and we call them $(\epsilon,(u_1,\zeta))$-points, where $(u_1,\zeta)\in\FF^{\ast}(-\beta)$. Fix a metric on the closed manifold $\MM^{\ast}(\beta)$ and write $V(u_2;\delta)$ for the $\delta$-neighborhood of $u_2\in\MM^{\ast}(\beta)$. 

Lemma \ref{Lem:singNN} implies that $\NN_0$ is a compact $0$-manifold such that the restriction $\pr_{2}|_{\NN_{0}}$ is an embedding into $\MM^{\ast}(\beta)$. In particular, $\pr_2(\NN_0)$ is a finite subset. We write 
\[
\pr_2(\NN_0)=\{u_2^1,\dots u_2^{m}\}  
\]
and for $\delta>0$ we let
\[
V(\delta)=\bigcup_{j=1}^{m} V(u_2^{j};\delta).
\]

\begin{Lemma}\label{Lem:interswnbhd}
For $\epsilon>0$ sufficiently small, there exists $\delta_1>\delta_0>0$ such that $V(u_2^{j};\delta_1)$, $j=1,\dots,m$ are mutually disjoint and such that the following hold for all $(u_1,\zeta)\in\FF^{\ast}(-\beta)$:
\begin{enumerate}
\item If $u_2\notin V(\delta_1)$ and $\ev(u_2)=\ev(u_1,\zeta)$ then $u_2$ has exactly 2 $(\epsilon,(u_1,\zeta))$-points in $I_{2\epsilon}$, one in each component of $I_{2\epsilon}-\{\infty\}$, and the corresponding intersections with $\pa B(u_1(\zeta);\epsilon)$ are transverse.
\item If $u_2\in V(\delta_1)-V(\delta_0)$ and $\ev(u_2)=\ev(u_1,\zeta)$, then there are 4 $(\epsilon,(u_1,\zeta))$-points in $I_{2\epsilon}$, one in one of the components of $I_{2\epsilon}-\{\infty\}$ and three in the other, and the corresponding intersections with $\pa B(u_1(\zeta);\epsilon)$ are transverse.
\item If $u_2\in V(\delta_0)$ for some $j=1,\dots,m$ and $\ev(u_2)=\ev(u_1,\zeta)$, then there are 2, 3, or 4 $(\epsilon,(u_1,\zeta))$-points in $I_{2\epsilon}$, one in one of the components of $I_{2\epsilon}-\{\infty\}$ and all others in the other component, and the two intersection points with $\pa B(u_1(\zeta);\epsilon)$ most distant from $\infty$ are transverse.
\end{enumerate}
Note also $\delta_1\to 0$ as $\epsilon\to 0$.
\end{Lemma}

\begin{proof}
This follows immediately from \eqref{Eq:normalform}.  The three cases are indicated schematically in Figure \ref{Fig:IntervalChoices}.
\end{proof}

\begin{center}
\begin{figure}[ht]
\includegraphics[scale=0.5]{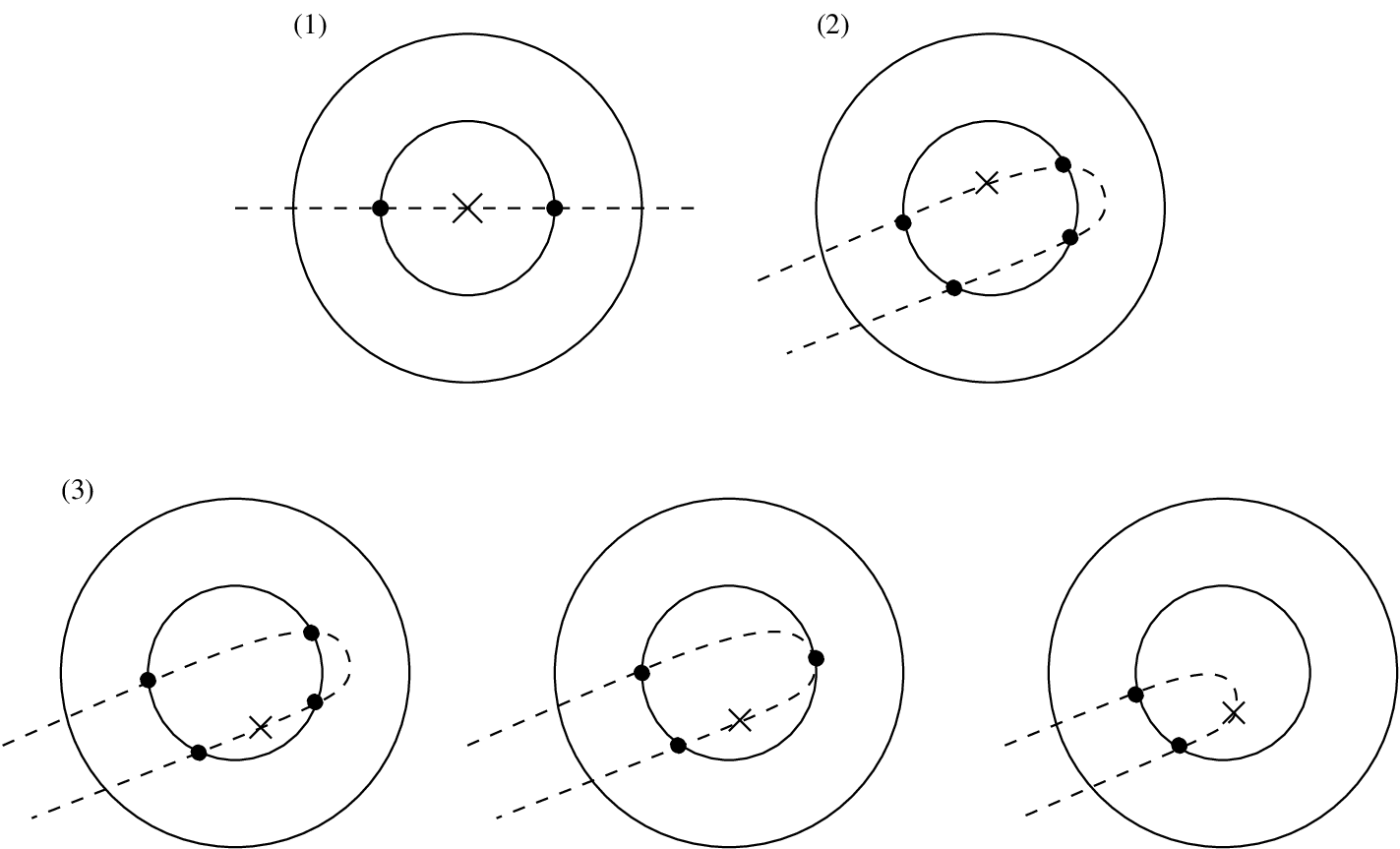}
\caption{The top line depicts cases $(1)$ and $(2)$ of Lemma \ref{Lem:interswnbhd} and the lower line the three possibilities in case $(3)$. The concentric circles represent $\pa B(u_1(\zeta),\epsilon)$ and $\pa B(u_1(\zeta),2\epsilon)$, $(\epsilon,(u_1,\zeta))$-points of $I_{2\epsilon}$ are black dots, and $\infty$ is the cross.}\label{Fig:IntervalChoices}
\end{figure}
\end{center}

Fix $0<\delta_0<\delta_1$ such that Lemma \ref{Lem:interswnbhd} holds, and consider the following two open sets of $\NN$:
\begin{align}
U^{\rm reg} &=\left\{\xi=((u_1,\zeta),u_2)\in\NN\colon u_2\notin V(\delta_0)\right\},\\
U^{\rm sing} &=\left\{\xi=((u_1,\zeta),u_2)\in\NN\colon u_2\in V(\delta_1)\right\}.
\end{align} 
Using the exponential map $\exp\colon TL\to L$ of a Riemannian metric on $L$, we may identify the disks $B(u_1(\zeta);\epsilon)$ with  $\epsilon$-disks in $T_{u_{1}(\zeta)}L$. In particular, we may consider the $(\epsilon,(u_1,\zeta))$-points at $\xi=((u_1,\zeta),u_2)\in\NN$ as points in the fiber over $\xi$ in the pull-back bundle
\[
\ev^{\ast} TL \ \longrightarrow \NN,
\]
where $\ev\colon\NN\to L$, $\ev((u_1,\zeta),u_2)=u_1(\zeta)=u_2(1)$. With this convention we get the following sections of this bundle:
\begin{itemize}
\item The $(\epsilon, (u_1,\zeta))$-point in the component of $I_{2\epsilon}-\{\infty\}$ to the left (right) of $\infty$ which is closest to $\infty$ gives a $C^{1}$-smooth section $q^{\rm reg}_{-}$ ($q^{\rm reg}_{+}$) of the bundle $\ev^{\ast} TL$ over $U^{\rm reg}$;
\item The $(\epsilon, (u_1,\zeta))$-point in the component of $I_{2\epsilon}-\{\infty\}$ to the left (right) of $\infty$ which is farthest from $\infty$ gives a $C^{1}$-smooth section $q^{\rm sing}_{-}$ ($q^{\rm sing}_{+}$) of the bundle $\ev^{\ast} TL$ over $U^{\rm sing}$.
\end{itemize}

We use these families to fix parametrizations for the $\MM^{\ast}(\beta)$-component of elements in $\NN$. If $\xi\in U^{\rm reg}\subset\NN$, $\xi=((u_1,\zeta),u_2)$ then we let $u_2^{\rm reg}\colon H\to\C^{n}$ denote the representative of $u_2$ such that $u_2^{\rm reg}(\pm 1)=q^{\rm reg}_{\mp}$. Similarly if $\xi\in U^{\rm sing}\subset\NN$, $\xi=((u_1,\zeta),u_2)$ then we let $u_2^{\rm sing}\colon H\to\C^{n}$ denote the representative of $u_2$ such that $u_2^{\rm sing}(\pm 1)=q^{\rm sing}_{\mp}$.

In order to define uniquely parametrizations over all of $\NN$, we interpolate between $u_{2}^{\rm reg}$ and $u_{2}^{\rm sing}$ in $U^{\rm reg}\cap U^{\rm sing}$. To this end, note that if $((u_1,\zeta),u_2)\in U^{\rm reg}\cap U^{\rm sing}$ then $u_2$ lies in  one of the disjoint annular regions $V(u_2^{j};\delta_1)- V(u_2^{j},\delta_0)$, $j=1,\dots,m$. Let $\delta(u_2)\in(\delta_0,\delta_1)$ denote the distance between $u_2$ and $u_2^{j}$; this defines a smooth function
\[
\delta\colon V(\delta_1)-V(\delta_0)\to (\delta_0,\delta_1).
\]

By Lemma \ref{Lem:interswnbhd}, if $((u_1,\zeta),u_2)\in \NN$ and $u_2\in V(\delta_1)-V(\delta_0)$ then there are exactly 4 $(\epsilon,(u_1,\zeta))$-points in $I_{2\epsilon}$, each corresponding to a transverse intersection with $\pa B(u_1(\zeta);\epsilon)$. Furthermore, one component of $I_{2\epsilon}-\{\infty\}$ contains 1 such point and the other component contains 3. Using the parameterization $u^{\rm reg}_{2}\colon H\to\C$ we think of $I_{2\epsilon}$ as an interval in $\pa H$ containing $\infty$.
\begin{enumerate}
\item In the case when the component with 1 intersection lies in the negative direction from $\infty$, we get the following intersection points in $I_{2\epsilon}$ (listed in the order induced by the orientation of $I_{2\epsilon}$), together with $\infty$:
\[
1 \ , \ \infty \ , \ -1 \ , \ \zeta_1 \ , \ \zeta_2.
\]
Here $1$ corresponds to $q^{\rm reg}_{-}=q^{\rm sing}_{-}$, $-1$ to $q^{\rm reg}_{+}$, and $\zeta_2$ to $q^{\rm sing}_{+}$.
\item In the case when the component with 1 intersection lies in the positive direction from $\infty$, we get the following intersection points in $I_{2\epsilon}$ (listed in the order induced by the orientation of $I_{2\epsilon}$), together with $\infty$:
\[
\zeta_1 \ , \ \zeta_2 \ , \ 1 \ , \ \infty \ , \ -1.
\]
Here $1$ corresponds to $q^{\rm reg}_{-}$, $\zeta_1$ to $q^{\rm sing}_{-}$, and $-1$ to $q^{\rm reg}_{+}=q^{\rm sing}_{+}$.
\end{enumerate}

Let $\alpha\colon [\delta_0,\delta_1]\to [0,1]$ be a smooth increasing function equal to $0$ near $\delta_0$ and equal to $1$ near $\delta_1$. For $\xi=((u_1,\zeta),u_2)\in U^{\rm reg}\cap U^{\rm sing}$, define the automorphism
\[
\psi[u_2]\colon H \to H
\]
as  follows: 
\begin{enumerate}
\item In case (1) above, with 1 intersection in the negative direction from $\infty$,   
$\psi[u_2]$ satisfies
\begin{equation}\label{Eq:interpol-}
\psi[u_2](1)=1 \ , \ \psi[u_2](\infty)=\infty \ , \ \psi[u_2](-1)= \alpha(\delta(u_2))(-1) +(1-\alpha(\delta(u_2)))\zeta_2. 
\end{equation}
\item In case (2) above, with 1 intersection  in the positive direction from $\infty$, 
$\psi[u_2]$ satisfies
\begin{equation}\label{Eq:interpol+}
\psi[u_2](1)=\alpha(\delta(u_2))1+(1-\alpha(\delta(u_2)))\zeta_1 \ , \ \psi[u_2](\infty)=\infty \ , \ \psi[u_2](-1)= -1. 
\end{equation}
\end{enumerate}

With the above established we turn to the main construction of this section and define a family $u^{\rm st}_{2}$ of stable maps representing the holomorphic component of any element in $\NN$. For $((u_1,\zeta),u_2)\in \NN$ and $u_2\in V(\delta_1)-V(\delta_0)$, define 
\[
u_{2}^{\rm st}\colon H\to\C^{n},\quad u_{2}^{\rm st}= u_2^{\rm reg}\circ \psi[u_2].
\]
Then $u_{2}^{\rm st}= u_{2}^{\rm reg}$ if $\delta(u_2)$ lies in a neighborhood of $\delta_1$ and $u_{2}^{\rm st}=u_{2}^{\rm sing}$ if $\delta(u_2)$ lies in a neighborhood of $\delta_0$;  we extend the family of parametrizations to all of $\NN$ by declaring that if $\xi=((u_1,\zeta),u_2)\in U^{\rm reg}- U^{\rm sing}$ then $u_2^{\rm st}= u_2^{\rm reg}$ and if $\xi=((u_1,\zeta),u_2)\in U^{\rm sing}- U^{\rm reg}$ then $u_2^{\rm st}= u_2^{\rm sing}$.

It is straightforward to compare the various parametrizations: if $u_2\in V(\delta_1)- V_1(\delta_0)$ then we have three parametrizations $u_2^{\rm st}$, $u_{2}^{\rm reg}$, and $u_2^{\rm sing}$. In the compact part of $H$, $H-E(R_0)$ for some large but fixed $R_0$, any two of these parametrizations differ by a diffeomorphism, whilst it follows from \eqref{Eq:movemarked1} and \eqref{Eq:movemarked2} that there are functions $k^{\rm reg}$ and $h^{\rm reg}$, $k^{\rm sing}$, and $h^{\rm sing}$ such that for $z\in[\rho_0,\infty)\times[0,1]\approx E(R_0)$ 
\begin{align*}
u_2^{\rm reg}(z) &=u_2^{\rm st}\left(z+k^{\rm reg}[u_2]+h^{\rm reg}[u_2](z)\right),\\
u_2^{\rm sing}(z) &=u_2^{\rm st}\left(z+k^{\rm sing}[u_2]+h^{\rm sing}[u_2](z)\right),
\end{align*} 
where the constants $k^{\rm reg}[u_2]=\Ordo(1)$ and $k^{\rm sing}[u_2]=\Ordo(1)$, i.e.~are uniformly bounded in $U^{\rm reg}\cap U^{\rm sing}$, and where similarly $h^{\rm reg}[u_2](z)=\Ordo(e^{-\pi z})$ and $h^{\rm sing}[u_2](z)=\Ordo(e^{-\pi z})$.

In complete analogy with the above we find an open cover $W^{\rm reg}\cup W^{\rm sing}$ of $\TT$ that restricts to $U^{\rm reg}\cup U^{\rm sing}$ over the boundary $\NN$. Here $W^{\rm sing}$ is a small neighborhood of $\TT_{0}$ and $W^{\rm reg}$ the complement of a slightly smaller such neighborhood. Exactly as for $U^{\rm reg}\cup U^{\rm sing}$ we use $W^{\rm reg}\cup W^{\rm sing}$ to produce stable parametrizations $u_2^{\rm st}$ for the holomorphic disk components of elements in the fibered product
\[
\DD\times_L\MM^{\ast}(\beta),
\]
that agree with the stable parametrizations defined above along the boundary $\NN=\pa \TT$.

\subsection{Marked points on disks at the boundary of $\FF(0\beta)$}
In order to control marked points under bubbling we will use auxiliary hypersurfaces that are Poincar{\'e} dual to the generator $\beta$ of $H_{1}(L)$. Consider a map $u_1\in\FF(-\beta)$. Then the $C^{1}$-map $\ev(u_1,\cdot)\colon\pa D\to L$ represents the homology class $-\beta$. Let $P$ be a hypersurface Poincar{\'e} dual to $\beta$. After small perturbation, $P$ is transverse to $\ev(u_1,\cdot)$ and the signed count of intersection points in $\ev(u_1,\cdot)^{-1}(P)$ equals $-1$. Noting that $P$ is orientable,   let $P'$ be a small shift of $P$ along a normal vector field. For sufficiently small shift,  $P'$ is also transverse to $\ev(u_1,\cdot)$.

By compactness of $\FF^{\ast}(-\beta)$ there is a finite collection of such hypersurfaces $P_{1},\dots,P_{m}$, with parallels $P_1',\dots, P_m'$, and an open cover $\FF^{\ast}(-\beta)=U_1\cup U_2\cup\dots\cup U_m$ such that if $(u_1,\zeta)\in U_j$ then $\ev(u_1,\cdot)$ is transverse to $P_j\cup P_j'$ and $u_1(\zeta)\notin P_j\cup P_j'$.

Let $\pi\colon\FF(0\beta)\to[0,\infty)$ denote the natural projection and recall that $\pi(\FF(0\beta))$ is contained in a compact subset. Consider a sequence of maps $u_j\colon (D,\pa D)\to(\C^{n},L)$, $j=1,2,\dots$ in $\FF(0\beta)$. By the Arzela-Ascoli theorem, if $|du_j|$ is uniformly bounded, then $u_{j}$ has a convergent subsequence $u_{j_k}$ with limit $u$ which solves the Floer equation $(du+\gamma_R\otimes X_{H_R})^{0,1}=0$ for $R=\lim_k\pi(u_{j_k})$. Consequently, for any $M>0$ the open set
\[
U_M=\{u\in\FF(0\beta)\colon \sup_{z\in D}|du(z)|>M\}
\]  
is a neighborhood of the Gromov-Floer boundary of $\FF(0\beta)$. 
For $M>0$ and $u\in\FF(0\beta)$, let $B^{(1)}(u,M)=|du|^{-1}([M,\infty))\subset D$ and let $\delta(u,M)=\diam(B^{(1)}(u,M))$.

\begin{Lemma}\label{Lem:diamblowup} 
If $\epsilon_M=\sup_{u\in U_M}\delta(u,M)$ then $\epsilon_M\to 0$ as $M\to\infty$.
\end{Lemma}
\begin{proof}
Assume that the statement is false.  Then there is a sequence of disks in $\FF(0\beta)$ with bubbles forming at (at least) two points. This contradicts Lemma \ref{Lem:bubbles}.
\end{proof}

It follows from Lemma \ref{Lem:diamblowup} that for any $\delta_0>0$ there exists $M>0$ such that for every $u\in U_M$, $\delta(u,M)<\delta_0$. Let $I_u\subset\pa D$ denote the smallest interval containing $\pa D\cap B^{(1)}(u,M)$ and let $\zeta_u$ denote the mid point of $I_u$.

\begin{Lemma}\label{Lem:blowuppointsinP}
For all sufficiently large $M>0$, if $u\in U_M$ then for some $P_j$, $j\in\{1,\dots,m\}$, there are points $\zeta,\zeta'\in I_u$ such that $u(\zeta)\in P_j$ and $u(\zeta')\in P_j'$.
\end{Lemma}

\begin{proof}
Assume that the lemma does not hold, then there is a sequence of maps $u_l\in U_{M_l}$, $M_l\to\infty$ as $l\to\infty$ such that $u_l(I_{u_l})\cap P_j=\varnothing$ or $u_l(I_{u_l})\cap P_j'=\varnothing$ for all $j=1,\dots,m$. But, by Gromov-Floer compactness and Lemma \ref{Lem:bubbles}, $(u_l,\zeta_{u_l})$ has a subsequence that converges to some $(\hat u_1,\zeta)\in\FF^{\ast}(-\beta)$ uniformly on compact subsets of $D-\zeta$. Now, $(\hat u_1,\zeta)\in U_j$ for some $j\in\{1,\dots,m\}$ and hence $\hat u_1|_{\pa D}$ intersects $P_j$ and $P_j'$ transversely with intersection number $-1$, and $\hat u_1(\zeta)\notin P_j\cup P_j'$. Since the intersection numbers of $u_l|_{\pa D}$ with $P_j$ and with $P_j'$ equal $0$ we find that $u_l(I_{u_l})\cap P_j\ne \varnothing$ and $u_l(I_{u_l})\cap P_j'\ne \varnothing$ for $l$ large enough and $u_l$ in the subsequence. This contradicts the assumption and the lemma follows.
\end{proof}

Let $H_{\delta}\subset H$ be a half disk around the origin $0$ parameterizing a small neighborhood in $D$ around $\zeta_u$ in such a way that $0\mapsto \zeta_{u}$. Let $u\in U_M$, fix $P_j$ as in Lemma \ref{Lem:blowuppointsinP}, and let $\zeta',\zeta$ be (ordered) points in $\pa H_{\delta}$ such that $u(\zeta)\in P_j$ and $u(\zeta')\in P'_j$ (or vice versa). Let $a=\frac{\zeta'-\zeta}{2}$ and $b=\frac{\zeta+\zeta'}{2}$; then the map $\psi\colon H\to H$, $\psi(z)=az+b$ satisfies $\psi(-1)=\zeta$ and $\psi(1)=\zeta'$. Since $\zeta_{u}$ corresponds to $0\in H_{\delta}$, Lemma \ref{Lem:diamblowup} implies that $a,b\to 0$ as $M\to\infty$. Let $\Omega_\rho=\psi^{-1}(H_{\delta})\subset H$. 

\begin{Lemma}\label{Lem:limit1}
For any sequence of maps $u_l\in U_{M_l}\subset \FF(0\beta)$ with $M_l\to\infty$, there exists $P_j$ such that the sequence $u_{l}\circ\psi\colon\Omega_\rho\to\C^{n}$ has a subsequence that converges uniformly on compact subsets to a holomorphic disk $u\colon (H,\pa H)\to(\C^{n},L)$ with $u(1)\in P_j'$ and $u(-1)\in P_j$, which represents an element in $\MM^{\ast}(\beta)$ with marked point at $\infty\in H$. 
\end{Lemma}

\begin{proof}
In case the derivative of $u_j\circ\psi$ is uniformly bounded we can extract a convergent subsequence, with limit which must be non-constant since the hypersurfaces $P_j$ and $P_j'$ are disjoint. The lemma then follows from Lemma \ref{Lem:bubbles}. 

Assume next that the derivative is not uniformly bounded. Then we have bubbling at at least one point $\xi\in\pa H$, and recover a non-constant holomorphic disk from there. Lemma \ref{Lem:bubbles} implies that this is the only bubble and hence $u_j\circ\psi$ must converge to a constant map on compact subsets of $H-\xi$. This however contradicts $P_j$ and $P_j'$ being disjoint, so we conclude that the derivative is uniformly bounded.
\end{proof}

\begin{Corollary}\label{Cor:limit2}
For each $\delta_0>0$, $\epsilon_0>0$, and $\epsilon>0$, there exists $M>0$ and $\rho_0>0$ such that for every $u\in U_M\subset\FF(0\beta)$ there exists $((u_1,\zeta),u_2)\in \NN$ such that the following holds:
\begin{enumerate}
\item The $C^{1}$-distance between $u|_{D-B(\zeta;\delta_0)}$ and $u_1|_{D-B(\zeta;\delta_0)}$ is smaller than $\epsilon_0$, where $B(\zeta;\delta_0)$ denotes a $\delta_0$-neighborhood of $\zeta$ in $D$.
\item There exists $j$ such that $(u_1, \zeta) \in U_j$ and a parameterization $u_2\colon H\to\C^{n}$ with $u_2(1)\in P_j'$ and $u_2(-1)\in P_j$ such that $u\circ\psi$ is defined on $H_{\rho_0}$ and lies at $C^{1}$-distance at most $\epsilon_0$ from $u_2|_{H_{\rho_0}}$. Furthermore, if $((u_1,\zeta),u_2)\in U^{\rm reg}$ (resp.~$((u_1,\zeta),u_2)\in U^{\rm sing}$) then the two $(\epsilon,(u_1,\zeta))$-points corresponding to $q^{\rm reg}_{\pm}$ (resp.~to $q^{\rm sing}_{\pm}$) lie in $u_2(\pa H_{\rho_0})$. 
\end{enumerate}
\end{Corollary}

\begin{proof}
If $(1)$ does not hold then we extract a sequence contradicting Lemma \ref{Lem:bubbles} and if the first part of $(2)$ does not hold then we extract a sequence contradicting Lemma \ref{Lem:limit1}. For the last part of $(2)$, observe that when parametrized as above, the points where $u_2\colon H\to\C^{n}$ intersects $\pa B(u_1(\zeta);\epsilon)$ lie at uniformly finite distance from $\pm 1$, by compactness of $\ev^{-1}(u_1(\zeta))\subset\MM^{\ast}(\beta)$.
\end{proof}

Corollary \ref{Cor:limit2} allows us to extract subsequences with bubbles that limit to maps of the form $u^{\rm st}_2$. More precisely, given $u$ as above, if $u_2\in U^{\rm reg}$ (resp.~if $u_2\in U^{\rm sing}$) then Corollary \ref{Cor:limit2} enables us to find the transverse intersection points with $\pa B(u_1(\zeta);\epsilon)$ corresponding to $q^{\rm reg}_{\pm}$ (resp.~to $q^{\rm sing}_{\pm}$). This allows us to define maps $\phi^{\rm st}\colon H\to H$, of the form $z\mapsto az+b$, which take $\pm 1$ to the intersection points corresponding to $q^{\rm reg}_{\pm}$ if $((u_1,\zeta),u_2)\in U^{\rm reg} - U^{\rm sing}$, to the intersection points corresponding to $q^{\rm sing}_{\pm 1}$ if $((u_1,\zeta),u_2)\in U^{\rm sing}- U^{\rm reg}$, and to the interpolation between intersection points determined by $\delta(u_2)$ as in \eqref{Eq:interpol-} and \eqref{Eq:interpol+} if $((u_1,\zeta),u_2)\in U^{\rm reg}\cap U^{\rm sing}$.

\begin{Corollary}\label{Cor:limit3}
Corollary \ref{Cor:limit2} holds with $(2)$ replaced by 
\begin{itemize}
\item[$(2')$] The map $u\circ\phi^{\rm st}$ is defined on $H_{\rho_0}$ and lies at $C^{1}$-distance at most $\epsilon_0$ from $u_2^{\rm st}|_{H_{\rho_0}}$.
\end{itemize}
\end{Corollary}

\begin{proof}
Noting that the coefficients $a$ and $b$ in the maps $\psi$ and $\phi^{\rm st}$ differ by a uniformly bounded amount, this follows immediately from Corollary \ref{Cor:limit2}.
\end{proof}

\section{Gluing} 
\label{Sec:gluing2}
In this section we define a Floer gluing map 
\[
\Phi\colon \NN\times[0,\infty)=\FF^{\ast}(-\beta)\times_{L}\MM^{\ast}(\beta)\times[0,\infty)\to \FF(0\beta)
\] 
following Floer's classical gluing scheme, and show that this map gives a $C^{1}$-parametrization of a neighborhood of the Gromov-Floer boundary of $\FF(0\beta)$.

\subsection{The Floer-Picard lemma with parameters}
Our main tool for gluing is the following result. Let $T$ and $M$ be finite dimensional smooth manifolds and let $\pi_{X}\colon X\to M\times T$ be a smooth bundle of Banach spaces over $M\times T$. Let $\pi_{B}\colon B\to T$ be a smooth bundle of Banach spaces over $T$. Let $f\colon X\to B$ be a smooth bundle map of bundles over $T$ and write $f_{t}\colon X_{t}\to B_{t}$ for the restriction of $f$ to the fiber over $t\in T$ (where the fiber $X_t$ is the bundle over $M$ with fiber $X_{(m,t)}$ at $m\in M$). If $m\in M$, then write $d_{m}f_{t}\colon X_{(m,t)}\to B_{t}$ for the differential of $f_{t}$ restricted to the vertical tangent space of $X_{t}$ at $0\in X_{(m,t)}$, where this vertical tangent space is identified with the fiber $X_{(m,t)}$ itself; similarly, the tangent spaces of $B_{t}$ are identified with $B_{t}$ using linear translations. Denote by $0_{B}$ the $0$-section in $B$,  and by $D(0_X;\epsilon)$  an $\epsilon$-disk sub-bundle of $X$.

\begin{Lemma}\label{Lem:FloerPicard}
Let $f\colon X\to B$ be a smooth Fredholm bundle map of $T$-bundles, with Taylor expansion in the fiber direction:
\begin{equation}\label{e:TaylorFP}
f_{t}(x)=f_{t}(0)+d_{m}f_{t}\,x+N_{(m,t)}(x),\quad\text{where }\; 0,x\in X_{(m,t)}.
\end{equation}
Assume that $d_{m}f_{t}$ is surjective for all $(m,t)\in M\times T$ and has a smooth family of uniformly bounded right inverses $Q_{(m,t)}\colon B_t\to X_{(m,t)}$, and that the non-linear term $N_{(m,t)}$ satisfies a quadratic estimate of the form
\begin{equation}\label{e:QuadraticFP}
\|N_{(m,t)}(x)-N_{(m,t)}(y)\|_{B_t}\le C\|x-y\|_{X_{(m,t)}}(\|x\|_{X_{(m,t)}}+\|y\|_{X_{(m,t)}}),
\end{equation}
for some constant $C>0$.  Let  $\ker(df_{t})\to M$ be the vector bundle with fiber over $m\in M$ equal to $\ker(d_mf_{t})$. 
If $\|Q_{(m,t)}f_{t}(0)\|_{X_{(m,t)}}\le\frac{1}{8C}$, then for $\epsilon<\frac{1}{4C}$,
$f^{-1}(0_{B})\cap D(0_{X};\epsilon)$  is a smooth submanifold
diffeomorphic to the bundle over $T$ with fiber at $t\in T$ the $\epsilon$-disk bundle in $\ker(df_{t})$.
\end{Lemma}

\begin{proof}
The proof for the case when $M$ and $T$ are both points appears in Floer \cite{Floer:mem} and generalizes readily to the case under study here. We give a short sketch pointing out some features that will be used below. Let $K_{(m,t)}=\ker(d_mf_{t})$ and choose a smooth splitting $X_{(m,t)}=X_{(m,t)}'\oplus K_{(m,t)}$ with
  projection $p_{(m,t)}\colon X_{(m,t)}\to K_{(m,t)}$. For $k_{(m,t)}\in K_{(m,t)}$, define the bundle map $\widehat{f}_{(m,t)}\colon X_{(m,t)} \to B_{t}\oplus K_{(m,t)}$, 
\[
\widehat{f}_{(m,t)}(x)=\bigl(f_t(x)\;,\;p_{(m,t)}\,x-k_{(m,t)}\bigr). 
\]
Then solutions to the equation $f_t(x)=0$, $x\in X_{(m,t)}$ with $p_{(m,t)}\,x=k_{(m,t)}$ are in one-to-one correspondence with solutions to the equation $\widehat{f}_{(m,t)}(x)=0$. Moreover, the differential $d\widehat{f}_{(m,t)}$ is an isomorphism with inverse 
$\widehat{Q}_{(m,t)}\colon B_t\oplus K_{(m,t)}\to X_{(m,t)}$, 
\[
\widehat{Q}_{(m,t)}\,(b,k)= Q_{(m,t)}\,b\; +\;k.
\]
On the other hand, solutions of the equation $\widehat{f}_{(m,t)}(x)=0$ are in one-to-one correspondence with fixed points of the map $F_{(m,t)}\colon X_{(m,t)}\to X_{(m,t)}$ given by
\[
F_{(m,t)}(x)=x\;-\;\widehat{Q}_{(m,t)}\,\widehat{f}_{(m,t)}(x).
\]
Fixed points are obtained from the Newton iteration scheme: if
\[
v_0=k_{(m,t)},\quad v_{j+1}=v_j\;-\;\widehat{Q}_{(m,t)}\,\widehat{f}_{(m,t)}(v_j),
\]
then $v_{j}$ converges to $v_\infty$ as $j\to\infty$ and
$F_{(m,t)}(v_\infty)=v_{\infty}$. Furthermore, if $\|f_{t}(0_{X_{(t,m)}})\|$ is sufficiently
small then there is $0<\delta<1$ such that:
\[
\|v_{j+1}-v_j\|\le \delta^{j}\|f_{t}(0)\|
\]
and consequently
\begin{equation}\label{Eq:iterest}
\|v_\infty-v_0\|\le \kappa \|f_{t}(0)\|,
\end{equation}
where $\kappa$ is a constant.
\end{proof}

\subsection{Pre-gluing maps of disks}
\label{sec:pregluemaps}
Recall from Section \ref{sec:cohertriv} that we associated to a pair of punctured disks $H_j$, $j=1,2$ and $\rho\in[1,\infty)$ a pre-glued disk $H_{1}\#_{\rho} H_{2}$ and that weight functions $\mathbf{e}_{\delta}$ on $H_j$ induced a weight function on $H_{1}\#_{\rho} H_{2}$. For simpler notation below we will write $D_{\rho}$ for $H_{1}\#_\rho H_{2}$.
For $((u_{1},\zeta),u_{2})\in\NN$ and $\rho\in[\rho_0,\infty)$ define the map
\[
w_{\rho}=u_1\#_\rho u_2\colon D_{\rho}\to\C^{n}
\]  
as follows. For $\rho_0$ large enough $u_1$ (resp.~$u_2^{\rm st}$) takes the region $(-\infty,-\rho_0]\times[0,1]$ (resp.~$[\rho_0,\infty)\times[0,1]$) to the holomorphic $(\C^{n},\R^{n})$-coordinates around $u_1(\zeta)\in L$. Define 
\[
u_1\#_\rho u_2(z)=
\begin{cases}
u_1(z) &\text{for }z\in H_1-\bigl((-\infty,-\rho+1)\times[0,1]\bigr) ,\\
u_2^{\rm st}(z) &\text{for }z\in H_2-\bigl((\rho+1,\infty)\times[0,1]\bigr),\\
(1-\alpha(z))u_1(z)+\alpha(z)u_2(z) &\text{for }z\in [-1,1]\times[0,1],
\end{cases}
\]
where all domains are subsets of $D_{\rho}=H_1\#_\rho H_2$. Here the domains of the first and second maps are clear. The third domain should be understood as 
\[
[-1,1]\times[0,1]\subset [-\rho,\rho]\times[0,1]\subset D_\rho,
\]
see \eqref{Eq:midstrip}, and $\alpha$ in the definition of the map is a smooth cut off function equal to $0$ near $\tau=-1$, equal to $1$ near $\tau=1$, and real valued and holomorphic on the boundary $[-1,1]\times\{0,1\}= (\pa D_{\rho})\cap [-1,1]\times[0,1]$.

\subsection{Adapted configuration spaces}
A pre-glued domain $D_{\rho}$ contains a middle strip
$[-\rho,\rho]\times[0,1]\subset D_{\rho}$,  
see \eqref{Eq:midstrip}. Write $0\in D_{\rho}$ for the point $(0,0)\in[-\rho,\rho]\times[0,1]$.  Let $\widehat{\cfig}$ denote the bundle over $[\rho_0,\infty)\times L$ the fiber of which over $(\rho,q)$ is the subset  
\[
\widehat{\cfig}(\rho,q)\subset\sblv^{2}_{\delta}(D_{\rho},\C^{n})
\]
of functions $u\colon D_{\rho}\to\C^{n}$ that satisfy the following:
\begin{enumerate}
\item $u$ takes the boundary $\pa D_{\rho}$ to $L$, the restriction of 
$(du)^{0,1}$ to the boundary (the trace) vanishes, and $u|_{\pa D_{\rho}}$ represents the zero homology class.
\item $u(0)=q\in L$.  
\end{enumerate}

We view $\widehat{\cfig}$ as a bundle over $[\rho_0,\infty)\times L$ with local trivializations in the $L$-directions given by composition with cut-off translation diffeomorphisms, as in \eqref{Eq:trivviadiffeo}, and with trivializations in the $[\rho_0,\infty)$-direction given by pre-composition with diffeomorphisms $\psi\colon D_{\rho}\to D_{\rho'}$ such that $\psi(0)=0$, $\psi=\id$ outside the middle strip, and inside the middle strip $\psi$ is close to a diffeomorphism that stretches/shrinks the $[-\rho,\rho]\times[0,1]$ in the first coordinate and which is holomorphic along the boundary. 

Each fiber $\widehat{\cfig}_{\rho}$ of  $\widehat{\cfig}\to[\rho_{0},\infty)$ fibers over $L$. The vertical tangent bundle of $\widehat{\cfig}\to[\rho_0,\infty)$ can be described as the bundle with fiber over $u\in\widehat{\cfig}_{\rho}$ equal to the direct sum 
\[
T_{u}\widehat{\cfig}_{\rho}=\EE(u)\oplus V_{\rm sol}(u).
\]
Here $\EE(u)\subset\sblv^{2}_{\delta}(D_{\rho})$ is the subspace of vector fields $v$ along $u$ that satisfy 
\begin{enumerate}
\item $v(z)\in T_{u(z)}L$ and $(\nabla v)^{0,1}|_{\pa D_{\rho}}=0$ (where $\nabla$ is the connection associated to the metric $\hat g$, see Section \ref{sec:Geomprel}),
\item $v(0,0)=0$,
\end{enumerate} 
and $V_{\rm sol}(u) \cong \R^n$ is the $n$-dimensional space of cut-off constant solutions coming from the values of the holomorphic coordinates at $u(0)$ (i.e. the space arising from linearizations of translation diffeomorphisms).  There is an exponential map $\exp_{u}\colon\EE(u)\oplus V_{\rm sol}(u)\to\widehat{\cfig}$ which can be written
\[
\exp_{u}(v,c)=\Phi_{c}(\Exp_{u}(v)),
\]
where $\Phi_{c}\colon \C^{n}\to\C^{n}$ is the translation diffeomorphism associated to $c\in\R^{n}$ and where $\Exp$ is defined through the metric $\hat g$, which gives $C^{1}$-charts of $\widehat{\cfig}_{\rho}$. 

\begin{Remark}
These charts vary in a $C^{1}$-smooth way with $\rho\in[\rho_0,\infty)$. Indeed, the local trivialization over the base is simply precomposition with a diffeomorphism $\psi\colon D_{\rho}\to D_{\rho'}$ that is holomorphic along the boundary, and the exponential map above is equivariant under precomposition with $\psi$:
\[
\exp_{u\circ\psi}(v\circ\psi, c)=\left(\exp_{u}(v,c)\right)\circ\psi.
\]
\end{Remark}

We will use $\widehat{\cfig}$ as the source space for the Floer operator. We denote the naturally corresponding target space $\widehat{\tcfig}$. This space is a locally trivial bundle over $[\rho_0,\infty)$ with fiber at $\rho$ equal to 
\[
\dot\sblv^{1}_{\delta}(D_{\rho},\Hom^{0,1}(TD_{\rho},\C^{n})),
\]
the subset of elements $A\in\sblv^{1}_{\delta}(D_{\rho},\Hom^{0,1}(TD_{\rho},\C^{n}))$ with $A|_{\pa D_{\rho}}=0$, and with local trivializations given by precomposition with the differentials $d\psi$ of the diffeomorphisms $\psi\colon D_{\rho}\to D_{\rho'}$ described above.

\subsection{Pre-gluing as a map and the Floer operator}
Consider the product $\widehat{\cfig}\times[0,\infty)$ and define the map 
\[
\Pre\colon\NN\times[\rho_0,\infty)\to\widehat{\cfig}\times[0,\infty)
\]
as follows (recall that $\NN=\FF^{\ast}(-\beta)\times_L\MM^{\ast}(\beta)$):
\[
\Pre(\xi,\rho)=(u_1\#_{\rho} u_{2},r(u_1)),
\]
where $\xi=((u_1,\zeta),u_2)$ and $r(u_1)$ is the coordinate of the Hamiltonian at $u_1$, i.e.~$u_1$ solves the Floer equation $(du_1+\gamma_{r(u_1)}\otimes X_H)^{0,1}=0$. By compactness of the moduli spaces involved we find that if $\rho_{0}$ is sufficiently large then $\Pre$ is a fiber preserving embedding (as a map of bundles over $[\rho_0,\infty)$). More precisely, the restriction 
\[
\Pre_{\rho}=\Pre|_{\NN\times\{\rho\}}\colon \NN\to\cfig_{\rho}
\] 
is a family of embeddings which depends smoothly on $\rho$.

Consider the normal bundle $N\NN_{\rho}$ of $\NN_{\rho}=\Pre_{\rho}(\NN)\subset\widehat{\cfig}_{\rho}$ as a sub-bundle of the restriction $T_{\NN_{\rho}}\widehat{\cfig}_{\rho}$ of the tangent bundle of $\widehat{\cfig}_{\rho}$ to $\NN_{\rho}$. The fiber $N_{w_{\rho}}\NN_{\rho}$ of this normal bundle at $w_{\rho}=u_{1}\#_{\rho} u_{2}\in\widehat{\cfig}_{\rho}$ is the $L^{2}$ complement, see Remark \ref{Rem:L2pairing} below, of the subspace $d\Pre(T_{\xi}\NN)\subset T_{w_\rho}\widehat{\cfig}_{\rho}$, where $\xi=((u_1,\zeta),u_2)\in\NN$. 

By transversality and a standard linear gluing result, see Lemma \ref{Lem:lingluetriv} below, $d\Pre(T_{\xi}\NN)= T_{w_{\rho}}\NN_{\rho}$ is given by
\[
T_{w_{\rho}}\NN_{\rho}=\ker(u_1)'\times_{T_{u_1(\zeta)}L}\ker(u_2)', 
\]  
where $\ker(u_i)'$ is spanned by cut-off versions of vector fields of the form $v'+c'$ where $v+c$ lies in the kernel of the linearized operator at $u_i$, with $v'$ a cut off version of $v$ and $c'$ the element in $V_{\rm sol}(u_1\#_{\rho} u_2)$ with the same constant value as $c'\in V_{\rm sol}(u_i)$. An element in the fibered product is of the form $v_1'+c'+v_2'$ where $(v_1',c')\in\ker(u_1)'$ and $(v_2', c')\in\ker(u_2)'.$

The bundles $T_{\NN}\widehat{\cfig}\to[\rho_0,\infty)$ and $N\NN\to [\rho_0,\infty)$ are locally trivial, with local trivializations given by precomposition with the diffeomorphisms $\psi\colon D_{\rho}\to D_{\rho'}$, in the latter case followed by $L^{2}$-projection to the normal bundle.  

\begin{Remark}\label{Rem:L2pairing}
The $L^{2}$ pairing on $T_{w_{\rho}}\widehat{\cfig}_{\rho}$ is to be understood as follows. Recall that
\[
T_{w_{\rho}}\widehat{\cfig}_{\rho}=\EE(w_{\rho})\oplus V_{\rm sol}(w_{\rho}).
\]
We define
\[
\langle (v,c) \, , \, (\tilde v, \tilde c)\rangle = \langle v \, ,\, \tilde v\rangle_{\sblv^{2}_{\delta}} + \langle c \, , \, \tilde c\rangle_{\R^{n}},
\]
where the first summand is the pairing on $\EE(w_{\rho})\subset\sblv^{2}_{\delta}(D_{\rho};\C^{n})$ induced by the weighted $L^{2}$-pairing on the ambient space and the second is the inner product of the values of the cut-off solutions at $0$ in $T_{w_{\rho}(0)}L$. 

We will also use the corresponding norm, if
\[
(v,c)\in \EE(w_{\rho})\oplus V_{\rm sol}(w_{\rho})=T_{w_{\rho}}\widehat{\cfig}_{\rho}\subset T_{\NN_{\rho}}\widehat{\cfig}_{\rho}
\]
then we define
\[
\|(v,c)\|_{T_{\NN_{\rho}}\widehat{\cfig}_{\rho}}=\|v\|_{2,\delta}+ \|c\|_{T_{w_{\rho}(0)}L},
\]
where $\|\cdot\|_{2,\delta}$ is the weighted Sobolev $2$-norm on $\sblv^{2}_{\delta}(D_{\rho},\C^{n})$ and where $\|\cdot\|_{T_qL}$ is the norm on the tangent space of $L$ at $q$ induced by the Riemannian metric. We will write $\|\cdot\|_{N\NN_{\rho}}$ for the restriction of this norm to the sub-bundle $N\NN_{\rho}$. 
\end{Remark}

The Floer equation now gives a smooth bundle map $f\colon T_{\NN}\widehat{\cfig}\to\widehat{\tcfig}$, where $f_{\rho}\colon T_{\NN_{\rho}}\widehat{\cfig}_{\rho}\to\widehat{\tcfig}_{\rho}$ is defined as follows. If $w_{\rho}= u_{1}\#_{\rho} u_{2}\in \NN_{\rho}$ then for $(v,c)\in T_{w_{\rho}}\widehat{\cfig}_{\rho}=\EE(w_\rho)\oplus V_{\rm sol}(w_\rho)$:
\begin{equation}\label{Eq:Floermaponglued}
f_{\rho}(v,c)=\left(
d \exp_{w_{\rho}}(v,c) + \gamma_{r(u_1)}\otimes X_{H}(\exp_{w_{\rho}}(v,c))
\right)^{0,1}.
\end{equation}
Below we will keep the notation $f$ for the restriction $f|_{N\NN}$. 

\subsection{The gluing map}
To construct the gluing map 
\[
\Phi\colon \NN\times[\rho_0,\infty)\to \FF(0\beta)
\]
which parametrizes a $C^{1}$-neighborhood of the Gromov-Floer boundary we will apply Lemma \ref{Lem:FloerPicard} to the map $f\colon N\NN\to\widehat{\tcfig}$ defined in \eqref{Eq:Floermaponglued}.  In the notation of Lemma \ref{Lem:FloerPicard}:

\begin{itemize}
\item  $[\rho_0,\infty)$ corresponds to $T$, $\NN$ corresponds to $M$;
\item  $N\NN$ corresponds to $X$,  $\widehat{\tcfig}$ corresponds to $B$; 
\item  we use the norm $\|\cdot\|_{N\NN_{\rho}}$ from Remark \ref{Rem:L2pairing} and  the norm $\|\cdot\|_{\widehat{\tcfig}_{\rho}}$ induced from the weighted Sobolev $1$-norm on $\sblv^{1}_{\delta}(D_{\rho}, \Hom^{0,1}(TD_{\rho},\C^{n}))$.
\end{itemize}

\begin{Lemma}\label{Lem:almostholomorphic}
There exists $\gamma>0$ such that for any $\xi=((u_1,\zeta),u_2)\in \NN$, if $w_\rho=u_{1}\#_{\rho} u_{2}$ then
\[
\|f_{\rho}(w_{\rho})\|_{\widehat{\tcfig}_{\rho}}=\Ordo(e^{-\gamma\rho}).
\]
\end{Lemma}

\begin{proof}
Write $u_2=u_2^{\rm st}$.
By Fourier expansion near $\infty\in H$, $|u_j^{(k)}(z)|=\Ordo(e^{-(\gamma+\delta)\rho})$, $j=1,2$, $k\le 2$, in the interpolation region where the weight function is bounded by $e^{\delta\rho}$.
\end{proof}

If $\xi=((u_1,\zeta),u_2)\in\NN$, let $w_{\rho}=u_{1}\#_{\rho} u_{2}$.  For the next result, in the correspondence with Lemma \ref{Lem:FloerPicard}, the fiber $N_{w_{\rho}}\NN$ of the normal bundle $N\NN_{\rho}$ at $w_{\rho}$ corresponds to $X_{(m,t)}$ and $w_{\rho}$ corresponds to $0\in X_{(m,t)}$.

\begin{Lemma}\label{Lem:partialglu}
For $\rho_0$ sufficiently large, the vertical differential 
\[
d_{w_{\rho}}f_{\rho}\colon N_{w_{\rho}}\NN_{\rho}\to \widehat{\tcfig}_{\rho}
\]
is surjective and admits a uniformly bounded right inverse. Moreover, the quadratic estimate \eqref{e:QuadraticFP} for the non-linear term in the Taylor expansion of $f_{\rho}$ holds.
\end{Lemma}

\begin{proof}
The quadratic estimate follows from \cite[Proof of Proposition 4.6]{EES}. In order to see that the differential is surjective and admits a uniformly bounded right inverse, we first note that the linearization 
\[
df_{\rho}\colon T_{w_{\rho}}\widehat{\cfig}_{\rho}\to \widehat{\tcfig}_{\rho}
\]
at $w_{\rho}$ is a Fredholm operator of index $n$. The subspace $T\NN_{\rho}\subset T_{w_{\rho}}\widehat{\cfig}_{\rho}=\EE(w_{\rho})\oplus V_{\rm sol}(w_{\rho})$ is also of dimension $n$ and we get a bounded right inverse as desired provided we prove an estimate of the form
\[
\|(v,c)\|_{T\widehat{\cfig}_{\rho}}\le C\|df_{\rho}(v,c)\|_{\widehat{\tcfig}_{\rho}}
\]
on the $L^{2}$-complement of $T_{w_{\rho}}\NN_{\rho}$. This follows from standard arguments, cf. Lemma \ref{Lem:lingluetriv}.
\end{proof}

Lemmas \ref{Lem:almostholomorphic} and \ref{Lem:partialglu} have the following consequence:

\begin{Corollary} \label{Cor:C1Structure}
The Newton iteration map with initial values in $\NN_{\rho}=\Pre(\NN\times\{\rho\})\subset N\NN_{\rho}$ gives a $C^{1}$ diffeomorphism
\[
\Phi\colon \MM^{\ast}(\beta)\times_{L}\FF^{\ast}(-\beta)\times[\rho_0,\infty)\to f^{-1}(0),
\]
where $0$ denotes the $0$-section in $\widehat{\tcfig}$,
such that $\Phi((u_1,\zeta),u_2,\rho)$ limits to a broken curve with components $(u_1,\zeta)\in\FF^{\ast}(-\beta)$ and $u_2\in\MM^{\ast}(\beta)$ as $\rho\to\infty$.
\end{Corollary}

\begin{proof}
Immediate from Lemma \ref{Lem:FloerPicard}.
\end{proof}

\subsection{Parameterizing a neighborhood of the Gromov-Floer boundary}
In this section we show that $\exp(f^{-1}(0))\subset\widehat{\cfig}$ gives a $C^{1}$-collar neighborhood $\approx \NN\times[0,\infty)$ of the Gromov-Floer boundary of $\FF(0\beta)$. There are two main points remaining. First, Corollary \ref{Cor:C1Structure} gives a $C^{1}$ $1$-parameter family, with parameter $\rho\in [\rho_0,\infty)$, of $\NN$-families of solutions in $\FF(0\beta)$ without giving much explicit information on how the solutions depend on $\rho$, and we must extract such information. Second, we need to establish the surjectivity of our construction, i.e.~show that all disks near the boundary are in the image of the gluing map.

Fix $\rho\in [\rho_0,\infty)$, and let $f^{-1}(0)_{\rho}=f^{-1}(0)\cap N\NN_{\rho}$. The domain of a map $u$ in $N\NN_{\rho}$ is the domain $D_{\rho}$ of the map $w_{_{\rho}}=u_{1}\#_{\rho} u_{2}$ (where $(u_{1},\zeta)\in\FF^{\ast}(-\beta)$ and $u_{2}\in\MM^{\ast}(\beta)$ with $u_{1}(\zeta)=u_{2}(1)$) in the fiber over which $u$ lies. In Section \ref{sec:cohertriv} we introduced the notation $H_{1}\#_{\rho} H_{2}=D_{\rho}$ for such domains with subsets $H_{1;\rho}\subset D_{\rho}$ and $H_{2;\rho}\subset D_{\rho}$ that are also subsets of the domains $H_{1}$ of $u_{1}$ and $H_{2}$ of $u_{2}$, see \eqref{Eq:H_1+H_2}. In particular, $H_{1;\rho}$ is a complement in $H_{1}$ of a strip neighborhood $(-\infty,-\rho)\times[0,1]$ of the boundary puncture $\zeta$. We identify $H_{1}$ with the upper half plane $H$ with $\zeta$ at $0$ and then think of $H_{1;\rho}$ as the complement $H-H_{e^{-\pi\rho}}$ of a half disk $H_{e^{-\pi\rho}}$ of radius $e^{-\pi\rho}$ centered at $0$. Similarly, we identify $H_{2}$ with the upper half plane with the marked point $1$ at $\infty$, and take $H_{2;\rho}$ to correspond to the disk $H_{e^{\pi\rho}}$ centered at $0$. Using these identifications we identify of $D_{\rho}$ with the upper half plane with $H-H_{e^{-\pi\rho}}$ corresponding to $H_{1;\rho}$, with $H_{2;\rho}$ corresponding to the disk $H_{e^{\pi\rho}}$ attached by the map which is scaling by $e^{-2\pi\rho}$.        

With these conventions we define a re-parametrization map
\[
\Psi_{\rho}\colon f^{-1}(0)_{\rho}\to\FF(0\beta),
\]  
as follows.  An element $u\in N\NN_{\rho}$ is a pair $(w_{\rho},(v,c))$ where $(v,c)\in N_{w_{\rho}}\NN_{\rho}$ and we consider it as a map $u=\exp_{w_{\rho}}(v,c)\colon D_{\rho}\to\C^{n}$, compare \eqref{Eq:Floermaponglued}. Using the above identification $D_{\rho}\approx H$, the map $u$ can be considered as a map from the upper half plane $H$. We further identify $H$ with the closed disk $D$ and let $\Psi_{\rho}(u)\colon D\to\C^{n}$ be the map $u$ considered as a map on $D$ after these re-parameterizations of the domain. Then the re-parameterized map $\Psi_{\rho}(u)$ solves the Floer equation if and only if $f(u)=0$. 

Let $\Psi\colon f^{-1}(0)\to \FF(0\beta)$ be the map which equals $\Psi_{\rho}$ on $f^{-1}(0)_{\rho}$ and note that as $\rho\to\infty$, 
\[
\inf_{u\in f^{-1}(0)_{\rho}}\{ \sup_{D}|d\Psi(u)|\}\to\infty.
\]

\begin{Lemma}\label{Lem:gluemb}
For $\rho$ sufficiently large, the map $\Psi$ is a $C^{1}$ embedding.
\end{Lemma}

\begin{proof}
The distance functions in $\FF(0\beta)$ and $\NN$ are induced from configuration spaces which are Banach manifolds with norms that control the $C^{0}$-norm. Since all norms are equivalent to the $C^{0}$-norm for (Floer) holomorphic maps, we may think of all distances $\mathrm{d}(\cdot,\cdot)$ in the calculations below as $C^{0}$-norms. Equation \eqref{Eq:iterest} and our rescaling implies that 
\begin{equation} \label{Eq:DerivativeBound}
\left|\frac{\pa(\Psi_{\rho}(\xi))}{\pa\rho}\right|_{C^{0}}\ge C e^{\pi\rho}.
\end{equation}
for some constant $C>0$. 

Assume now that that there is a sequence of pairs $(\xi,\rho)\ne (\xi',\rho')$ such that 
\[
\mathrm{d}(\Psi_{\rho}(\xi),\Psi_{\rho'}(\xi'))=\Ordo(\mathrm{d}(\xi,\xi')).
\]
Then it would follow that
\begin{equation} \label{Eq:InjectiveGlue}
\mathrm{d}(\Psi_{\rho}(\xi),\Psi_{\rho'}(\xi))\le \mathrm{d}(\Psi_{\rho'}(\xi),\Psi_{\rho'}(\xi'))+\Ordo(\mathrm{d}(\xi,\xi'))=\Ordo(\mathrm{d}(\xi,\xi')).
\end{equation}
However, taking $\epsilon_0>0$ small and $\mathrm{d}(\xi',\xi)\le \epsilon_0$, \eqref{Eq:InjectiveGlue} contradicts the derivative bound \eqref{Eq:DerivativeBound} once $\rho, \rho'$  are sufficiently large. We conclude that for small enough $\epsilon_0$,  the map is a $C^{1}$ embedding in an $\epsilon_0$-neighborhood of any point.

It follows that the map is also a global embedding: since $\Psi$ approaches the pregluing map as $\rho\to \infty$, there is $\rho_0>0$ so that for $\rho, \rho' > \rho_0$, if  $\mathrm{d}(\xi,\xi')>\epsilon_0$ then $\mathrm{d}(\Psi_{\rho}(\xi),\Psi_{\rho'}(\xi'))\ge \frac12\epsilon_0$.
\end{proof}

We are now ready to state our main gluing result:
\begin{Theorem}\label{Thm:gluing}
For $\rho_0$ sufficiently large, the map $\Psi\colon \NN\times[\rho_0,\infty) \rightarrow \FF(0\beta)$ is a $C^{1}$ embedding onto a neighborhood of the Gromov-Floer boundary.
\end{Theorem}

\begin{proof}
Take $\rho_0$ sufficiently large that the conclusion of Lemma \ref{Lem:gluemb} holds. 
It only then remains to show that (perhaps for some larger $\rho_0$) the map $\Psi$ is surjective onto a neighborhood of the Gromov-Floer boundary of $\FF(0\beta)$. A Gromov-Floer convergent sequence $w_{j}$ with non-trivial bubbling converges uniformly on compact subsets. We need to show that $w_j$ eventually lies in the image under the exponential map of a small neighborhood of $\NN \subset N\NN$ where the Newton iteration map is defined. To see this, we note that by Corollary \ref{Cor:limit3}, for any fixed $\rho_0$, the maps $w_j$ $C^{1}$-converge on the two ends of the strip region $[-\rho+\rho_0,\rho-\rho_0]\times[0,1]$ to $u_1$ and $u_2^{\rm st}$, respectively. On the remaining growing strip we have a holomorphic map which converges to a constant (since there cannot be further bubbling). 

Because of the weight in the Sobolev norm in pre-glued domains, knowing that the limiting holomorphic map on the strip region is constant is not quite sufficient. However, 
by action considerations, if $\rho > \rho_0 \gg 0$ is large enough, then both end-segments of $[-\rho+\rho_0,\rho-\rho_0]\times[0,1]$  must map into the same holomorphic coordinate chart.  Then the whole strip must map into this chart as well, again for action reasons. Thus, on this region, the map is holomorphic and maps into $(\C^{n},\R^{n})$ with standard holomorphic coordinates. It then has a Fourier expansion
\[
z\mapsto c_0+\sum_{n<0} c_n e^{n\pi z}. 
\]  
As in the proof of \cite[Theorem 1.3]{Ekholmtrees}, the $C^{0}$-norm near the ends of the strip controls the weighted norm and the shift, i.e. the norm in $\EE(w_{\rho})\oplus V_{\rm sol}(w_{\rho})$, where $w_{\rho}=u_{1}\#_{\rho} u_2^{\rm st}$.
\end{proof}

\section{The index bundle and coherent trivializations} \label{Sec:IndexBundle}
In this section we study the index bundle over the configuration space $\cfig_{\rm sm}$ of smooth maps $(D^2, S^1) \rightarrow (\C^{2k}, W')$ with Lagrangian boundary condition given by the Lagrangian $W' \approx S^{1}\times S^{2k-1} \subset \C^{2k}$ that results from Lagrange surgery on the Whitney sphere. In particular, we establish stable triviality of the index bundle and existence of coherent trivializations. Recall that Lemma \ref{Lem:GaussInterpolate} shows that the corresponding results then hold for $\cfig_{\rm sm}$ with Lagrangian boundary conditions in a general $L\subset\C^{2k}$ constructed from Lagrange surgery on an immersed homotopy sphere with one double point.

The starting point of our analysis is the map $\psi\colon\cfig_{\rm sm}\to \LL W'$ that takes a disk $u\colon (D,\pa D)\to (\C^{n}, W')$ to its restriction to the boundary $u|_{\pa D}$, which is an element in the free loop space $\LL W'$ of $W'$. Since the fiber of this map is contractible by linear homotopies, $\psi$ is a homotopy equivalence. Hence it suffices to study the index bundle over $\LL W'=\LL(S^{1}\times S^{2k-1})$, where the operator at a loop $\gamma$ is the $\bar\pa$-operator with Lagrangian boundary conditions given by the tangent planes of $W'$ along $\gamma$. We write $\LL_{j}(S^{1}\times S^{2k-1})$ for the subset of loops in homology class $j\in\Z=H_{1}(S^{1}\times S^{2k-1})$.

\subsection{A $(6k-7)$-skeleton for $\LL(S^{1}\times S^{2k-1})$}
\label{sec:skeleton}
We use the decomposition $\LL(S^{1}\times S^{2k-1})=\LL S^{1}\times \LL S^{2k-1}$ and consider the factors separately. The space $\LL S^{1}$  is homotopy equivalent to $S^{1}\times \Z$, where $S^{1}\times\{j\}$ consists of the geodesics that traverse $S^{1}$ $j$ times and where the $S^{1}$-coordinate of such a geodesic corresponds to its starting point. (For future reference, we remark that the choice of base-loop in each homotopy class is inessential.) We thus have
\[
\LL S^{1} \ \simeq  \ \bigcup_{j\in\Z} S^{1}_j 
\]
where the subscript denotes the homotopy class.  In order to find a skeleton for $\LL S^{2k-1}$ we consider the unit tangent bundle 
\[
\pi\colon U S^{2k-1} \longrightarrow S^{2k-1}.
\]
Write 
\[
\kappa\colon Q\longrightarrow US^{2k-1}
\]
for the vertical tangent bundle of this bundle, with fiber $Q_v=\kappa^{-1}(v)$ at $v\in U S^{2k-1}$ equal to $\ker(d\pi_v)$. We will think of $Q$ geometrically as follows: if $x\in S^{2k-1}\subset\R^{2k}$  is a unit vector in $\R^{2k}$, $|x|=1$, then $U_xS^{2k-1}$ is the $(2k-2)$-sphere of unit vectors $v\in\R^{2k}$ perpendicular to $x$:
\[
U_x S^{2k-1}=\{v\in\R^{2k}\colon |v|=1,\; \langle v, x\rangle=0\},
\]
and the fiber $Q_{v}$ is the tangent space to $U_{x}S^{2k-1}$ at $v$, which is the $(2k-2)$-space of vectors $q\in\R^{2k}$ that are perpendicular to both $x$ and $v$:
\[
Q_v=\{q\in\R^{2k}\colon \langle q, x\rangle=0,\; \langle q,v\rangle=0\}.
\]
Finally let $\widehat{Q}\to S^{2k-1}$ denote the fiber wise Thom space of $Q$ with fiber over $x\in S^{2k-1}$ the Thom space $M T (U_xS^{2k-1})$, which is the one-point compactification of $T(U_x S^{2k-1})$. We write $*_x$ for the point at infinity in $M T(U_x S^{2k-1})$. 

We next define an embedding $\Phi\colon\widehat{Q}\to\LL S^{2k-1}$ which, as we shall see, gives a $(6k-7)$-skeleton for $\LL S^{2k-1}$. Here we think of $\widehat{Q}_x$ as the unit disk bundle in the tangent bundle of $U_xS^{2k-1}$ with all points on the boundary collapsed to the single point $*_x$. We define the map using the geometric interpretation of $Q$ above, as follows.
\begin{enumerate}
\item For $0\in Q_v$, define $\Phi(0)$ to be the great circle through $x=\pi(v)$ with tangent vector $v$:
\[
\Phi[0](t)= (\cos t)x + (\sin t)v, \quad 0\le t\le 2\pi.
\]
\item For $q\in Q_v$ with $0<|q|\le\frac12$, define $\Phi(q)$ to be the piecewise smooth curve that consists of two great half-circles in the 2-hemisphere determined by $\Phi[0]$ and $q$ meeting at $\pm x$ at an angle $2\pi|q|$:
\[
\Phi[q](t)= 
\begin{cases}
(\cos t)x + (\sin t)v, &\text{ for }0\le t\le \pi,\\
(\cos t)x  + (\sin t)
\left((\cos 2\pi |q|) v + (\sin 2\pi |q|)|q|^{-1}q\right) &\text{ for }\pi\le t\le 2\pi.
\end{cases}
\]
Note that $\Phi[q]=\Phi[q']$ for any $q,q'$ with $|q|=|q'|=\frac12$.
\item For $q\in Q_v$ with $\frac12\le |q|\le 1$,  define $\Phi[q]$ to be a shortening of the curve $\Phi[q_0]$ for $|q_0|=\frac12$:
\[
\Phi[q](t)
=\begin{cases}
(\cos 2(1-|q|)t)x + (\sin 2(1-|q|)t)v, &\text{ for }0\le t\le \pi,\\
(\cos 2(1-|q|)(2\pi-t))x  + (\sin 2(1-|q|)(2\pi-t))v, &\text{ for }\pi\le t\le 2\pi. 
\end{cases}
\]
\item Define $\Phi[*_x]$ as the constant curve at $x$.
\end{enumerate}

\begin{Lemma}\label{Lem:skeleton}
The image  $\Phi(\widehat{Q})\subset \LL S^{2k-1}$ of the embedding $\Phi$ is a $(6k-7)$ skeleton of $\LL S^{2k-1}$, in the sense that the pair $(\LL S^{2k-1},\Phi(\widehat{Q}))$ is $(6k-7)$-connected.
\end{Lemma}

\begin{proof}
Consider the fibration of the free loop space 
$\pi\colon \LL S^{2k-1}\to S^{2k-1}$,
with fiber $\pi^{-1}(x)=\Omega (S^{2k-1},x)$, the space of loops based at $x$.  
Consider the round metric on $S^{2k-1}$. This induces an energy functional on $\Omega(S^{2k-1},x)$ which is a Bott-Morse function and which has critical manifolds as follows:
\begin{enumerate}
\item A point corresponding to the constant geodesic at $x$.
\item A $(2k-2)$-sphere $S^{2k-2}_j$ of geodesic loops that start at $x$ and traverse a great circle $j$ times. The index of $S^{2k-2}_j$ is $(2j-1)(2k-2)$, since a geodesic $\gamma$ that traverses a great circle $j$ times has $2k-2$ linearly independent normal Jacobi-fields that vanish at $ \pm x$. Thus, there are $2j-1$ conjugate points along $\gamma$, each of multiplicity $2k-2$.
\end{enumerate}
Consider the gradient flow lines connecting $S^{2k-2}_1$ to the constant geodesic at $x$. Fix some $v\in S^{2k-2}_1$. The part of the stable manifold of $x$ emanating from $v$ can be identified with the disk inside $S^{2k-1}$ normal to the geodesic $v$ at its base-point $x$. By symmetry of the metric, the flow line in any fixed direction $q$ shrinks the great circle $v$ inside the 2-hemisphere determined by $v$ and $q$. The map $\Phi$ defined above gives a particular such shrinking, which is homotopic to the gradient flow. Considering such gradient flows for varying $x$, it follows that the unstable manifold of $S^{2k-1}_1$ is homotopic to $\Phi(\widehat{Q})$. 
Since for $j\geq 2$ the index of the Bott manifold $S_j^{2k-2}$ is at least $(2j-1)(2k-2)\ge 6k-6$ the lemma follows. 
\end{proof}

Summing up, we get a $(6k-7)$-skeleton of $\LL_{j}(S^{1}\times S^{2k-1})$ of the form
$S^{1}_{j}\times \widehat{Q}$,
where a point $(y,q)$ corresponds to the loop which first traverses the loop corresponding to $y$ inside the $S^{1}$-factor, while being constant at the base point in the $S^{2k-1}$-factor, and then follows the loop $\Phi[q]$ in the $S^{2k-1}$-factor, while being constant at the base point in the $S^{1}$-factor.

\subsection{A linear gluing lemma} \label{Sec:LinearGluing}
In this section we establish a linear gluing result that we use both to construct stable trivializations and in formulating the definition of coherent stable trivializations. The argument is standard; we present a version tailored for the applications in this paper.  Essentially the same argument would establish  the uniform bound of the right inverse operator in Lemma \ref{Lem:partialglu}.

Let $D$ be a disk with a loop $\{ \Lambda(z) \subset \bC^n\}_{z\in\pa D}$ of Lagrangian $n$-planes in $\C^{n}$ specified along its boundary. Let $\bar\pa_{D}$ denote a Cauchy-Riemann operator on $\C^{n}$-valued vector fields on $D$. We write $\sblv^{2}(D;\Lambda)$ for the Sobolev space of maps $v\colon D\to\C^{n}$ with two derivatives in $L^{2}$ and which satisfy the following:
\begin{enumerate}
\item $v(z)\in \Lambda(z)$, $z\in\pa D$ and
\item $\bar\pa_{D} v|_{\pa D}=0$ (where the restriction to the boundary should be interpreted as the trace).  
\end{enumerate}  
We write $\dot \sblv^{1}(D;\Hom^{0,1}(TD,\C^{n}))$ for the Sobolev space of complex anti-linear maps $TD\to \C^{n}$ with one derivative in $L^{2}$ that vanish on the boundary, and write
\[
\bar\pa_{\Lambda}\colon\sblv^{2}(D;\Lambda)\to \dot\sblv^{1}(D;\Hom^{0,1}(TD,\C^{n}))
\]
for the operator that maps $v$ to $\bar\pa_{D}v$. Then  $\bar\pa_{\Lambda}$ is a Fredholm operator of index 
\[
\ind(\bar\pa_{\Lambda})=n+\mu(\Lambda),
\]
where $\mu(\Lambda)$ is the Maslov index of $\Lambda$. The kernel of $\bar\pa_{\Lambda}$ is spanned by smooth vector fields. 

Fix $\zeta\in\pa D$, puncture $D$ at $\zeta$, and identify a neighborhood of $\zeta$ with a half-infinite strip $[0,\infty)\times[0,1]$ as in Section \ref{sec:cohertriv}. Fix $\delta\in(0,\pi)$ and consider the Sobolev space $\sblv^{2}_{\delta}(D,\zeta;\Lambda)$ weighted by a function which equals $1$ except in the half strip around $\zeta$ where it is given by $e^{\delta|\tau|}$ for $\tau+it\in[0,\infty)\times[0,1]$. Let $\dot\sblv^{1}_{\delta}(D,\zeta;\Hom^{0,1}(TD,\C^{n}))$ denote the corresponding weighted Sobolev space of complex anti-linear maps that vanish along the boundary. Assume now that $\Lambda(z)=\R^{n}\subset\C^{n}$ in a neighborhood of $\zeta$ and that $\bar\pa_{D}$ agrees with the standard $\bar\pa$-operator in some neighborhood of $\zeta$. Fix a cut-off function $\alpha\colon [0,\infty)\times[0,1]$ which equals $1$ for $\tau>\tau_0$, which equals $0$ for $\tau<\tau_0-1$, and which is real valued and holomorphic on the boundary. Assume furthermore that $\tau_0$ is sufficiently large that the boundary condition $\Lambda$ is constant and the operator $\bar\pa_{D}$ is standard for $\tau>\tau_0$. Let 
\[
V_{\rm sol}(\zeta)=\R\langle \alpha e_1,\dots,\alpha e_n\rangle,  
\] 
where $e_1,\dots,e_n$ is the standard basis in $\R^{n}$,  so  elements in $V_{\rm sol}(\zeta)$ are cut off constant solutions. For $v_{\rm sol}\in V_{\rm sol}(\zeta)$, define $\bar\pa_{\Lambda} v_{\rm sol}=\bar\pa v_{\rm sol}\in\dot\sblv^{1}_{\delta}(D,\zeta;\Hom^{0,1}(TD;\C^{n}))$. Then the operator
\[
\bar\pa_{\Lambda}\colon \sblv^{2}_{\delta}(D,\zeta;\Lambda)\oplus V_{\rm sol}(\zeta)\to
\dot\sblv^{1}_{\delta}(D,\zeta;\Hom^{0,1}(TD,\C^{n}))
\] 
is Fredholm of index $n+\mu(\Lambda)$, and there is a canonical isomorphism between the kernel of this operator and the kernel of the operator discussed above. To see this, note that Taylor coefficients at $\zeta$ of solutions of the problem on the closed disk correspond to Fourier coefficients of solutions on the punctured disk, where the added cut-off solutions on the Fourier side provide the constant terms in the Taylor expansion.   
 
We next consider stabilizations. Let $A_1,\dots,A_m\in \dot\sblv^{1}(D;\Hom^{0,1}(TD,\C^{n}))$. For convenience we take the $A_j$ to be smooth sections with support outside a neighborhood of the marked point $\zeta\in D$. We define
\[
\psi\colon \R^{m}\to \dot\sblv^{1}_{\delta}(D;\Hom^{0,1}(TD,\C^{n}))
\]
by $\psi(e_j)=A_j$, $j=1,\dots,m$, where $e_1,\dots,e_m$ are basis vectors in $\R^{m}$. We now consider the stabilized operators
\begin{align}  \label{Eq:StabilizedOperators}
\bar\pa_{\Lambda}\oplus\psi&\colon
\sblv^{2}(D;\Lambda)\oplus \R^{m}\to \dot\sblv^{1}(D;\Hom^{0,1}(TD,\C^{n})), \\
\bar\pa_{\Lambda}\oplus\psi&\colon
\sblv^{2}_{\delta}(D,\zeta;\Lambda)\oplus V_{\rm sol}(\zeta)\oplus\R^{m}\to
\dot\sblv^{1}_{\delta}(D,\zeta;\Hom^{0,1}(TD,\C^{n})),
\end{align} 
where the second operator is stabilized exactly as the first (recall that we took the images $\psi(e_j) = A_j$ to be smooth sections supported outside the strip neighborhood of the puncture).
Again we have a canonical identification between the kernels of the two problems.

Consider now two operators as above: $\bar\pa_{\Lambda_1}$ on $D_1$ and $\bar\pa_{\Lambda_2}$ on $D_2$. Assume that $D_1$ is punctured at $\zeta$ and $D_2$ at $1$, and consider the punctured versions of the operators as described above. Assume furthermore that $\Lambda_1(\zeta)=\Lambda_2(1)$. Then we can pre-glue the domains, see Section \ref{sec:cohertriv} and \ref{sec:pregluemaps}, to obtain a domain $D_\rho$ with an induced weight function. Furthermore, the boundary conditions $\Lambda_1$ and $\Lambda_2$ give a boundary condition $\Lambda_1\#\Lambda_2$ on $\pa D_{\rho}$ and the operators glue to an operator
$\bar\pa_{\rho}$ on $\C^{n}$-valued vector fields over $D_\rho$.

Define $\widehat{\sblv}_{\delta}^{2}(D_{\rho};\C^{n})$ as the weighted Sobolev space with two derivatives in $L^{2}$, of vector fields which satisfy the usual boundary conditions, and the extra condition that $v(0)=0$, where $0$ is the distinguished point in the strip region $[-\rho,\rho]\times[0,1]\subset D_{\rho}$, see Section \ref{sec:cohertriv}. Let $V_{\rm sol}(0)$ denote the space of cut-off constant solutions in this strip region. Consider the operator 
\[
\bar\pa_{\rho}\colon \widehat{\sblv}_{\delta}^{2}(D_{\rho};\C^{n})\oplus V_{\rm sol}(0)
\to \dot\sblv_{\delta}^{1}(D_{\rho};\Hom^{0,1}(TD_{\rho},\C^{n}))
\]
Assume now we have stabilizations $(\psi_1,\R^{m_1})$ and $(\psi_2,\R^{m_2})$ of $\bar\pa_{\Lambda_i}$ that make the operators of \eqref{Eq:StabilizedOperators} surjective. They then induce a stabilization $(\psi_1\oplus\psi_2,\R^{m_1+m_2})$ of $\bar\pa_{\rho}$. Let $\psi_0$ denote the auxiliary stabilization that adds another copy of $V_{\rm sol}(0)$ to the domain, this copy being orthogonal to the first with respect to the $L^{2}$-pairing, see Remark \ref{Rem:L2pairing}.

Note also that there are evaluation maps $\ev_{\zeta}\colon\ker(\bar\pa_{\Lambda_1}\oplus\psi_1)\to \R^{n}$ and $\ev_1\colon\ker(\bar\pa_{\Lambda_2}\oplus\psi_2)\to\R^{n}$. We say that the kernels are transverse at the gluing point if these images together span $\R^{n}$.

\begin{Lemma}\label{Lem:lingluetriv}
For all $\rho$ that are sufficiently large,  $L^{2}$-projection gives an isomorphism 
\[
\ker(\bar\pa_{\rho}\oplus\psi_{1}\oplus\psi_{2}\oplus\psi_0)= \ker(\bar\pa_{\Lambda_1}+\psi_{1})\oplus \ker(\bar\pa_{\Lambda_2}+\psi_{2}).
\]
Furthermore, if the kernels are transverse at the gluing point then $L^{2}$-projection gives an isomorphism
\[
\ker(\bar\pa_{\rho}\oplus\psi_{1}\oplus\psi_{2})= \ker(\bar\pa_{\Lambda_1}+\psi_{1})\times_{\R^{n}} \ker(\bar\pa_{\Lambda_2}+\psi_{2}),
\]
where the fibered product is with respect to the evaluation map $\ev_\zeta\times\ev_1$ to $\R^{n}\times\R^{n}$.
\end{Lemma}

\begin{proof}
We start with the first statement. Write $L=\bar\pa_{\rho}\oplus\psi_{1}\oplus\psi_{2}\oplus\psi_{0}$ and observe that the index of 
\[
L\ \colon \
\widehat{\sblv}_{\delta}^{2}(D_{\rho};\C^{n})\oplus \R^{m_1 + m_2} \oplus V_{\rm sol}^{1}(0)\oplus V^{2}_{\rm sol}(0) \ \longrightarrow \
\dot\sblv_{\delta}^{1}(D_{\rho};\Hom^{0,1}(TD_{\rho},\C^{n})),
\]
where the superscripts on the $V_{\rm sol}(0)$ are just to tell the two copies apart,
equals 
\begin{align*}
&n+\mu(\Lambda_1\#\Lambda_2)+m_1+m_2+n \ = \ (n+\mu(\Lambda_{1})+m_1)+(n+\mu(\Lambda_{2})+m_2)\\
&=\dim(\ker(\bar\pa_{\Lambda_{1}}\oplus\psi_1))+\dim(\ker(\bar\pa_{\Lambda_{2}}\oplus\psi_2)).
\end{align*}
Consider an element $v_{j}$ in the kernel of $\bar\pa_{\Lambda_j}+\psi_j$. It can be written uniquely as
\[
v_j=v_{j;\delta} + v_{j;{\rm sol}} + w_j,
\]
where $v_{j;\delta}\in \sblv^{2}_{\delta}(D,\zeta;\Lambda_j)$, $v_{j;{\rm sol}}\in V_{\rm sol}(\zeta)$, and $w_j\in\R^{m_j}$. Let $\beta_j$ be a smooth cut-off function that equals $1$ on $H_{1;\rho+1}$, equals $0$ outside $H_{1;\rho+2}$, and is real valued and holomorphic on the boundary. Define
\[
\tilde v_j=\beta_j v_{j;\delta} + v^{j}_{\rm sol} + w_j,
\]
where $v^{j}_{\rm sol}$ is the cut-off constant solution in $V_{\rm sol}^{j}(0)$ with the same value as $v_{j;{\rm sol}}$. 

Pick bases $v_1^{1},v_{1}^{2},\dots,v_{1}^{M_{1}}$ and  $v_2^{1},v_{2}^{2},\dots,v_{2}^{M_{2}}$ in $\ker(\bar\pa_{\Lambda_{1}}\oplus\psi_{1})$ and $\ker(\bar\pa_{\Lambda_{2}}\oplus\psi_{2})$, respectively. We claim that the operator $L$ is invertible on the $L^{2}$-complement of the subspace spanned by
\[
\tilde v^{1}_{1} \ , \ \dots \ ,\ \tilde v^{M_1}_{1} \ , \ \tilde v^{1}_{2}
\ , \ \dots \ , \ \tilde v_{2}^{M_2}.
\]
Indeed, suppose not. Then there exists a sequence $u_j$ in the $L^{2}$-complement such that
\[
\|u_j\|=1,\quad \|Lu_{j}\|\to 0\text{ as }j\to\infty.
\]
Write 
\[
u_j= u_{j;\delta} + u_{j; {\rm sol}}^{1} + u_{j;{\rm sol}}^{2},
\]
where  $u_{j;{\rm sol}}^{k}\in V_{\rm sol}^{k}(0)$, $k=1,2$. Consider now the sequence 
\[
u_{1;j}= \beta_{1}u_{j;\delta} + u_{j;{\rm sol}}^{1}.
\]
These functions are orthogonal to the kernel of $\bar\pa_{\Lambda_1}+\psi_1$ and $\|L u_{1;j}\|\to 0$. Thus $u_{1;j}\to 0$. Repeating this argument for the other half of the disk, we find that $u_{2;j}\to 0$; but then $u_j\to 0$, which contradicts $\|u_j\|=1$. We conclude the desired invertibility, and that $L^{2}$-projection gives an isomorphism, as claimed.

To prove the last statement we argue similarly, inverting instead on the fibered product of the kernels. The only difference is in the last step: there is now only one copy of $V_{\rm sol}(0)$, and writing the component of $u_j$ along this copy as $u_{j;{\rm sol}}$ our transversality assumption implies that either $\beta_1u_{j;\delta}+ u_{j;{\rm sol}}$ is orthogonal to the (approximate) kernel of $\bar\pa_{\Lambda_1}\oplus\psi_1$ or $\beta_2u_{j;\delta}+ u_{j;{\rm sol}}$ is orthogonal to that of $\bar\pa_{\Lambda_2}\oplus\psi_2$. This means that we can extract a subsequence for which one of these alternatives hold. Noting that $\beta_ku_{j;\delta}$ is orthogonal to the kernel of $\bar\pa_{\Lambda_{k}}\oplus\psi_k$, $k=1,2$ we then get a contradiction using this subsequence, exactly as in the first argument.  
\end{proof}

\subsection{Stable trivializations}
Write $\mathbf{I}_{j\beta}$ for the index bundle over $\LL_{j}W'$ with operator over the loop $\gamma$ 
\begin{equation}\label{Eq:linoploop}
D(\bar\pa_{j\beta})\colon T_{u}\cfig\to \tcfig,
\end{equation}
where $u\colon (D,\pa D)\to(\C^{n},L)$ is any map with $u|_{\pa D}=\gamma$. Note that \eqref{Eq:linoploop} is independent of the particular choice of $u$. For convenient notation, we 
write
\[
\bar\pa_{\gamma}\colon \SS_{\gamma}\to \tcfig
\]
for the operator in \eqref{Eq:linoploop}. We remark that in the subsequent computations, we work with the \emph{standard} $\bar\pa$-operator (for the almost complex structure $J_0$, with no Hamiltonian perturbation), which is sufficient since we are essentially working up to homotopy.

\begin{Lemma}\label{Lem:indextrivial}
$\mathbf{I}_{j\beta}$ is stably trivial over the $(6k-7)$-skeleton $S^{1}_{j}\times\Phi(\widehat{Q})$ from Lemma \ref{Lem:skeleton}.
\end{Lemma}

\begin{proof}
Consider the model of $W'$ from Example \ref{Ex:LefschetzWhitney}:
\begin{equation} \label{Eq:WhitneyAgain}
W'=\bigcup_{e^{i\theta}\in S^{1}} e^{\frac{i\theta}{2}}\cdot S^{2k-1}  \;\;\subset \;\; \C^{2k}.
\end{equation}
Recall that we think of the loops $\gamma$ in the skeleton $S^{1}_{j}\times\Phi(\widehat{Q})$ as loops which first go around the $S^{1}$-factor and then along the $S^{2k-1}$-factor. We take the corresponding operator $\bar\pa_{\gamma}$ to act on vector fields on a pre-glued disk $D_{\rho}=H_{1}\#_{\rho} H_{2}$ for some sufficiently large $\rho>0$, where the sub-loops of $\gamma$ in the $S^{1}$ and $S^{2k-1}$ factor are parameterized by $\pa H_1$ and $\pa H_{2}$, respectively. It follows from Lemma \ref{Lem:lingluetriv} that it is sufficient to find stable trivializations over loops which are constant in either factor separately. For the $S^{1}$-factor this is obvious since there is only one loop. We thus consider trivializations over loops in the $S^{2k-1}$-factor that are constant in the $S^{1}$-factor. Take the constant value to be $1\in S^{1}$ and write $(x,v,q)$ for points in  $\widehat{Q}=MTUS^{2k-1}$, where $v\in U_{x}S^{2k-1}$ and $q\in Q_{v}$, cf. the notation of Section \ref{sec:skeleton}.

Consider first $\bar\pa_{\Phi(x,v,0)}$ with Lagrangian boundary condition corresponding to a simple closed geodesic starting at $x$ with tangent vector $v$. Recalling that we use the complex structure $J_0$, the operator $\bar\pa_{\Phi(x,v.0)}$ splits into one-dimensional problems as follows.
\begin{enumerate}
\item There are $(2k-2)$ one-dimensional problems with constant real boundary conditions,  corresponding to a set of basis vectors in  $x^{\perp}\cap v^{\perp}$.
\item In the two-dimensional space spanned by $x$ and $v$, the Lagrangian boundary condition along $(\cos t)x+(\sin t)v$, $0\le t\le 2\pi$ is spanned by the vectors 
\[
i\bigl((\cos t)x+(\sin t)v\bigr)\quad\text{ and }\quad (-\sin t)x +(\cos t)v. 
\]
In the complex basis $ix + v, ix -v$ this boundary condition splits into the two one-dimensional boundary conditions
\[
e^{it}(ix +v) \quad\text{ and }\quad e^{-it}(ix-v),
\] 
where the former has Maslov index $2$ (hence index $3$), and the latter Maslov index $-2$ (hence index $-1$).
\end{enumerate}
The argument principle implies that one-dimensional problems have only kernel or only cokernel. We thus find that 
\[
\ker(\bar\pa_{\Phi(x,v,0)})=\bigl\langle \nu_1,\dots,\nu_{2k-2}, a_1,a_2,a_3 \bigr\rangle,
\]
where $\langle\cdot\rangle$ denotes the linear span,
where $\nu_j$ are linearly independent constant solutions in directions normal to $x$ and $v$, where $a_1,a_2,a_3$ are three linearly independent vector fields in the $(ix+v)$-line which correspond to linearized automorphisms of the unit disk in this complex line.  Similarly, we deduce that
\[
\coker(\bar\pa_{\Phi(x,v,0)})=\bigl\langle b \bigr\rangle,
\]
where $b$ is a vector field in the $(ix-v)$-line which $L^{2}$-pairs non-trivially with the constant solution of the adjoint problem.

We next consider the family of operators $\bar\pa_{\Phi(x,v,q)}$ for $q$ in the $(2k-2)$-disk $Q_v$ which we think of as the normal disk to $\Phi(x,v,0)$ at its start point, cf. the proof of Lemma \ref{Lem:skeleton}. For $q$ in an $\epsilon$-disk in $Q_{v}$ we keep the boundary condition constant, but change the disk to a pre-glued disk $D_{\rho}=H_1\#_\rho H_2$ where the first half of the boundary condition, $0\le t\le \pi$, is parameterized by $\pa H_1$ and the second half by $\pa H_2$. The argument principle shows that the kernel and cokernel are unchanged by such deformations. However, by (the second statement in) Lemma \ref{Lem:lingluetriv}, for $\rho>0$ large enough we get an approximate kernel, isomorphic to the actual kernel by $L^{2}$-projection, spanned by the following sections:
\begin{enumerate}
\item constant kernel functions in directions perpendicular to $x$ and $v$,
\item a cokernel function which pairs non-trivially with the constant solution of the adjoint problem in the $(x-iv)$-line,
\item three approximate kernel functions in the $(x+iv)$-line: the solution over $H_1$ that vanishes at the puncture, the solution over $H_2$ that vanishes at the puncture, and the sum of the solutions in $H_1$ and $H_2$ that both equal $1$ at the puncture, cf.~ Lemma \ref{Lem:lingluetriv}.
\end{enumerate}

We now fix $q$, and start rotating the boundary condition of the second half-plane in direction $q$, as in the definition of $\Phi$, see Section \ref{sec:skeleton}. For rotation angle in $(0,\pi)$ we claim that the kernel is $2k$-dimensional and the cokernel trivial. To see this we invert the operator on the complement of the space spanned by the following:
\begin{enumerate}
\item the $(2k-3)$-dimensional space of constant solutions (orthogonal to $x,v,q$);
\item the cut off solutions of the pieces that satisfy the incidence condition at the gluing point. 
\end{enumerate}
The usual linearized gluing argument, Lemma \ref{Lem:lingluetriv}, shows that this is possible: as in that Lemma, one shows that for a sequence of functions perpendicular to the space spanned by the elements above, the functions must tend to zero  in each of $H_1$, $H_2$, and in the middle strip. (The difference with the degenerate case considered previously is that now the  estimate for the problem over $H_1$ implies that the $v$-component of a function in the $L^2$-orthogonal must vanish, whilst that over $H_2$ implies that the $q$-component must vanish.)

When the rotation angle equals $\pi$ we still get a $2k$-dimensional kernel, now spanned by the $(2k-2)$-dimensional space of constant sections perpendicular to $\Phi_{(x,v,0)}$ and the two solutions on $H_1$ and $H_2$ that vanish at the point where the disks are joined. We finally shorten the curve until it is constant, and the argument principle again implies that the kernel does not change.   

In summary, after stabilizing with one auxiliary direction we may identify the vectors in the kernel with $ix+v$ and $ix-v$ (by evaluation at $1$). Hence the kernel of the once stabilized problem gives a vector bundle over the unstable manifold of $S_1^{2k-2}$, for which a stable trivialization of the constant loops (corresponding to a trivialization of $TS^{2k-1}$ over the $*$-section) extends, as described, over the fibers $Q_v$. The lemma follows.
\end{proof}

We next consider \emph{stable trivializations} of these stably trivial index bundles. Let $A\subset \LL_{j} W'$ be a compact subset. Then there exists $N>0$ and a map $\psi\colon A\to\Hom(\R^{N},\tcfig)$, $\gamma\mapsto \psi_{\gamma}$ such that for any $\gamma\in A$,
\[
\bar\pa_{\gamma}\oplus\psi_\gamma\colon \SS_{\gamma}\oplus \R^{N}\to\tcfig
\] 
is surjective. Lemma \ref{Lem:indextrivial} implies that the bundle $\ker(\bar\pa\oplus\R^{N})$ over $A$ with fiber over $\gamma\in A$ given by $\ker(\bar\pa_{\gamma}\oplus \psi_\gamma)$ is trivial over any CW-complex $B \subset A$ of dimension at most $6k-7$. A \emph{stable trivialization} of $\mathbf{I}_{j\beta}$ over $B$ is the stable homotopy class of a trivialization of $\ker(\bar\pa \oplus \R^N)|_{B}$ where we stabilize by adding an arbitrary map $\psi'\colon B\to\Hom(\R^{M},\tcfig)$. To clarify, observe that if $\bar\pa_{\gamma}\oplus\psi_\gamma$ is surjective then for each $e\in\R^{M}$ there is $v\in \SS_{\gamma}\oplus\R^{N}$ such that $\bar\pa_{\gamma}(v)+\psi_\gamma(v)=-\psi'(e)$,  and such $v$ is unique up to addition of a vector  in $\ker(\bar\pa_{\gamma}\oplus\psi_\gamma)$. It follows that a trivialization $Z$ of $\ker(\bar\pa\oplus\R^{N})$ and the standard basis in $\R^{M}$ induces a trivialization of $\ker(\bar\pa\oplus\R^{N}\oplus\R^{M})$.

\begin{Remark}
Consider stabilizations $Z$ of  $\bar\pa\oplus\R^{N}$ and $Z'$ of $\bar\pa\oplus \R^{M}$, which both give everywhere surjective operators. To compare the trivializations on the kernel bundles, we consider the bundle $\ker(\bar\pa\oplus\R^{N}\oplus\R^{M})$, and compare the trivializations $Z\oplus \R^{M}$ and $Z'\oplus\R^{N}$.  \end{Remark}

\begin{Remark}
In the applications below we will only be concerned with compact subsets of the various  mapping spaces involved. For simplicity, we fix throughout a sufficiently large compact subset of the loop space which contains all the spaces (and homotopies) relevant to our problem. 
\end{Remark}

\begin{Remark}\label{Rem:stabletrivhmtpy}
If $B\subset\LL_{j}W'$, if $Z$ is a stable trivialization of $\mathbf{I}_{j\beta}$ over $B$, and if $g\colon C\to B$ is any map, then $Z$ induces a stable trivialization $g^{\ast}Z$ of $g^{\ast}\mathbf{I}_{j\beta}$, $g^{\ast}Z(c)=Z(g(c))$. In particular, if $\Sigma\subset B$, and if $g\colon C\to B$ is a map which is homotopic to a map into $\Sigma$ then the stable trivialization $g^{\ast}Z$ is determined by the restriction $Z|_{\Sigma}$. Indeed, fix a homotopy $g_t\colon C\to B$, $0\le t\le 1$ with $g_0=g$ and $g_{1}(C)\subset \Sigma$; then the homotopy of trivializations $Z(g_t(c))$ connects $g_0^{\ast}Z$ to $g_{1}^{\ast}Z$, and the latter is determined by $Z|_\Sigma$.  
\end{Remark}

\subsection{Pre-gluing and coherent stable trivializations}
We next focus on the loop space components relevant to our main problem: $\LL_{-1} W'$, $\LL_{1} W'$, and $\LL_{0} W'$. Write $\LL_{j}^{\ast}W'=\LL_{j}W'\times \pa D$ for the space of free loops with one marked point and let $\mathbf{I}_{j\beta}^{\ast}$ denote the pull-back of the index bundle $\mathbf{I}_{j\beta}$ under the natural projection map that forgets the marked point. We recall the notion of coherent trivializations from Section \ref{sec:cohertriv}.

Consider a compact CW-complex $N$ with maps $p\colon N\to \LL_{-1}^{\ast} W'$ and $q\colon N\to\LL_{1} W'$ such that $\ev\circ \,p=\ev_1\circ\, q$. Assume furthermore that the map $p$ factors as follows:
\[
\begin{CD}
N @>{p'}>> A\times S^{1} @>{a\times\id}>> \LL_{-1}^{\ast} W' = \LL_{-1} W'\times S^{1}
\end{CD},
\]
where $A$ is a compact CW-complex. Write $\Pre\colon N\to\LL_{0}W'$ for the map $\Pre\circ\, (p\times q)$ and consider the pull-back bundle $\Pre^{\ast}\mathbf{I}_{0\beta}$ over $N$. By Lemma \ref{Lem:lingluetriv} we find that there are two stable trivializations of this bundle: one given by $p^{\ast}Z_{-\beta}^{\ast}\oplus q^{\ast}Z_{\beta}$ and one given by $\Pre^{\ast}(Z_{0\beta}\oplus Z_{TW'})$, where $Z_{TW'}$ is a fixed trivialization of $TW'$.  The triple of trivializations $Z_{-\beta},Z_{\beta}, Z_{0\beta}$ were called $(d',d)$-coherent if the two stable trivializations $p^{\ast}Z_{-\beta}^{\ast}\oplus q^{\ast}Z_{\beta}$ and $\Pre^{\ast}(Z_{0\beta}\oplus Z_{TW'})$ are homotopic for all $N$ and $A$ as above with $\dim(A)\le d'$ and $\dim(N)\le d$. 

In the following Lemma, trivializations of $\mathbf{I}_{j\beta}$ refer to trivializations over $S_{j}^{1}\times\Phi(\widehat{Q})$, see Lemma \ref{Lem:indextrivial}. In the proof, we use notation as in the definition of stable trivializations in Section \ref{sec:cohertriv}.

\begin{Lemma}
For any stable trivialization $Z_{-\beta}^{\ast}$ of $\mathbf{I}^{\ast}_{-\beta}$, there are stable trivializations $Z_{\beta}$ of $\mathbf{I}_{\beta}$ and $Z_{0\beta}$ of $\mathbf{I}_{0\beta}$ such that $(Z_{-\beta}^{\ast},Z_{\beta},Z_{0\beta})$ is $(1,d)$-coherent for any $d \leq 6k-7$. 
\end{Lemma}

\begin{proof}
Remark \ref{Rem:stabletrivhmtpy} implies that the trivialization $p^{\ast}Z_{-\beta}^{\ast}$ is determined by the restriction of the trivialization $Z_{-\beta}^{\ast}$ to the preimage of the $1$-skeleton of $\LL_{-1}W'$. We take the $1$-skeleton as $S^{1}\times \{x\}$, where the $S^{1}$-factor determines the starting point of the geodesics which are constant curves at $x$ in the $S^{2k-1}$-factor. More explicitly, fix a homotopy $a_{t}$ of the map $a$ such that $a_{1}$ maps to the $1$-skeleton. Then $\ev\circ(a_t\times\id)$ gives a homotopy from $\ev_{1}\circ \,q$ to $\ev_{1}\circ \,q_1$, where $q_1$ is the map which conjugates the loop $q$ with the trace of the homotopy $\ev\circ(a_t\times\id)$. Therefore $\ev_1$ is homotopic to a map with second component mapping into the based loop space $\Omega(S^{2k-1},x)$. We can perform a further homotopy of $q_1$, without changing the start-points of loops in the image, so that the $S^{1}$-component of any one of its loops is also geodesic, and so that the loops in the image have the standard product form discussed previously (first following the $S^{1}$-component,  then the $S^{2k-1}$-component). 

It follows that, after a homotopy of $p$ and $q$ and the induced homotopy of $\Pre\circ(p\times q)$, the map  $\Pre\colon N\to\LL_{0}W'$  factors through the fibered product $\Xi$ of (the preimages of) the $1$-skeleton $S^{1}\times\{x\}\times S^{1}\subset \LL_{-1}^{\ast} W'$, and the $d$-skeleton of  $S^1 \times \Omega(S^{2k-1},x)\, \subset \LL_1(W')$.

We now use a skeleton of $\LL_{0}W'$ which differs from the one constructed previously as follows: every loop in the second factor is replaced by a negatively oriented geodesic with a positively oriented geodesic inserted at $-1$ (the antipodal point of the base point). It is then easy to see that the map $N \rightarrow \LL_0W'$ is homotopic to a map in which all $S^{1}$-components have the form of the loops in $\Xi$. 

Conveniently, the fibered product defining $\Xi$ is actually homeomorphic to a cell complex.  More explicitly, the fibered product of the lowest-dimensional cells, meaning those of dimension $\leq 6k-7$, is  homeomorphic to
\begin{equation} \label{Eq:FibreProductOfCells}
S^{1}\times S^{1\ast}\times S_1^{2k-2},
\end{equation}
where the first coordinate on the torus $S^{1}\times S^{1\ast}$ is the starting point of the loop, the second is  the point along the parameterizing $S^{1}$, and the final factor corresponds to the Bott manifold $S^{2k-2}_1$ of Lemma \ref{Lem:skeleton} (recall that the higher cells in the based loop space have index $\geq 6k-6$).  

From the construction of the skeleton, and the fact that we are dealing with the component $\LL_0 W'$ containing the constant loops, the starting point loop maps to a non-contractible loop in $\LL_{0}W'$, whilst the second factor $S^{1\ast}$ maps to a contractible loop. The condition that ``the trivialization of the pull-back agrees with the pullback of the trivialization" is that the loop which is mapped to a trivial loop under $\Pre$ gets the trivial (null-cobordant)  framing.  The framing of this loop is given by
\[
Z_{-\beta}^{\ast}|_{S^{1\ast}}\oplus Z_{\beta}|_{S^{1}},
\]  
so this condition  is fulfilled provided we take $Z_{\beta}$ to have the same framing class as $Z_{-\beta}^{\ast}$. The framing on the non-trivial base point loop in $\LL_{0}W'$ should then also be chosen to be
\[
Z_{-\beta}^{\ast}|_{S^{1\ast}}\oplus Z_{\beta}|_{S^{1}}.
\]
This framing can now be extended trivially over the third factor of \eqref{Eq:FibreProductOfCells}. The lemma follows.
\end{proof}



\bibliographystyle{alpha}

\end{document}